\documentclass[11pt]{amsart}

\newcommand{\GG}{\mathbb{G}}
\newcommand{\NN}{\mathbb{N}}
\newcommand{\PP}{\mathbb{P}}

\newcommand{\ZZ}{\mathbb{Z}}

\newcommand{\cA}{\mathcal{A}}
\newcommand{\cB}{\mathcal{B}}
\newcommand{\cC}{\mathcal{C}}

\newcommand{\cE}{\mathcal{E}}
\newcommand{\cF}{\mathcal{F}}
\newcommand{\cG}{\mathcal{G}}
\newcommand{\cH}{\mathcal{H}}

\newcommand{\cN}{\mathcal{N}}

\newcommand{\cU}{\mathcal{U}}
\newcommand{\cV}{\mathcal{V}}
\newcommand{\cW}{\mathcal{W}}

\newcommand{\fA}{\mathfrak{A}}
\newcommand{\fb}{\mathfrak{b}}
\newcommand{\fE}{\mathfrak{E}}

\newcommand{\fg}{\mathfrak{g}}

\newcommand{\fh}{\mathfrak{h}}
\newcommand{\fK}{\mathfrak{K}}

\newcommand{\fsl}{\mathfrak{sl}}

\newcommand{\fu}{\mathfrak{u}}

\newcommand{\dact}{\boldsymbol{.}}
\newcommand{\lra}{\longrightarrow}
\newcommand{\tha}{\twoheadrightarrow}

\DeclareMathOperator{\add}{add}
\DeclareMathOperator{\Ad}{Ad}
\DeclareMathOperator{\aurk}{rk_{au}}
\DeclareMathOperator{\Aut}{Aut}
\DeclareMathOperator{\Char}{char}
\DeclareMathOperator{\CJT}{CJT}
\DeclareMathOperator{\CR}{CR}
\DeclareMathOperator{\coker}{coker}
\DeclareMathOperator{\cx}{cx}

\DeclareMathOperator{\EIP}{EIP}
\DeclareMathOperator{\End}{End}
\DeclareMathOperator{\Ext}{Ext}
\DeclareMathOperator{\GL}{GL}

\DeclareMathOperator{\gr}{gr}
\DeclareMathOperator{\HH}{H}

\DeclareMathOperator{\Hom}{Hom}
\DeclareMathOperator{\im}{im}
\DeclareMathOperator{\ind}{ind}
\DeclareMathOperator{\id}{id}
\DeclareMathOperator{\Jt}{Jt}

\DeclareMathOperator{\Lie}{Lie}
\DeclareMathOperator{\modd}{mod}
\DeclareMathOperator{\PSL}{PSL}
\DeclareMathOperator{\Pt}{{\Pi}t}
\DeclareMathOperator{\pr}{pr}
\DeclareMathOperator{\pt}{Pt}
\DeclareMathOperator{\ql}{q\ell}
\DeclareMathOperator{\Rad}{Rad}
\DeclareMathOperator{\rk}{rk}

\DeclareMathOperator{\ssrk}{rk_{ss}}
\DeclareMathOperator{\SL}{SL}

\DeclareMathOperator{\Soc}{Soc}
\DeclareMathOperator{\Spec}{Spec}
\DeclareMathOperator{\supp}{supp}
\DeclareMathOperator{\Top}{Top}

\DeclareMathOperator{\Proj}{Proj}

\usepackage{pictex}
\usepackage{amscd}
\usepackage{latexsym}
\usepackage{amssymb}
\usepackage{eucal}
\usepackage{eufrak}
\usepackage{verbatim}

\usepackage{amsfonts}
\usepackage{amsxtra}
\usepackage{color}

\numberwithin{equation}{section}

\newtheorem{Theorem}{Theorem}[section]
\newtheorem{Lemma}[Theorem]{Lemma}
\newtheorem{Corollary}[Theorem]{Corollary}

\theoremstyle{Theorem}
\newtheorem{Thm}{Theorem}[subsection]
\newtheorem{Lem}[Thm]{Lemma}
\newtheorem{Prop}[Thm]{Proposition}
\newtheorem{Cor}[Thm]{Corollary}

\newtheorem*{thm*}{Theorem A}
\newtheorem*{thm**}{Theorem B}
\theoremstyle{remark}
\newtheorem*{Remark}{Remark}
\newtheorem*{Remarks}{Remarks}
\newtheorem*{Definition}{Definition}
\newtheorem*{Example}{Example}
\newtheorem*{Examples}{Examples}

\begin{document}

\title{Categories of Modules given by Varieties of $p$-Nilpotent Operators}

\author[R. Farnsteiner]{Rolf Farnsteiner}

\address[R. Farnsteiner]{Mathematisches Seminar, Christian-Albrechts-Universit\"at zu Kiel, Ludewig-Meyn-Str. 4, 24098 Kiel, Germany}
\email{rolf@math.uni-kiel.de}
\thanks{Supported by the D.F.G. priority program SPP1388 `Darstellungstheorie'.}

\begin{abstract} For a finite group scheme $\cG$ over an algebraically closed field $k$ of characteristic $p>0$ we study $\cG$-modules $M$, which are defined in terms of properties of their
pull-backs $\alpha^\ast_K(M_K)$ along $\pi$-points $\alpha_K$ of $\cG$. We show that the corresponding subcategories strongly depend on the structure of $\cG$. The second part of
the paper discusses recent work by Carlson-Friedlander-Suslin \cite{CFS} concerning the subcategory of equal images modules from the vantage point of Auslander-Reiten theory.  \end{abstract}

\maketitle

\section*{Introduction}
Let $\cG$ be a finite group scheme with coordinate ring $k[\cG]$, defined over an algebraically closed field $k$ of characteristic $\Char(k)=p>0$. In general, the category $\modd \cG$ of
finite-dimensional $\cG$-modules is of wild representation type, rendering the complete understanding of its objects a rather hopeless task. One is thus led to considering subcategories of
$\modd \cG$ which can be better controlled.

In recent work \cite{CFP,CFS,FPe3}, the authors have introduced new classes of $\cG$-modules, whose definition employs nilpotent operators given by $\pi$-points of representation theoretic support spaces, \cite{FPe1,FPe2}. The objects of the corresponding full subcategories of $\modd \cG$ are those having the {\it equal images property}, being of {\it constant Jordan type}, and of {\it constant $j$-rank}. We denote these categories by $\EIP(\cG)$, $\CJT(\cG)$ and $\CR_j(\cG)$ ($j \in \{1,\ldots,p\!-\!1\}$), respectively. There are obvious inclusions $\EIP(\cG) \subseteq \CJT(\cG) \subseteq \CR_j(\cG)$, while $\CJT(\cG)=\bigcap_{j=1}^{p-1}\CR_j(\cG)$.

In \cite{CFS}, the important example of the category $\EIP(\ZZ/(p)\!\times\!\ZZ/(p))$ was investigated. Whereas the indecomposable objects of $\EIP(\ZZ/(p)\!\times\!\ZZ/(p))$ correspond to the pre-injective modules of the Kronecker quiver for $p=2$, the authors provide large families of equal images modules for $p\ge 3$. Building on these findings, one purpose of the present paper is the analysis of these categories for other types of finite group schemes. It turns out that their nature his highly sensitive to the structure of the underlying finite algebraic group. While we verify in Section $4$ that, for $p\ge 3$, the indecomposable objects of $\EIP(\ZZ/(p)\times\!\ZZ/(p))$ cannot be described by finitely many algebraic parameters, the categories of equal images modules over Frobenius kernels of reductive groups are semi-simple. 

\bigskip

\begin{thm*} Let $G_r$ be the $r$-th Frobenius kernel of a connected reductive algebraic group $G$.
\begin{enumerate}
\item Every $M \in \EIP(G_r)$ is a direct sum of one-dimensional modules.
\item We have $\CR_j(G_1)=\CJT(G_1)$ for all $j \in \{1,\ldots,p\!-\!1\}$ and every $M \in \CJT(G_1)$ is isomorphic to its twists by the adjoint action of $G$ on $G_1$.
\end{enumerate} \end{thm*}

\bigskip
\noindent
In the second part, we shall study the aforementioned categories from a different point of view, namely by considering the stable Auslander-Reiten quiver $\Gamma_s(\cG)$ of the self-injective algebra $k\cG:=k[\cG]^\ast$. By definition, the vertices of $\Gamma_s(\cG)$ are the non-projective indecomposable $\cG$-modules. Arrows reflect the presence of so-called irreducible morphisms. The quiver $\Gamma_s(\cG)$ is fitted with an automorphism, the {\it Auslander-Reiten translation}, which is closely related to the second Heller shift. Earlier work implies that components of $\Gamma_s(\cG)$ containing modules of constant Jordan type are usually of isomorphism type $\ZZ[A_\infty]$ and thus have the following shape, where the dotted arrow indicates the AR-translation:

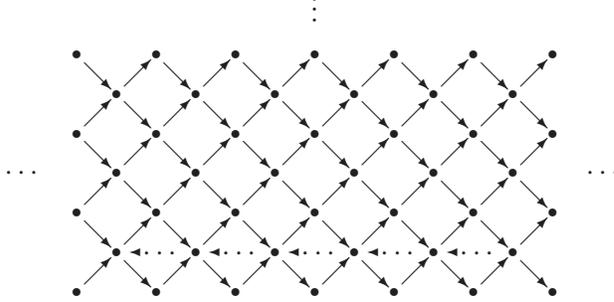
\begin{center}
\begin{picture}(200,130)
\multiput(0,0)(0,30)4{\multiput(0,0)(30,0)7{\circle*{3}}}
\multiput(15,15)(0,30)3{\multiput(0,0)(30,0)6{\circle*{3}}}
\multiput(3,3)(0,30)3{\multiput(0,0)(30,0)6{\vector(1,1){10}}}
\multiput(18,18)(0,30)3{\multiput(0,0)(30,0)6{\vector(1,1){10}}}
\multiput(3,27)(0,30)3{\multiput(0,0)(30,0)6{\vector(1,-1){10}}}
\multiput(18,12)(0,30)3{\multiput(0,0)(30,0)6{\vector(1,-1){10}}}
\put(90,110){\makebox(0,0){$\vdots$}}
\put(200,45){\makebox(0,0){$\cdots$}}
\put(-20,45){\makebox(0,0){$\cdots$}}
\multiput(20,15)(30,0)5{\vector(-1,0){0}}
\multiput(25,12)(30,0)5{$\cdots$}
\end{picture}
\end{center}
The modules belonging to the bottom row are called {\it quasi-simple}. This notion derives from the fact that the $\cG$-modules of such components usually afford a filtration whose factors are quasi-simple.

Given a component $\Theta \subseteq \Gamma_s(\cG)$, it was shown in \cite{CFP,FPe3} that $\Theta$ is contained in $\CJT(\cG)$ or $\CR_j(\cG)$, whenever $\Theta$ meets the corresponding subcategory. By contrast, equal images modules usually intersect AR-components only in small subsets. When specialized to the case $\cG=\ZZ/(p)\!\times\!\ZZ/(p)$, Theorem B below shows that, for $p\ge 3$, an arbitrary component $\Theta$ of $\Gamma_s(\ZZ/(p)\!\times\!\ZZ/(p))$ containing an equal images module of Loewy length $\le p\!-\!2$ only has a finite intersection with $\EIP(\ZZ/(p)\!\times\!\ZZ/(p))$. Moreover, each object of $\Theta\cap\EIP(\ZZ/(p)\!\times\!\ZZ/(p))$ is quasi-simple. The precise statement involves the Jordan type of a module, that is, the isomorphism class of its restriction to suitable subalgebras of type $k[T]/(T^p)$. The latter only has one $i$-dimensional indecomposable module $[i]$ for each $i \in \{1,\ldots,p\}$.

\bigskip

\begin{thm**} Let $\cG$ be a finite group scheme containing an abelian unipotent subgroup scheme of complexity $\ge 2$. Suppose that $\Theta \subseteq \Gamma_s(\cG)$ is a regular component of
tree class $A_\infty$. Then the following statements hold:
\begin{enumerate}
\item If $M \in \Theta\cap\EIP(\cG)$ has Jordan type $\Jt(M)=\bigoplus_{i=1}^{p-1}a_i[i]$, then $M$ is quasi-simple and $\Theta$ contains only finitely many such modules.
\item If there exists $M \in \Theta\cap\EIP(\cG)$ of Jordan type $\Jt(M)=\bigoplus_{i=1}^{p-2}a_i[i]$, then $\Theta\cap\EIP(\cG)$ is a finite set consisting of quasi-simple modules. \end{enumerate}\end{thm**}

\bigskip
\noindent
In order to describe the contents of our paper in more detail, we recall basic notions of support spaces, as expounded by Friedlander and Pevtsova, cf.\ \cite{FPe1,FPe2}. We fix an indeterminate $T$ over $k$ and consider $\fA_p := k[T]/(T^p)$, the $p$-truncated polynomial ring over $k$. If $K$ is an extension field of $k$, we write $\fA_{p,K} := \fA_p\!\otimes_k\!K$ and denote the canonical generator by $t$. Given a finite group scheme $\cG$ over $k$, we let $\cG_K$ be the extended $K$-group scheme associated to $\cG$ and put $K\cG := K\cG_K \cong k\cG\!\otimes_k\!K$, where $k\cG := k[\cG]^\ast$ is the algebra of measures of $\cG$.

Following Friedlander-Pevtsova \cite{FPe2}, an algebra homomorphism $\alpha_K : \fA_{p,K} \lra K\cG$ is called a {\it $\pi$-point}, provided
\begin{enumerate}
\item[(a)] $\alpha_K$ is left flat, and
\item[(b)] there exists an abelian unipotent subgroup scheme $\cU \subseteq \cG_K$ such that $\im \alpha_K \subseteq K\cU$. \end{enumerate}
We denote by $\Pt(\cG)$ the set of $\pi$-points of $\cG$.

Given a $\cG$-module $M$, we let $M_K := M\!\otimes_k\!K$ be the $\cG_K$-module obtained via base change. In view of (a), the pull-back functor
\[ \alpha_K^\ast : \modd \cG \lra \modd \fA_{p,K} \ \ ; \ \ M \mapsto \alpha_K^\ast(M_K)\]
between the categories of finite-dimensional $\cG$-modules and finite-dimensional $\fA_{p,K}$-modules takes projectives to projectives. Two $\pi$-points $\alpha_K$ and $\beta_L$ are {\it equivalent}
($ \alpha_K\sim \beta_L$) if for every $M \in \modd \cG$ the $\fA_{p,K}$-module $\alpha_K^\ast(M_K)$ is projective exactly when the $\fA_{p,L}$-module $\beta^\ast_L(M_L)$ is projective. Thanks to \cite[(7.5)]{FPe2}, the set $\Pi(\cG)$ of equivalence classes of $\Pi$-points has the structure of a $k$-scheme of finite type. It is referred to as the {\it $\pi$-point scheme} of $\cG$. 

If $\cH \subseteq \cG$ is a closed subgroup scheme, then the map
\[ \iota_{\ast,\cH} : \Pi(\cH) \lra \Pi(\cG) \ \ ; \ \ [\alpha_K] \mapsto [\iota_K\circ \alpha_K]\]
is well-defined and continuous, see \cite[(2.7),(3.6)]{FPe2}. Here $\iota_K : K\cH \lra K\cG$ denotes the embedding induced by the canonical inclusion $\cH \hookrightarrow \cG$.

Since the subcategories mentioned above are defined via properties given by $\pi$-points, we study in Section $1$ the subgroup scheme $\cG_\pi \subseteq \cG$ generated by all $\pi$-points. As we show in Theorem \ref{SFS3}, $\cG_\pi \subseteq \cG$ is the unique minimal $\cG(k)$-invariant subgroup scheme of $\cG$ such that $\iota_{\ast,\cG_\pi}(\Pi(\cG_\pi))=\Pi(\cG)$. In preparation for Theorem A, we determine $\cG_\pi$ in case $\cG=G_r$ is the $r$-th Frobenius kernel of a smooth reductive group $G$. Our proof of Theorem A also rests on the stability results presented
in Section $2$, which we formulate in a slightly wider context. Given a self-injective $k$-algebra $\Lambda$ and a connected subgroup $G$ of its automorphism group, we show that the full subcategory $\modd_\Lambda^G \subseteq \modd_\Lambda$ of the category of finitely generated $\Lambda$-modules, whose objects are isomorphic to all their twists by elements of $G$, has almost split sequences.
In Section $3$ we study the category $\CR_j(G_1)$ of modules of constant $j$-rank over the first Frobenius kernel $G_1$ of a smooth reductive algebraic group $G$. This class of modules, which was introduced by Friedlander-Pevtsova in \cite{FPe3}, naturally generalizes the earlier notion of modules of constant Jordan type, cf.\ \cite{CFP}. By definition, the ranks of the linear maps
\[ M_K \lra M_K \ \ ; \ \ m\mapsto \alpha_K(t)^jm \ \ \ \ \ \ \alpha_K \in \Pt(\cG),\]
associated to an object $M \in \CR_j(\cG)$, are constant. In case $\cG=G_1$, we show that such a module already has constant Jordan type, a result which fails for higher Frobenius kernels.

The final two sections are devoted to the study of equal images modules and their Auslander-Reiten theory. After recording a few  straightforward consequences of \cite{CFS}, we show in Section $4$ that equal images modules of Frobenius kernels of reductive groups are sums of one-dimensional modules. Thus, the category $\EIP(\cG)$ reflects differences between the categories $\modd \SL(2)_1$ and $\modd (\ZZ/(2)\!\times\!\ZZ(2))$, which, from the vantage point of abstract Auslander-Reiten theory, are rather similar. The study of almost split sequences involving equal images modules necessitates results on their Heller shifts. The key result in this context implies that Auslander-Reiten translates of equal images modules do usually not belong to $\EIP(\cG)$. In Section $5$ we
exploit this fact and establish a number of results concerning the structure of the intersection $\Theta\cap\EIP(\cG)$ of an AR-component with the category of equal images modules. 

As was shown in \cite[(4.4)]{CFS}, certain equal images modules $W_{n,d}$ are ubiquitous within the category $\EIP (\ZZ/(p)\!\times\! \ZZ/(p))$ in the sense that every object is an image of a suitable $W_{n,d}$. Moreover, the subcategory of $\CJT(\ZZ/(p)\!\times\!\ZZ/(p))$ generated by the $W_{n,d}$ essentially exhausts all Jordan types that can be realized by modules of constant Jordan type. In terms of AR-theory, $\ZZ/(p)\!\times\!\ZZ/(p)$-modules $W_{n,d}$ are rather special: They are quasi-simple and a component containing $W_{n,d}$ intersects $\EIP(\ZZ/(p)\!\times\!\ZZ/(p))$ in either in one or infinitely many vertices, with the latter case occurring exactly when $W_{n,d}$ is a $p$-Koszul module in the sense of \cite{GMMZ04}.

\bigskip

\section{Subgroups with full supports}
The results of this section are motivated by the question to what extent $\cG$-modules are determined by their pull-backs $\alpha_K^\ast(M_K)$ along $\pi$-points. To that end, we consider closed subgroups with the following property:

\bigskip

\begin{Definition} A closed subgroup scheme $\cH \subseteq \cG$ is called {\it $\Pi$-essential} if $\iota_{\ast,\cH}(\Pi(\cH))=\Pi(\cG)$.\end{Definition}

\bigskip

\begin{Example} Let $G$ be a finite group. Using Quillen stratification \cite[(4.12)]{FPe2} one can show that a subgroup $H\subseteq G$ is $\Pi$-essential if and only if for every $p$-elementary
abelian subgroup $E \subseteq G$ there exists $g \in G$ such that the conjugate group $E^g$ of $E$ is contained in $H$, cf.\ Lemma \ref{SFS2} below. \end{Example}

\bigskip
\noindent
Let $\NN$ be the set of positive integers and set $\NN_0 := \NN\cup\{0\}$. Given $r \in \NN_0$, we let $\GG_{a(r)} := \Spec_k(k[T]/(T^{p^r}))$ be the r-th Frobenius kernel of the additive group $\GG_a := \Spec_k(k[T])$. If $\cH \subseteq \GG_{a(r)}$ is a closed subgroup scheme, then $\cH = \GG_{a(s)}$ for some $s\le r$.

A finite group scheme $\cG$ is {\it infinitesimal}, provided its coordinate ring $k[\cG]$ is local. In that case, the least number $r \in \NN_0$ such that $x^{p^r}=0$ for every $x$ belonging to the
augmentation ideal $k[\cG]^\dagger$ of $k[\cG]$ is called the {\it height} of $\cG$.

\bigskip

\begin{Lemma} \label{SFS1} Let $\cH \subseteq \cG$ be a $\Pi$-essential subgroup of an infinitesimal group $\cG$. If $\cG' \subseteq \cG$ is a subgroup, then $\cH \cap \cG'$ is a $\Pi$-essential
subgroup of $\cG'$.\end{Lemma}

\begin{proof} Suppose that $\cG$ has height $r$. Following \cite{SFB1}, we let
\[V_r(\cG) := {\mathcal H O}\mathcal{M}(\GG_{a(r)},\cG)\]
be the affine scheme of infinitesimal one-parameter subgroups of $\cG$. Thanks to \cite[(5.2)]{SFB2} and \cite[(3.6)]{FPe2}, there exists a natural homeomorphism $\Proj(V_r(\cG)) \stackrel{\sim}{\lra} \Pi(\cG)$. Our assumption thus yields a commutative diagram
\[\begin{CD} \Proj(V_r(\cH\cap \cG')) @ >\iota_{\ast,\cH\cap\cG'} >> \Proj(V_r(\cG'))\\
@V\iota_{\ast,\cH\cap\cG'} VV @VV\iota_{\ast,\cG'}V\\
\Proj(V_r(\cH))@ >\iota_{\ast,\cH}>> \Proj(V_r(\cG)),
\end{CD} \]
where $\iota_{\ast,\cH}$ is surjective.

Let $\varphi \in V_r(\cG')$. Then there exists $\psi \in V_r(\cH)$ such that
\[ [ \iota_\cH\circ \psi] = [ \iota_{\cG'}\circ \varphi].\]
It follows from the proof of \cite[(6.1)]{SFB2} that the images of $\varphi$ and $\psi$ coincide. In particular, $\varphi$ factors through $\cH\cap \cG'$. Hence $[\varphi] \in \iota_{\ast,\cH\cap\cG'}(\Proj(V_r(\cH\cap \cG')))$, so that $\cH\cap\cG'$ is a $\Pi$-essential subgroup of $\cG'$. \end{proof}

\bigskip
\noindent
Recall that every finite group scheme $\cG$ over a perfect field is a semi-direct product
\[ \cG = \cG^0\rtimes\cG_{\rm red},\]
where $\cG^0$ is an infinitesimal normal subgroup and $\cG_{\rm red}$ is reduced, cf.\ \cite[(6.8)]{Wa}. If $G$ is a finite group with group algebra $kG$, then $G$ defines a reduced finite group scheme $\tilde{G}:= \Spec_k((kG)^\ast)$ satisfying $\tilde{G}(k)=G$ and $k\tilde{G} \cong kG$. In particular, $\tilde{G}$ is completely determined by its group $G$ of $k$-rational points. To a subgroup $G \subseteq \cG(k)$ of the finite group $\cG(k)$ of $k$-rational points of a finite group scheme $\cG$, there corresponds an embedding $\tilde{G} \hookrightarrow \cG_{\rm red}$. It is customary to write $\tilde{G} =G$.

If $\cG_1$ and $\cG_2$ are finite group schemes such that $\cG_1$ is infinitesimal and $\cG_2$ is reduced, then $\Hom(\cG_1,\cG_2)$ contains only the trivial homomorphism. Thus, if
$\cH \subseteq \cG$ is a subgroup scheme, then $\cH^0 \subseteq \cG^0$ and $\cH_{\rm red} \subseteq \cG_{\rm red}$.

Every element $g \in \cG(k)$ defines an inner automorphism $\kappa_g : \cG \lra \cG$ of the group scheme $\cG$. We refer to a subgroup $\cH \subseteq \cG$ as {\it $\cG(k)$-invariant}
if $\kappa_g$ induces an automorphism of $\cH$ for every $g \in \cG(k)$.

Let $E \subseteq \cG_{\rm red}$ be a $p$-elementary abelian subgroup. Following \cite[\S4]{FPe2}, we set
\[\Pi_0(\cG,E) := \iota_{\ast,(\cG^0)^E\times E}(\Pi((\cG^0)^E\!\times\!E))\!\smallsetminus\!\bigcup_{F\subsetneq E} \iota_{\ast,(\cG^0)^F\times F}(\Pi((\cG^0)^F\!\times\!F)).\]
For a closed subgroup $\cH \subseteq \cG$, we let $\fE(\cH)$ be the set of $p$-elementary abelian subgroups of $\cH(k)$.

\bigskip

\begin{Lemma} \label{SFS2} Suppose that $\cH \subseteq \cG$ is a $\cG(k)$-invariant $\Pi$-essential subgroup. Then the following statements hold:
\begin{enumerate}
\item We have
\[ \Pi(\cG) = \bigcup_{E \in \fE(\cH)} \Pi_0(\cG,E).\]
\item If $E \subseteq \cG_{\rm red}$ is $p$-elementary abelian, then $E \subseteq \cH_{\rm red}$.
\item $\cH_{\rm red}$ is a $\Pi$-essential subgroup of $\cG_{\rm red}$.
\item $\cH^0$ is a $\Pi$-essential subgroup of $\cG^0$.\end{enumerate}\end{Lemma}

\begin{proof} (1) Let $x \in \Pi(\cG)$. By Quillen decomposition \cite[(4.12)]{FPe2}, there exists a $p$-elementary abelian group $E \subseteq \cG_{\rm red}$ such that $x \in \Pi_0(\cG,E)$. By the same token, $\cH \subseteq \cG$ being $\Pi$-essential yields a $p$-elementary abelian subgroup $F \subseteq \cH_{\rm red}$ and $y \in \Pi_0(\cH,F)$ such that $x = \iota_{\ast,\cH}(y)$. Let $\alpha_K \in \Pt((\cH^0)^F\!\times\!F)$ be a representative of $y$. Then the finite subset
\[ \{ F' \subseteq F \ ; \ \exists \ \beta_L \in x \ \text{factoring through} \ (\cG^0)^{F'}\!\times\!F'\}\]
is not empty and hence possesses a minimal element $F_0 \subseteq F$. This readily implies that $x \in \Pi_0(\cG,F_0)$ and \cite[(4.11)]{FPe2} ensures that $E$ and $F_0$ are $\cG(k)$-conjugate. Since $\cH$ is $\cG(k)$-invariant, it follows that $E \subseteq \cH$.

(2) Let $E \subseteq \cG_{\rm red}$ be a $p$-elementary abelian subgroup of rank $1$. If $x \in kE$ is a nilpotent generator, then the $\pi$-point $\alpha_x : \fA_p \lra kE \ ; \ t \mapsto x$ belongs to $\Pi_0(\cG_{\rm red},E)$. We claim that $\iota_{\ast,\cG_{\rm red}}([\alpha_x]) \in \Pi_0(\cG,E)$. Alternatively, there exists a $\pi$-point $\beta_L$ factoring through $(\cG^0)^{E'}\!\times\!E'$
for some subgroup $E' \subsetneq E$ and such that $\iota_{\cG_{\rm red}}\circ \alpha_x \sim \beta_L$. Since $E$ has rank $1$, we obtain $E' = (0)$.

Let $\pi : k\cG \lra k\cG_{\rm red}$ be the canonical projection and denote by $\pi^\ast : \modd \cG_{\rm red} \lra \modd \cG$ the corresponding pull-back functor. Then we have $\pi \circ \iota_{\cG_{\rm red}} = \id_{k\cG_{\rm red}}$. Since $\iota_{\cG_{\rm red}}\circ \alpha_x \sim \beta_L$
and $(\iota_{\cG_{\rm red}}\circ \alpha_x)^\ast(\pi^\ast(k\cG_{\rm red})) = \alpha_x^\ast(k\cG_{\rm red})$ is projective, it follows that $\beta^\ast_L(\pi^\ast_L(L\cG_{\rm red})) =
(\pi_L \circ \beta_L)^\ast(L\cG_{\rm red})$ is projective. As $\pi_L \circ \beta_L(t) = 0$, we have reached a contradiction.

Consequently, $\iota_{\ast,\cG_{\rm red}}([\alpha_x]) \in \Pi_0(\cG,E)$. On the other hand, (1) provides an elementary abelian subgroup $F \subseteq \cH_{\rm red}$ such that
$\iota_{\ast,\cG_{\rm red}}([\alpha_x]) \in \Pi_0(\cG,F)$. It now follows from \cite[(4.11)]{FPe2} that $E$ and $F$ are conjugate. Since $\cH(k) \unlhd \cG(k)$ is a normal subgroup,
we obtain $E \subseteq \cH(k)$, whence $E \subseteq \cH_{\rm red}$. As every $p$-elementary abelian group is a direct product of cyclic groups, our assertion follows.

(3) Let $x \in \Pi(\cG_{\rm red})$. Thanks to \cite[(4.12)]{FPe2}, there exists a $p$-elementary abelian subgroup $E \subseteq \cG_{\rm red}$ such that $x \in \Pi_0(\cG_{\rm red},E)$. By (2),
$E \subseteq \cH_{\rm red}$, so that $x \in \iota_{\ast,\cH_{\rm red}}(\Pi(\cH_{\rm red}))$.

(4) Let $x \in \Pi(\cG^0)$ with $\varphi_K : \fA_{p,K} \lra \cG^0_K$ representing $x$. Quillen decomposition provides a $p$-elementary abelian subgroup $E \subseteq \cH_{\rm red}$ and $y \in \Pi_0(\cH,E)$ such that $\iota_{\ast,\cG^0}(x) = \iota_{\ast,\cH}(y)$. In particular, there exists a $\pi$-point $\alpha_L \in \Pt((\cH^0)^E\!\times\!E)$ that represents $y$. A consecutive application of \cite[(2.6)]{FPe2} and \cite[(4.1)]{FPe1} ensures the existence of $\pi$-points $\beta_Q \in \Pt((\cH^0)^E)$ and $\gamma_Q \in \Pt(E)$ and $c,d \in \{0,1\}$ such that $[\alpha_L] \in \Pi((\cH^0)^E\!\times\!E)$ is represented by $\beta_Q\otimes c + d \otimes \gamma_Q$. Consequently,
\[ \iota_{\ast,\cG^0}([\varphi_K]) = \iota_{\ast,\cH}([\beta_Q\otimes c + d \otimes \gamma_Q]).\]
Let $M := \pi^\ast(k\cG_{\rm red})$. Then the pull-back $\varphi^\ast_K(M_K)$ of $M_K$ along $\varphi_K$ is a trivial $\fA_{p,K}$-module, so that $\iota_{\ast,\cG^0}(x) \in \Pi(\cG)_M$. On the other hand, the $\fA_{p,Q}$-module $(\beta_Q\otimes c + d \otimes \gamma_Q)^\ast(M_Q) = (d\gamma_Q)^\ast(M_Q)$ is projective, unless $d=0$. Thus, $d=0$, $c=1$ and $\iota_{\ast,\cG^0}([\varphi_K])=\iota_{\ast,\cH^0}([\beta_Q]) = \iota_{\ast,\cG^0}(\iota^{\cG^0}_{\ast,\cH^0}([\beta_Q]))$, where $\iota^{\cG^0}_{\ast,\cH^0}$ is induced by the inclusion $\cH^0 \hookrightarrow \cG^0$. Thanks to \cite[(4.12)]{FPe2}, we conclude that $[\varphi_K] \in \Pi(\cG^0)$ is the conjugate $\pi$-point  $\iota^{\cG^0}_{\ast,\cH^0}([\beta_Q])^g$ for some $g \in \cG(k)$. Since $\cH$ is $\cG(k)$-invariant, this yields $[\varphi_K] \in \iota^{\cG^0}_{\ast,\cH^0}(\Pi(\cH^0))$. Consequently, $\cH^0$ is a $\Pi$-essential subgroup of $\cG^0$. \end{proof}

\bigskip
\noindent
Following Friedlander-Pevtsova \cite{FPe1}, we call a subgroup $\cE \subseteq \cG$ {\it quasi-elementary} if $\cE$ is isomorphic to a direct product $\GG_{a(r)}\!\times\! E$ of $\GG_{a(r)}$
and a $p$-elementary abelian group $E$.

Since the intersection of closed subgroups of $\cG$ is again a closed subgroup (cf.\ \cite[(I.1.2)]{Ja2}), there exists the smallest closed subgroup $\cG_\pi \subseteq \cG$ containing all quasi-elementary subgroups of $\cG$. Being a characteristic subgroup of $\cG$, the group $\cG_\pi$ is $\cG(k)$-invariant.

\bigskip

\begin{Theorem} \label{SFS3} The group $\cG_\pi$ is the unique minimal $\cG(k)$-invariant $\Pi$-essential subgroup of $\cG$. \end{Theorem}

\begin{proof} According to \cite[(7.5)]{FPe2}, the support space $\Pi(\cG)$ has the structure of a projective scheme. As a result, the morphism $\iota_{\ast,\cG_\pi} : \Pi(\cG_\pi) \lra
\Pi(\cG)$ is closed.

Let $x \in \Pi(\cG)$ be a closed point. In view of \cite[(4.7)]{FPe2}, there exists a $\pi$-point $\alpha_k : \fA_{p,k} \lra k\cG$ over $k$ representing $x$. Thus, \cite[(4.2)]{FPe1} ensures the existence of a quasi-elementary subgroup $\cE \subseteq \cG$ such that $\im \alpha_k \subseteq k\cE$. Consequently, $x \in \iota_{\ast,\cG_\pi}(\Pi(\cG_\pi))$. Since the set $C$ of closed points of $\Pi(\cG)$ lies dense in $\Pi(\cG)$ (see \cite[(3.35)]{GW}) and is contained in the closed subset $\im \iota_{\ast,\cG_\pi}$, we obtain $\Pi(\cG)= \iota_{\ast,\cG_\pi}(\Pi(\cG_\pi))$. Hence $\cG_\pi$ is a $\Pi$-essential subgroup of $\cG$.

Let $\cH \subseteq \cG$ be a $\cG(k)$-invariant $\Pi$-essential subgroup of $\cG$. According to Lemma \ref{SFS2}(4), the connected component $\cH^0$ is a $\Pi$-essential subgroup of $\cG^0$. If $\cE \subseteq \cG$ is a quasi-elementary subgroup, then $\cE^0 \cong \GG_{a(r)}$ is a subgroup of $\cG^0$. By virtue of Lemma \ref{SFS1}, the group $\cH^0\cap\cE^0$ is $\Pi$-essential in $\cE^0$. Since $\cH^0\cap \cE^0 \cong \GG_{a(s)}$ for some $s \le r$, it follows from \cite[(5.8)]{FPe2} that
\[ s = \dim \Pi(\cH^0\cap\cE^0)\!+\!1 \ge \dim \Pi(\cE^0)\!+\!1 = r.\]
Consequently, $\cH^0\cap\cE^0=\cE^0$, so that $\cE^0 \subseteq \cH^0$. In view of Lemma \ref{SFS2}(2), $\cE_{\rm red}$ is a subgroup of $\cH_{\rm red}$. As a result, $\cE$ is contained in $\cH$. Consequently, $\cG_\pi \subseteq \cH$, showing that $\cG_\pi$ is the unique minimal $\Pi$-essential $\cG(k)$-invariant subgroup of $\cG$. \end{proof}

\bigskip
\noindent
A finite group scheme $\cG$ over $k$ is referred to as {\it linearly reductive} if $k\cG$ is semi-simple. By Nagata's Theorem \cite[(IV,\S3,3.6)]{DG}, this is equivalent to $\cG^0$ being diagonalizable and $p$ not dividing the order of $\cG(k)$. In particular, $\cG$ contains no non-trivial quasi-elementary subgroups.

\bigskip

\begin{Lemma} \label{SFS4} Let $\cG$ be a finite group scheme. If $\cN \unlhd \cG$ is a normal subgroup such that $\cG/\cN$ is linearly reductive, then $\cG_\pi = \cN_\pi$. \end{Lemma}

\begin{proof} Since every quasi-elementary subgroup $\cE \subseteq \cN$ is contained in $\cG_\pi\cap\cN$, it follows that $\cN_\pi \subseteq \cG_\pi \cap \cN$. Let $\cE \subseteq \cG$ be quasi-elementary. As $\cG/\cN$ is linearly reductive, Nagata's Theorem yields $\cE \subseteq \cN$, whence $\cE \subseteq \cN_\pi$. As a result, $\cG_\pi \subseteq \cN_\pi$.\end{proof}

\bigskip
\noindent
We turn to the case where the underlying group scheme is a Frobenius kernel of a reduced group scheme.

\bigskip

\begin{Lemma} \label{SFS5} Let $G$ be a reduced algebraic group scheme. If $\cN \unlhd G$ is a normal finite subgroup scheme, then $\cN_\pi \subseteq \cN$ is a normal subgroup of $G$. \end{Lemma}

\begin{proof} Given $g \in G(k)$, we consider the automorphism $\kappa_g : G \lra G$ effected by $g$. Since $\cN \unlhd G$ is a normal subgroup of $G$, the restriction $\kappa_g|_{\cN}$ is an
automorphism of $\cN$, whence $\kappa_g(\cN_\pi) = \cN_\pi$.

As a result, $G(k) \subseteq N_G(\cN_\pi)(k)$, where $N_G(\cN_\pi) \subseteq G$ denotes the normalizer of $\cN_\pi$ in $G$, see \cite[(I.2.6)]{Ja2}. Since $\cN_\pi$ is a locally free
closed subfunctor of $G$, it follows from \cite[(I.2.6(8))]{Ja2} that $N_G(\cN_\pi)$ is closed. As $G(k)$ is dense in $G$ (cf.\ \cite[(I.6.16)]{Ja2}), we obtain $N_G(\cN_\pi) = G$,
so that $\cN_\pi$ is a normal subgroup of $G$. \end{proof}

\bigskip
\noindent
In the sequel, we let $X(\cG):=\Hom(\cG,\GG_m)$ be the {\it character group} of the group scheme $\cG$. Its elements are homomorphisms $\lambda : \cG \lra \GG_m$ taking values in the multiplicative
group $\GG_m=\GL(1)$.

\bigskip

\begin{Theorem} \label{SFS6} Let $G$ be a reduced reductive algebraic group scheme.
\begin{enumerate}
\item $(G_r)_\pi$ is a normal subgroup of $G_r$ such that $(G_r)_\pi T_r = G_r$ for every maximal torus $T \subseteq G$.
\item We have $(G_r,G_r) \subseteq (G_r)_\pi \subseteq (G,G)_r$. Moreover, the group $G_r/(G_r)_\pi$ is diagonalizable. \end{enumerate} \end{Theorem}

\begin{proof} (1) Since $G_r \unlhd G$ is a normal subgroup, Lemma \ref{SFS5} implies that $(G_r)_\pi$ is a normal subgroup of $G$. Consequently, $(G_r)_\pi$ is also normal in $G_r$.
Let $T \subseteq G$ be a maximal torus with root system $\Phi \subseteq X(T)$. Given a root $\alpha \in \Phi$, we let $U_\alpha \subseteq G$ be the corresponding root subgroup,
see \cite[(II.1.2)]{Ja2}. Then there exists an isomorphism $\varphi : \GG_{a(r)} \lra (U_\alpha)_r$, so that $(U_\alpha)_r \subseteq G_r$ is quasi-elementary and $(U_\alpha)_r \subseteq (G_r)_\pi$.

According to \cite[(II.3.2)]{Ja2}, the multiplication map induces an isomorphism
\[ m : \prod_{\alpha \in \Phi^+}(U_\alpha)_r \times T_r \times \prod_{\alpha \in \Phi^-}(U_\alpha)_r \stackrel{\sim}{\lra} G_r\]
of schemes. Here $\Phi^+$ and $\Phi^-$ denote the sets of positive and negative roots, respectively. There results a commutative diagram
\[\begin{CD} \prod_{\alpha \in \Phi^+}(U_\alpha)_r \times T_r \times \prod_{\alpha \in \Phi^-}(U_\alpha)_r @ >m >> G_r\\
@VVV @ AA\mu A\\
(G_r)_\pi\times T_r \times (G_r)_\pi@ >\omega>> (G_r)_\pi\rtimes T_r,
\end{CD} \]
where $\omega(a,t,b) = (atbt^{-1},t)$ and $\mu$ is again the multiplication. Since $m$ is bijective, the comorphism $\mu^\ast : k[G_r] \lra k[(G_r)_\pi\!\rtimes\!T_r]$ is injective. Thus, $\mu$ is a
quotient map, so that $(G_r)_\pi T_r = G_r$, cf.\ \cite[(15.1)]{Wa}, \cite[(I.6.2)]{Ja2}.

(2) In view of (1), we have an isomorphism $G_r/(G_r)_\pi \cong T_r/((G_r)_\pi\cap T_r)$, showing that the former group is diagonalizable, cf.\ \cite[(IV,\S1,1.7)]{DG}. As $T_r$ is abelian, this also
implies $(G_r,G_r) \subseteq (G_r)_\pi$. By general theory \cite[\S 27]{Hu}, the group $G/(G,G)$ is diagonalizable. Since the sequence
\[ e_k \lra (G,G)_r \lra G_r \lra (G/(G,G))_r\]
is exact, it follows that $G_r/(G,G)_r \hookrightarrow (G/(G,G))_r$ is a subgroup of a diagonalizable group scheme. Hence $G_r/(G,G)_r$ is diagonalizable \cite[(IV,\S1,1.7)]{DG}, and Lemma \ref{SFS4} yields $(G_r)_\pi \subseteq (G,G)_r$.  \end{proof}

\bigskip

\section{The category of $G$-stable Modules}
Let $k$ be an algebraically closed field of characteristic $p>0$. Throughout this section, we consider a finite-dimensional $k$-algebra $\Lambda$ and an algebraic group $G$ that acts on $\Lambda$ via automorphisms. Thus,
\[ G \times \Lambda \lra \Lambda \ \ ; \ \ (g,x) \mapsto g.x\]
is a morphism of affine varieties and $x\mapsto g.x$ is an automorphism of $\Lambda$ for every $g \in G$.

We let $\modd_\Lambda$ be the category of finitely generated $\Lambda$-modules. The group $G$ acts on $\modd_\Lambda$ via auto-equivalences. For $g \in G$ and $M \in \modd_\Lambda$, we denote
by $M^{(g)}$ the $\Lambda$-module with underlying $k$-space $M$ and action
\[ x\dact m := (g^{-1}.x)m \ \ \ \ \ \ \forall \ x \in \Lambda, \, m \in M.\]
A $\Lambda$-module $M$ is referred to as {\it $G$-stable} if $M^{(g)} \cong M$ for every $g \in G$. We denote by $\modd^G_\Lambda$ the full subcategory of $\modd_\Lambda$, whose objects are the
$G$-stable modules.

\bigskip

\begin{Example} Let $k(\ZZ/(p)^r)$ be the group algebra of the $p$-elementary abelian group of rank $r$, so that $k(\ZZ/(p)^r) =k[X_1,\ldots,X_r]/(X_1^p,\ldots,X_r^p)$. Accordingly, the natural action of $\GL(r)$ on $k^r$ induces an operation
\[ \GL(r)\!\times\! k(\ZZ/(p)^r) \lra k(\ZZ/(p)^r)\]
of $\GL(r)$ on $k(\ZZ/(p)^r)$ via automorphisms. Since the group $\GL(r)$ acts transitively on the dense subset $P(\ZZ/(p)^r) \subseteq \Pi(\ZZ/(p)^r)$ of closed $\pi$-points of $\ZZ/(p)^r$, every $\GL(r)$-stable $k(\ZZ/(p)^r)$-module $M$ has constant Jordan type.

If $r=2$, then \cite[(4.7)]{CFS} implies that the $k(\ZZ/(p)^2)$-modules $W_{n,d}$ of \cite[\S2]{CFS} are $\GL(2)$-stable. \end{Example}

\bigskip

\begin{Lemma} \label{SM1} Let $M$ be a $\Lambda$-module. Then
\[ G_M := \{g \in G \ ; \ M^{(g)} \cong M\}\]
is a closed subgroup of $G$. \end{Lemma}

\begin{proof} Let $d:=\dim_kM$. We consider the affine variety $\modd_\Lambda^d$ of $d$-dimensional $\Lambda$-modules with underlying $k$-space $M$. Thus, an element of $\modd_\Lambda^d$ is an algebra homomorphism $\varrho : \Lambda \lra \End_k(M)$. The algebraic groups $G$ and $\GL(M)$ act on $\modd_\Lambda^d$ via
\[(g\dact\varrho)(x) := \varrho(g^{-1}\dact x) \ \ \ \ \ \ \forall \ g \in G, \, \varrho \in \modd^d_\Lambda, \, x \in \Lambda\]
and
\[(\psi \ast \varrho)(x) = \psi \circ \varrho(x) \circ \psi^{-1}  \ \ \ \ \ \forall \ \psi \in \GL(M), \, \varrho \in \modd^d_\Lambda, \, x \in \Lambda.\]
Note that the $\GL(M)$-orbits correspond to the isomorphism classes of $d$-dimensional $\Lambda$-modules. Let $\varrho_M \in \modd_\Lambda^d$ be the representation afforded by the $\Lambda$-module $M$. It follows that
\[ G_M = \{ g \in G \ ; \ g\dact\varrho_M \in \GL(M)\!\ast\!\varrho_M\}.\]
Since $g \mapsto g\dact \varrho_M$ is a morphism of affine varieties and the orbit $\GL(M)\!\ast\!\varrho_M$ is locally closed, the subgroup $G_M \subseteq G$ is locally closed. Thus, $G_M$ is an open subgroup of the closed subgroup $\overline{G}_M$ of $G$. Consequently, $G_M$ is also closed in $\overline{G}_M$ and hence closed in $G$. \end{proof}

\bigskip
\noindent
Given a non-projective indecomposable $\Lambda$-module $M$, we denote by
\[ \fE(M): \ \ \ \ \ (0) \lra \tau_\Lambda(M) \lra E_M \lra M \lra (0)\]
the {\it almost split sequence} terminating in $M$. Similarly,
\[ (0) \lra N \lra F_N \lra \tau^{-1}_\Lambda(N) \lra (0)\]
denotes the almost split sequence originating in the non-injective indecomposable $\Lambda$-module $N$. We have $E_{\tau^{-1}_\Lambda(M)}\cong F_{\tau_\Lambda(M)}$, see \cite{ARS} for further details.

\bigskip

\begin{Lemma} \label{SM2} Suppose that $G$ is connected. Then the following statements hold:
\begin{enumerate}
\item Every simple $\Lambda$-module belongs to $\modd^G_{\Lambda}$.
\item The category $\modd^G_{\Lambda}$ is closed under direct sums and direct summands.
\item If $\Lambda$ is a Hopf algebra and $G$ acts on $\Lambda$ via automorphisms of Hopf algebras, then $\modd^G_\Lambda$ is closed under tensor products.
\item If $M \in \modd^G_\Lambda$ is a non-projective indecomposable $\Lambda$-module, then $E_M, \tau_\Lambda(M)\in \modd^G_\Lambda$.
\item If $N \in \modd^G_\Lambda$ is a non-injective indecomposable $\Lambda$-module, then $E_{\tau_\Lambda^{-1}(N)}, \tau_\Lambda^{-1}(N) \in \modd_\Lambda^G$. \end{enumerate} \end{Lemma}

\begin{proof} (1) This is well-known.

(2) It is clear that $\modd^G_\Lambda$ is closed under direct sums. Let $M \in \modd^G_\Lambda$ and suppose that $X$ is an indecomposable direct summand of $M$.
The theorem of Krull-Remak-Schmidt implies that $G_X$ is a subgroup of $G$ of finite index. Since $G$ is connected, Lemma \ref{SM1} yields $G_X=G$. As a result, $X$ is $G$-stable. Since every direct summand of $M$ is a direct sum of indecomposable direct summands of $M$, our assertion follows.

(3) Let $M,N \in \modd^G_\Lambda$. Given $g \in G$, there exist isomorphisms $\varphi : M \lra M^{(g)}$ and $\psi : N \lra N^{(g)}$. Since $G$ acts on $\Lambda$ via automorphisms of Hopf
algebras, we have
\begin{eqnarray*}
(\varphi\!\otimes\!\psi)(u.(m\!\otimes\!n)) & = & \sum_{(u)}u_{(1)}\dact\varphi(m)\otimes u_{(2)}\dact \psi(n) = \sum_{(u)}(g^{-1}.u_{(1)})\varphi(m)\otimes (g^{-1}. u_{(2)})\psi(n)\\
& = & (g^{-1}. u)(\varphi\otimes\psi)(m\otimes n)
\end{eqnarray*}
for $m\in M, \, n \in N,\, u \in \Lambda$. Hence $\varphi\otimes \psi: M\!\otimes_k\!N \lra (M\!\otimes_k\!N)^{(g)}$ is an isomorphism, showing that $M\!\otimes_k\!N$ is $G$-stable.

(4) Let $M \in \modd^G_\Lambda$ be indecomposable and suppose that
\[ (0) \lra N \lra E \lra M \lra (0)\]
is the almost split sequence terminating in $M$. Given $g \in G$, the functor $X \mapsto X^{(g)}$ is an auto-equivalence of $\modd_\Lambda$. Consequently,
\[ (0) \lra N^{(g)} \lra E^{(g)} \lra M \lra (0)\]
is another almost split sequence ending in $M$. According to \cite[(V.1.16)]{ARS}, this implies $E^{(g)} \cong E$ and $N^{(g)} \cong N$. Thus, $N,E \in \modd^G_\Lambda$, as desired.

(5) This follows analogously. \end{proof}

\bigskip
\noindent
Suppose that $\Lambda$ is self-injective. We denote by $\Gamma_s(\Lambda)$ the {\it stable Auslander-Reiten quiver} of $\Lambda$. By definition, $\Gamma_s(\Lambda)$ is a directed graph, whose
vertices are the isomorphism classes of the non-projective indecomposable $\Lambda$-modules. There is an arrow $X \rightarrow M$ if and only if $X$ is a direct summand of $E_M$. The Auslander-Reiten translation $\tau_\Lambda$ is an automorphism of $\Gamma_s(\Lambda)$. For further details, we refer to \cite[VII]{ARS}.

Let $\cC \subseteq \modd_\Lambda$ be a full subcategory that is closed under isomorphisms. Abusing notation, we shall write $\Theta\cap \cC$ for the set of those vertices of the AR-component $\Theta \subseteq \Gamma_s(\Lambda)$, whose representatives belong to $\cC$. The notation $\Theta \subseteq \cC$ indicates that $\Theta\cap\cC=\Theta$. Thus, we often do not distinguish between the directed graph
and its underlying set of vertices. Given an algebraic group $G$, we denote its connected component by $G^0$.

\bigskip

\begin{Lemma} \label{SM3} Suppose that $\Lambda$ is self-injective and let $\Theta \subseteq \Gamma_s(\Lambda)$ be a component. Then the following statements hold:
\begin{enumerate}
\item If $G$ is connected and $\Theta \cap \modd^G_\Lambda \ne \emptyset$, then $\Theta \subseteq \modd^G_\Lambda$.
\item We have $G_M^0 = G_N^0$ for all $M,N \in \Theta$.
\end{enumerate} \end{Lemma}

\begin{proof} (1) Let $M \in \Theta$ be an indecomposable $G$-stable $\Lambda$-module. By (4),(5) and (2) of Lemma \ref{SM2}, all successors and all predecessors of $M$ belong to $\Theta\cap \modd_\Lambda^G$. By the same token, the set $\Theta \cap \modd^G_\Lambda$ is $\tau_\Lambda$-invariant as well as $\tau_\Lambda^{-1}$-invariant. As $\Theta$ is connected, we obtain $\Theta \cap \modd^G_\Lambda = \Theta$, as desired.

(2) Fix a $\Lambda$-module $M \in \Theta$ and consider the set
\[ \cA_M := \{X \in \Theta \ ; \ G_X^0 = G_M^0\}.\]
If $X \in \cA_M$ and $Y \in \Theta$ is a predecessor of $X$, then a twofold application of (1) yields $G_X^0 = G_Y^0$. By the same token, $G_{\tau^{\pm 1}_\Lambda(X)}^0 =G_X^0$. As a result, the set $\cA_M$ is $\tau^{\pm 1}_\Lambda$-invariant as well as closed under taking successors and predecessors. Consequently, $\cA_M=\Theta$, as desired. \end{proof}

\bigskip

\section{Modules of constant rank for reductive Lie algebras}
Throughout this section, the algebraically closed field $k$ is assumed to have $\Char(k)=p\ge 3$. Let $\cG$ be a finite group scheme over $k$. Recall that every $\pi$-point $\alpha_K \in \Pt(\cG)$ gives rise to a pull-back functor
\[ \alpha_K^\ast : \modd \cG_K \lra \modd \fA_{p,K}\]
that sends projectives to projectives. Given $M \in \modd \cG$, we have
\[ \alpha_K^\ast(M_K) \cong \bigoplus_{i=1}^pa_i [i],\]
where $[i]$ denotes the unique indecomposable $\fA_{p,K}$-module of dimension $i$. Hence there exists an $\fA_p$-module $N$ with
\[ N_K \cong \alpha^\ast_K(M_K),\]
whose isomorphism class is the {\it Jordan type} $\Jt(M,\alpha_K)$ of $M$ relative to $\alpha_K$. If $\Jt(M,\alpha_K) = \bigoplus_{i=1}^pa_i[i]$, then $\bigoplus_{i=1}^{p-1}a_i[i]$ is referred to as the {\it stable Jordan type} of $M$ relative to $\alpha_K$.  

We let
\[ \Jt(M) := \{ \Jt(M,\alpha_K) \ ; \ \alpha_K \in \Pt(\cG)\}\]
be the finite set of Jordan types of $M$. Following \cite{CFP}, the $\cG$-module $M$ is said to have {\it constant Jordan type}, provided $|\Jt(M)|=1$. In that case, $\Jt(M)$ denotes the unique Jordan type of $M$. 

Given an extension field $K$ of $k$ and $u \in K\cG$, we denote by $u_{M_K}$ the left multiplication of $M_K$ effected by $u$. Let $j \in \{0,\ldots,p\!-\!1\}$. Following Friedlander-Pevtsova \cite{FPe3}, we say that a $\cG$-module $M$ has {\it constant $j$-rank} $\rk^j(M)=d$, provided
\[\rk(\alpha_K(t)^j_{M_K}) = d \ \ \text{for all} \ \alpha_K \in \Pt(\cG).\]
Thus, a $\cG$-module $M$ has constant Jordan type if and only if $M$ has constant $j$-rank for every $j \in \{1,\ldots,p\!-\!1\}$.

In the sequel we shall often identify infinitesimal groups of height $1$ with restricted Lie algebras. By definition, a restricted Lie algebra $(\fg,[p])$ is a pair consisting of a Lie algebra
$\fg$ and a $p$-th power operator $[p] : \fg \lra \fg$, sending $x \in \fg$ to $x^{[p]}$, that satisfies the formal properties of an associative $p$-th power. We refer the reader to \cite[Chap.II]{SF}
for the details. Let $U(\fg)$ be the universal enveloping algebra of $\fg$. Then $I := (\{x^p\!-\!x^{[p]} \ ; \ x \in \fg\})$ is a Hopf ideal of $U(\fg)$ and
\[ U_0(\fg) := U(\fg)/I\]
is a finite-dimensional cocommutative Hopf algebra, whose Lie algebra of primitive elements coincides with $\fg$. The algebra $U_0(\fg)$ is called the {\it restricted enveloping algebra} of $\fg$.

Let $\cG$ be an algebraic group scheme. The first Frobenius kernel $\cG_1$ of $\cG$ corresponds to the Lie algebra $\fg := \Lie(\cG)$ in the sense that there is an isomorphism $k\cG_1 \cong U_0(\fg)$
of Hopf algebras, cf.\ \cite[(II,\S7,${\rm n}^{\rm o}$4)]{DG}. Thus, we may identify the set $\Pi(\cG_1)$ of $\pi$-points of $\cG_1$ with the set of $\pi$-points $\Pi(\fg)$ of the Lie algebra $\fg$.
The latter space is closely related to the {\it nullcone}
\[ \cV_\fg := \{x \in \fg \ ; \ x^{[p]}=0\}\]
of $\fg$. Since $\cV_\fg \subseteq \fg$ is conical, there is the associated projective variety $\PP(\cV_\fg) \subseteq \PP(\fg)$.

Let $M$ be a $\cG$-module. According to \cite{FPe2}, the {\it $\Pi$-support}
\[ \Pi(\cG)_M := \{[\alpha_K] \in \Pi(\cG) \ ; \ \alpha_K^\ast(M_K) \ \text{is not projective}\}\]
of $M$ is a closed subset of $\Pi(\cG)$. If $\cG$ has height $1$, then there exists a homeomorphism between the closed points of $\Pi(\cG)$ and $\PP(\cV_\fg)$ sending the closed
points of $\Pi(\cG)_M$ onto $\PP(\cV_\fg(M))$, where
\[ \cV_\fg(M) := \{ x \in \cV_\fg \ ; \ M|_{k[x]} \ \text{is not projective}\} \cup \{0\}.\]
is the {\it rank variety} of the $U_0(\fg)$-module $M$.

Let $G$ be a smooth (reduced) algebraic group. Then $G$ acts on $\fg := \Lie(G)$ and $U_0(\fg)$ via the adjoint representation
\[ \Ad : G \lra \Aut(U_0(\fg)).\]
In fact, $\Ad(g)$ is an automorphism of Hopf algebras for every $g \in G$. We thus have the notion of a $G$-stable $U_0(\fg)$-module. Note that every rational $G$-module $M$ gives rise to a
$G$-stable $U_0(\fg)$-module via restriction to the first Frobenius kernel $G_1$.

\bigskip

\begin{Lemma} \label{MCR1} Let $M$ be a $U_0(\fsl(2))$-module of constant $j$-rank for some $j \in \{1,\ldots,p\!-\!1\}$. Then $M$ is an $\SL(2)$-stable module of constant Jordan type.\end{Lemma}

\begin{proof} We write
\[ M = M_0 \oplus M_1 \oplus M_2,\]
with $M_i$ being a (possibly empty) direct sum of indecomposables, whose rank varieties have dimension $i$. According to \cite[Thm.]{Pr} and \cite[(8.1.1)]{Fa5}, the $U_0(\fsl(2))$-modules
$M_0$ and $M_2$ are restrictions of rational $\SL(2)$-modules and have constant Jordan type.

Since $\rk(\alpha_K(t)_{(M_1)_K}^j) = \rk(\alpha_K(t)_{M_K}^j) \!-\! \rk(\alpha_K(t)_{(M_2)_K}^j) \!-\!
\rk(\alpha_K(t)_{(M_0)_K}^j)$ for every $\pi$-point $\alpha_K \in \Pt(\fsl(2))$, it follows that $M_1$ has constant $j$-rank. Consequently, we have $\Pi(\fsl(2))_{M_1} \in \{\Pi(\fsl(2)),\emptyset\}$. As $\dim \Pi(\fsl(2))_{M_1} \le 0$, we conclude that $M_1$ is projective, whence $M_1=(0)$. As a result, $M = M_0\oplus M_2$ has the asserted properties.\end{proof}

\bigskip
\noindent
Given a reductive group $G$, the {\it semi-simple rank} $\ssrk(G)$ of $G$ is the dimension of a maximal torus of the derived group $(G,G)$.

\bigskip

\begin{Theorem} \label{MCR2} Let $G$ be a connected reductive algebraic group with Lie algebra $\fg = \Lie(G)$. If $M$ is a
$U_0(\fg)$-module of constant $j$-rank for some $j \in \{1,\ldots,p\!-\!1\}$, then the following statements hold:
\begin{enumerate}
\item $M$ is $G$-stable.
\item $M$ has constant Jordan type. \end{enumerate}\end{Theorem}

\begin{proof} (1) It follows from \cite[(6.15)]{Sp} that the subgroup $G' := (G,G)$ is semi-simple. By Lemma \ref{SM1}, the stabilizer $G'_M$ is a closed subgroup of $G'$.

Let $\alpha \in \Phi$ be a root of $G'$ relative to some maximal torus $T \subseteq G'$, and consider the subgroup $G'_{(\alpha)} :=
Z_{G'}(\ker \alpha)^0$, the connected component of the centralizer of $\ker \alpha$ in $G'$. Thanks to \cite[(9.3.5)]{Sp}, $G'_{(\alpha)}$ is a closed, reductive subgroup of
$G'$ of semi-simple rank $1$. Consequently, the derived group $H_{(\alpha)} :=(G_{(\alpha)}, G_{(\alpha)})$
is semi-simple of rank $1$, and \cite[(8.2.4)]{Sp} yields $H_{(\alpha)} \cong \SL(2), \, \PSL(2)$. Owing
to \cite[(9.3.6)]{Sp}, the root subgroup $U_\alpha$ is contained in $H_{(\alpha)}$. Note that
$\fh_{(\alpha)} := \Lie(H_{(\alpha)}) \cong \fsl(2)$ is a $p$-subalgebra of $\fg$. Lemma \ref{MCR1} now
implies that the module $M|_{\fh_{(\alpha)}}$ is $\SL(2)$-stable. Hence $M$ is also $H_{(\alpha)}$-stable,
so that the root subgroup $U_\alpha$ is contained in $G'_M$. Thanks to \cite[(27.5)]{Hu}, we thus have
$G' \subseteq G'_M \subseteq G_M$. General theory \cite[(6.15)]{Sp} yields $G=Z(G)^0 G'$. Since the
connected center $Z(G)^0$ acts trivially on $U_0(\fg)$, we have $Z(G)^0 \subseteq G_M$. Hence $G_M=G$,
so that $M$ is $G$-stable.

(2) According to (1), the $U_0(\fg)$-module $M$ is $G$-stable. We consider the conical variety
\[ \cV_\fg^<(M) := \bigcup_{i=1}^{p-1}\cV_\fg^i(M),\]
where $\cV_\fg^i(M) := \{x \in \cV_\fg \ ; \ \rk(x^i_M)<\rk^i(M)\}$ is the closed, conical set of
operators of non-maximal $i$-rank. In view of \cite[(4.7)]{FPe2} and \cite[(3.8)]{FPe1}, the projective
variety $\PP(\cV^<_\fg(M))$ is homeomorphic to the space of closed points of the scheme $\Gamma(G_1)_M$
of $\pi$-points of non-maximal Jordan type for $M$. According to \cite[(3.6)]{CFP}, the module $M$ has
constant Jordan type if and only if $\Gamma(G_1)_M = \emptyset$. Since the set of closed points of a
scheme of finite type is dense \cite[(3.35)]{GW}, it follows that $M$ has constant Jordan type if and only
if $\PP(\cV^<_\fg(M)) = \emptyset$.

The adjoint representation of $G$ on $\fg$ induces an action on $\cV_\fg$ as well as on the associated projective variety $\PP(\cV_\fg)$. Since $M$ is $G$-stable,
the variety $\PP(\cV^<_\fg(M))$ is a $G$-invariant subset of $\PP(\cV_\fg)$.
Suppose that $\PP(\cV^<_\fg(M)) \ne \emptyset$ and let $T \subseteq G'$ be a maximal torus. Borel's fixed point theorem \cite[(7.2.5)]{Sp} ensures
the existence of a root vector $x_\alpha \in \cV_\fg$ such that $[x_\alpha]\in \PP(\cV^<_\fg(M))$.
Consequently, the $U_0(\fh_{(\alpha)})$-module $M|_{U_0(\fh_{(\alpha)})}$ does not have constant Jordan
type. It now follows from Lemma \ref{MCR1} that $M|_{U_0(\fh_{(\alpha)})}$ does not have constant
$j$-rank either, a contradiction. \end{proof}

\bigskip

\begin{Remark} The foregoing result does not hold for higher Frobenius kernels. Consider the second Frobenius kernel $\SL(2)_2$ of the reductive group scheme
$\SL(2)$. Let $L(\lambda)$ be the simple $\SL(2)_2$-module with highest weight $\lambda = \lambda_1\!+\!p\lambda_2$, where $0\le \lambda_i \le p\!-\!1$.
In view of \cite[(4.12)]{FPe3}, $L(\lambda)$ has constant $j$-rank whenever $j>\lambda_1\!+\!\lambda_2$. By the same token, the trivial module and the Steinberg module are the
only simple $\SL(2)_2$-modules of constant Jordan type. \end{Remark}

\bigskip
\noindent
Let $G$ be an algebraic group with Lie algebra $\fg := \Lie(G)$. A $p$-subalgebra $\fb \subseteq \fg$ is called a {\it Borel subalgebra}, provided $\fb=\Lie(B)$ for some Borel subgroup $B
\subseteq G$. We denote by ${\rm Bor}(\fg)$ and ${\rm Car}(\fg)$ the sets of Borel subalgebras and Cartan subalgebras of $\fg$, respectively. We record the following observation:

\bigskip

\begin{Lemma} \label{MCR3} Let $G$ be an algebraic group. Then we have
\[ \fg = G.\fb\]
for every $\fb \in {\rm Bor}(\fg)$. \end{Lemma}

\begin{proof} Thanks to \cite[(14.4)]{Hu1}, the group $G$ acts transitively on ${\rm Bor}(\fg)$. Let $\fb \in {\rm Bor}(\fg)$, $B \subseteq G$ be a Borel subgroup of $G$ such that $\Lie(B)=\fb$. Since $\fb$ is $B$-invariant, it follows from \cite[(0.15)]{Hu3} that $\bigcup_{\fb' \in {\rm Bor}(\fg)}\fb'=G.\fb$ is a closed subset of $\fg$.

According to \cite[(7.7)]{Fa2}, the set $\bigcup_{\fh \in {\rm Car}(\fg)}\fh$ lies dense in $\fg$. Given $\fh \in {\rm Car}(\fg)$, a consecutive application of \cite[(15.5)]{Hu1} and \cite[(12.1)]{Hu1} provides a closed, connected, solvable subgroup $H \subseteq G$ such that $\Lie(H)=\fh$. Since $H$ is contained in some Borel subgroup of $G$, it follows that $\fh \subseteq \bigcup_{\fb'\in {\rm Bor}(\fg)}\fb'$. As an upshot of the above discussion we conclude that
\[ \fg = \bigcup_{\fb'\in {\rm Bor}(\fg)}\fb'=G.\fb,\]
as desired. \end{proof}

\bigskip

\begin{Corollary} \label{MCR4} Let $G$ be a reductive group, $\fb \subseteq \fg$ be a Borel subalgebra. For a $U_0(\fg)$-module $M$ the following statements are equivalent:
\begin{enumerate}
\item $M$ has constant $j$-rank.
\item $M$ is $G$-stable and $M|_{U_0(\fb)}$ has constant $j$-rank. \end{enumerate} \end{Corollary}

\begin{proof}  (1) $\Rightarrow$ (2). This is a direct consequence of Theorem \ref{MCR2}.

(2) $\Rightarrow$ (1). Let $x$ be an element of $\cV_\fg$. Lemma \ref{MCR3} provides $y \in \cV_\fb$ and $g\in G$ such that $x=g^{-1}.y$. As $M$ is $G$-stable, there exists an isomorphism $\psi : M \stackrel{\sim}{\lra} M^{(g)}$. We therefore obtain
\[ \im x^j_M = \im \Ad(g^{-1})(y)^j_M = \im y^j_{M^{(g)}} = \im \psi \circ y^j_M,\]
so that
\[ \rk(x^j_M) = \rk(y^j_M).\]
Since $M|_{U_0(\fb)}$ has constant $j$-rank, it follows that $\rk(x_M^j) = \rk^j(M|_{U_0(\fb)})$. As a result, the closed subscheme $\Gamma^j(G_1)_M$ of equivalence classes of $\pi$-points of
non-maximal $j$-rank has an empty set of closed points. Thus, $\Gamma^j(G_1)_M = \emptyset$, and \cite[(4.5)]{FPe3} shows that $M$ has constant $j$-rank. \end{proof}

\bigskip

\section{Equal Images Modules}
Throughout this section, $\cG$ denotes a finite group scheme, defined over an algebraically closed field $k$ of characteristic $\Char(k)=p>0$. When dealing with equal images modules we shall mainly
use the notation introduced in \cite{CFS}.

\subsection{Preliminaries} Given $M \in \modd \cG$ and $\alpha_K \in \Pt(\cG)$,
we consider the $K$-linear map
\[ \ell_{\alpha_K} : M_K \lra M_K \ \ ; \ \ m \mapsto \alpha_K(t)m.\]
Thus, $\ell_{\alpha_K}=\alpha_K(t)_{M_K}$, and we shall use the latter notation when an emphasis of the underlying module is expedient. The following definition naturally extends the ones from \cite{CF} and \cite{CFS}.

\bigskip

\begin{Definition} A $\cG$-module $M$ is said to have the {\it equal images property}, if $\ell_{\alpha_K}^i(M_K)_Q = \ell_{\beta_L}^i(M_L)_Q$ for $1 \le i \le p\!-\!1$ and all $\alpha_K, \beta_L \in \Pt(\cG)$ and for all common field extensions $Q$ of $K$ and $L$.\end{Definition}

\bigskip
\noindent
Recall that for every $\pi$-point $\alpha_K \in \Pt(\cG)$ and every field extension $Q\!:\!K$, there is the extended $\pi$-point $\alpha_Q$, obtained from $\alpha_K$ via base change, cf.\ \cite[\S2]{FPe2}. It readily follows that
\[ \ell_{\alpha_K}^j(M_K)_Q = \ell_{\alpha_Q}^j(M_Q)\]
for every and $j \in \{1,\ldots,p\!-\!1\}$.

Note that a $\cG$-module $M$ has the equal images property if and only if for every $j \in \{1,\ldots,p\!-\!1\}$ there exists a subspace $V_j\subseteq M$ such that
\[ \ell_{\alpha_K}^j(M_K)=V_j\!\otimes_k\!K\]
for every $\alpha_K \in \Pt(\cG)$.

Suppose $\alpha_k \in \Pt(C(\cG))$ is a $\pi$-point of the center $C(\cG)$ of $\cG$. Let $M$ be a $\cG$-module such that $\ell_{\beta_K}(M_K)=\ell_{\alpha_K}(M_K)$ for all $\beta_K \in \Pt(\cG)$. Then we have
\[ \ell_{\beta_K}^i(M_K)= \ell_{\beta_K}^{i-1}(\ell_{\alpha_K}(M_K)) = \ell_{\alpha_K}(\ell_{\beta_K}^{i-1}(M_K)),\]
so that induction on $i$ implies that $M$ has the equal images property. In particular, our definition coincides with
the one of \cite{CFS} in case $\cG$ is abelian.

In the sequel, we let $\EIP(\cG)$ be the full subcategory of $\modd \cG$ consisting of those modules having the equal images property. Note that every equal images module has constant Jordan type. Given $d \in \{1,\ldots,p\}$, we denote by $\EIP(\cG)_d \subseteq \EIP(\cG)$ the full subcategory of $\EIP(\cG)$, whose objects satisfy $\Jt(M) = \bigoplus_{i=1}^da_i[i]$ for some $a_i \in \NN_0$. There results a filtration $\EIP(\cG)_1 \subseteq \EIP(\cG)_2 \subseteq \cdots \subseteq \EIP(\cG)_p = \EIP(\cG)$.

\bigskip
\noindent
Let $G=\ZZ/(p)\!\times\!\ZZ/(p)$. In \cite[\S2]{CFS} the authors introduce the equal images $G$-modules $W_{n,d}$. For future reference, we recall the definition:

\bigskip

\begin{Definition} Let $kG=k[X,Y]/(X^p,Y^p)$ and denote by $x$ and $y$ the residue classes of $X$ and $Y$, respectively. Let $1 \le d\le p$ and $n\ge d$. Then
\[ W_{n,d} := (\bigoplus_{i=1}^n kGv_i)/N_{n,d},\]
where $N_{n,d} := \langle \{x.v_1,y.v_n\} \cup \{x^d.v_i \ ; \ 2 \le i \le n\}\cup \{y.v_i\!-\!x.v_{i+1} \ ; \ 1\le i \le n\!-\!1\}\rangle$. \end{Definition}

\bigskip

\begin{Remark} Let $G=\ZZ/(p)\!\times\!\ZZ/(p)$. Each $W_{n,d}$ belongs to $\EIP(G)_d$ and has Loewy length $\ell\ell(W_{n,d})=d$, cf.\ \cite[(1.6)]{CFS}.\end{Remark}

\bigskip
\noindent
We denote by $k(\bullet \rightrightarrows \bullet)$ the path algebra of the Kronecker quiver. There is a canonical functor $F : \modd k(\bullet \rightrightarrows \bullet) \lra \modd k[X,Y]/(X^p,Y^p)$ which commutes with direct sums, while preserving indecomposables, and whose essential image coincides with the full subcategory of $\modd G$ consisting of all modules $M$ of Loewy length $\ell\ell(M) \le 2$, cf.\ \cite[(4.3)]{Be1}. The representations of the Kronecker quiver are well-understood, see \cite[(VIII.7)]{ARS}. In particular, there are three classes of indecomposable modules, namely the pre-projective modules of dimension vectors $(n,n\!+\!1)$, the pre-injective modules of dimension vectors $(n\!+\!1,n)$, and the regular
modules of dimension vectors $(n,n)$. Moreover, the modules belonging to the former two classes are uniquely determined by their dimension vectors.

\bigskip

\begin{Lem} \label{EIPr1} Let $G=\ZZ/(p)\!\times\!\ZZ/(p)$.
\begin{enumerate}
\item If $M \in \modd k(\bullet \rightrightarrows \bullet)$ is regular, then $F(M) \in \bigcap_{j=2}^{p-1}\CR_j(G)\!\smallsetminus\!\CR_1(G)$ does not have constant Jordan type.
\item If $M \in \modd k(\bullet \rightrightarrows \bullet)$ is pre-projective of dimension $2n\!+\!1 \ge 3$, then $F(M)$ has constant Jordan type $\Jt(F(M))=[1]\oplus n[2]$,
but $F(M) \not \in \EIP(G)$.
\item If $M \in \modd k(\bullet \rightrightarrows \bullet)$ is pre-injective of dimension $2n\!+\!1$, then $F(M) \in \EIP(G)$ has constant Jordan type $\Jt(F(M))=[1]\oplus n[2]$.
\end{enumerate}\end{Lem}

\begin{proof} (1) If $M$ is a regular module, then the assertion follows directly from \cite[(4.3.2)]{Be1}.

(2) If $M=(M_1,M_2)$ is a pre-projective module of the Kronecker quiver, then $\dim_kM_1=n$ and $\dim_kM_2=n\!+\!1$. Direct computation shows that $F(M)$ is the $kG$-module $V_n$ considered in \cite[(2.6)]{CFP}. This module has constant Jordan type, but does not belong to $\EIP(G)$. (In fact, these modules have the equal kernels property, which we shall discuss below.)

(3) If $M=(M_1,M_2)$ is a pre-injective module of the Kronecker quiver, then $\dim_kM_1=n\!+\!1$ and $\dim_kM_2=n$. Direct computation shows that $F(M)$ is the $kG$-module $W_{n+1,2}$ considered in \cite[(2.3)]{CFS}. \end{proof}

\bigskip
\noindent
For future reference we record the following basic fact, cf.\ \cite[(1.10)]{CFS}.

\bigskip

\begin{Lem} \label{EIPr2} Let $d \in \{1,\ldots,p\}$. The full subcategory $\EIP(\cG)_d$ of $\modd \cG$ is closed under finite direct sums, images of $\cG$-linear maps and direct summands. \end{Lem}

\begin{proof} It is clear that $\EIP(\cG)_d$ is closed under finite direct sums. Let $f : M \lra M'$ be an epimorphism originating in $M \in \EIP(\cG)_d$. Given a $\pi$-point $\alpha_K \in \Pt(\cG)$, we have
\[ \ell_{\alpha_K}^j(M'_K)= f_K(\ell_{\alpha_K}^j(M_K)),\]
so that $M' \in \EIP(\cG)$. As $M \in \EIP(\cG)_d$, the above identity yields $\ell_{\alpha_K}^d(M'_K)=(0)$, whence $M' \in \EIP(\cG)_d$. If $N$ is a direct summand of $M \in \EIP(\cG)_d$, then $N$ is an image of $M$ and thus contained in $\EIP(\cG)_d$. \end{proof}

\bigskip
\noindent
As a result, the full subcategory of equal images modules of Loewy length $\le 2$ is also closed under direct summands. Our above observations concerning $G = \ZZ/(p)\!\times\!\ZZ/(p)$ thus imply that such a module is of the form $\bigoplus_{i=1}^rW_{n_i,2}$. This was first observed in \cite[(4.1)]{CFS}.

The following result is inspired by \cite[\S 6]{CFS}.

\bigskip

\begin{Prop} \label{EIPr3} Let $\cG$ be a finite group scheme, $M$ be a $\cG$-module. For every $d \in \{1,\ldots,p\}$ there exists a unique submodule $\fK_d(M) \subseteq M$ such that
\begin{enumerate}
\item[(a)] $\fK_d(M) \in \EIP(\cG)_d$, and
\item[(b)] if $N \subseteq M$ belongs to $\EIP(\cG)_d$, then $N \subseteq \fK_d(M)$.\end{enumerate}\end{Prop}

\begin{proof} It is clear that $\fK_d(M)$ is uniquely determined. Let $\fK_d(M) \subseteq M$ be a submodule of maximal dimension subject to $\fK_d(M) \in \EIP(\cG)_d$. If $N \subseteq M$ belongs to
$\EIP(\cG)_d$, then Lemma \ref{EIPr2} shows that both, $\fK_d(M)\!\oplus\!N$ and its image $\fK_d(M)\!+\!N \subseteq M$, are contained in $\EIP(\cG)_d$. Consequently, $\fK_d(M)\!+\!N = \fK_d(M)$, so that $N \subseteq \fK_d(M)$.\end{proof}

\bigskip
\noindent
Recall that $X(\cG) = \Hom(k\cG,k)$ is the character group of the finite group scheme $\cG$.

\bigskip

\begin{Cor} \label{EIPr4} Let $\cG$ be a finite group scheme.
\begin{enumerate}
\item If $\varphi : M \lra N$ is a homomorphism of $\cG$-modules, then $\varphi(\fK_d(M)) \subseteq \fK_d(N)$.
\item If $\lambda \in X(\cG)$, then $M \in \EIP(\cG)_d$ if and only if $M\!\otimes_k\!k_\lambda \in \EIP(\cG)_d$.
\item The assignment $\fK_d : \modd \cG \lra \EIP(\cG)_d$ is a functor, which is right adjoint to the inclusion functor $\iota_d : \EIP(\cG)_d \lra \modd \cG$.
\item The category $\EIP(\cG)_d$ has enough injectives. Moreover, if $M \in \EIP(\cG)_d$  and $M \hookrightarrow I_M$ is an injective hull of $M$ in $\modd \cG$, then $M \hookrightarrow \fK_d(I_M)$ is an injective hull of $M$ in $\EIP(\cG)_d$. \end{enumerate}\end{Cor}

\begin{proof} (1) In view of Lemma \ref{EIPr2}, we have $\varphi(\fK_d(M)) \in \EIP(\cG)_d$, whence $\varphi(\fK_d(M)) \subseteq \fK_d(N)$.

(2) Let $\eta$ be the antipode of the Hopf algebra $k\cG$. The functor
\[ T_\lambda : \modd \cG \lra \modd \cG \ \ ; \ \ M \mapsto M\!\otimes_k\!k_\lambda\]
is an auto-equivalence with inverse $T_{\lambda\circ \eta}$. It thus suffices to show that $M\!\otimes_k\!k_\lambda \in \EIP(\cG)_d$ for every $M \in \EIP(\cG)_d$.
The $\cG$-module $M\!\otimes_k\!k_\lambda$ is canonically isomorphic to the module $M^{(\lambda)}$ with underlying $k$-space $M$ and action given by
\[ a\dact m := \sum_{(a)} \lambda(a_{(2)})a_{(1)}.m  \ \ \ \ \ \ \forall \ a \in k\cG,\ m \in M.\]
Let $\alpha_K \in \Pt(\cG)$ be a $\pi$-point. Then there exists an abelian unipotent subgroup $\cU \subseteq \cG_K$ such that $\im \alpha_K \subseteq K\cU$.
Since $\cU$ is unipotent, $\lambda_K|_{KU} = \varepsilon_{\cU_K}$ coincides with the co-unit of $K\cU$, so that $\alpha_K(t^i)\dact m = \alpha_K(t^i)m$ for every $m \in M_K$. Consequently,
$M^{(\lambda)} \cong M\!\otimes_k\!k_\lambda \in \EIP(\cG)_d$, as desired.

(3) In view of (1), $\fK_d$ is a functor such that the inclusion $\fK_d(X) \subseteq X$ gives rise to the identity
\[ \Hom_\cG(M,X) = \Hom_\cG(M,\fK_d(X))\]
for $M \in \EIP(\cG)_d$ and $X \in \modd \cG$. Thus, $\fK_d$ is right adjoint to the inclusion functor.

(4) Let $E$ be an injective $\cG$-module. According to (3), an exact sequence $(0) \lra M' \lra M$ of equal images modules induces an exact sequence
\[ \Hom_\cG(M,\fK_d(E)) \lra \Hom_\cG(M',\fK_d(E))\lra (0),\]
showing that $\fK_d(E)$ is an injective object in $\EIP(\cG)_d$. Similarly, if $M \hookrightarrow E_M$ is an injective hull of $M \in \EIP(\cG)_d$, then (1) shows that $M \hookrightarrow \fK_d(E_M)$ is an injective hull of $M$ in $\EIP(\cG)_d$. \end{proof}

\bigskip

\begin{Remark} Let $G=\ZZ/(p)\!\times\!\ZZ/(p)$. Being the category of trivial $G$-modules, $\EIP(G)_1$ is semi-simple and thus has enough projectives. By contrast, $(0)$ is the only projective object of $\EIP(G)_d$ for $d\ge 2$: Suppose that $P \in \EIP(G)_d\!\smallsetminus\!\{(0)\}$ is a projective object. Then $\ell\ell(P)\le d$, and if $\ell\ell(P)\ge 2$, then \cite[(4.4)]{CFS} provides a surjection $W_{n,\ell\ell(P)} \twoheadrightarrow P$, so that $P$ is a direct summand of $W_{n,\ell\ell(P)}$. Since $\ell\ell(P) \ge 2$, the module $W_{n,\ell\ell(P)}$ is indecomposable, implying the indecomposability of $P\cong W_{n,\ell\ell(P)}$. As direct sums of projectives are projective, and we conclude that $P=(0)$. Hence $\ell\ell(P)=1$, so that $P$ is a trivial $G$-module. The three-dimensional $W$-module $W_{2,2}$ belongs to $\EIP(\cG)_d$. Let $f : P \lra \Top(W_{2,2})$ be a nonzero homomorphism and consider the canonical projection $\pi : W_{2,2} \lra \Top(W_{2,2})$. Since $P$ is trivial, we have $\im \zeta \subseteq \Soc(W_{2,2}) = \Rad(W_{2,2})$ for every $\zeta \in \Hom_G(P,W_{2,2})$. It follows that $f$ does not factor through $\pi$. Thus, $P=(0)$, as desired.\end{Remark}

\bigskip

\begin{Example} Let $G=\ZZ/(p)\!\times\!\ZZ/(p)$, $d \in \{1,\ldots,p\}$. Every injective object $E \in \EIP(G)_d$ is a direct sum $E=\bigoplus_{i=1}^nW_{d,d}$ for some $n \in \NN_0$.

Since direct summands of injectives are injective and $\EIP(G)_d$ is closed under direct summands, we may assume that $E$ is indecomposable. It follows that $E$ is the injective hull of the trivial $G$-module. Consequently, Corollary \ref{EIPr4} shows that $E\cong \fK_d(kG)$. Thanks to \cite[(2.2)]{CFS}, there exists an isomorphism
\[ W_{d,d} \cong \Rad^{2p-d-1}(kG),\]
so that $\Rad^{2p-d-1}(kG)\subseteq \fK_d(kG)$. On the other hand, $\ell\ell(\fK_d(kG)) \le d$, whence $\fK_d(kG)\subseteq \Soc_d(kG)$. Direct computation shows that the latter space coincides with
$\Rad^{2p-d-1}(kG)$, implying that
\[ \fK_d(kG) =\Rad^{2p-d-1}(kG) \cong W_{d,d}\]
is injective.\end{Example}

\bigskip
\noindent
We shall write $\fK(M):= \fK_p(M)$. In \cite{CFS}, the authors provide an explicit description of $\fK(M)$ as the ``generic kernel" of operators in case $\cG$ is a $p$-elementary abelian group of rank $2$.

The full subcategory $\EIP(\cG) \subseteq \modd \cG$ is usually not closed under extensions. We consider the reduced group $G = \ZZ/(p)\!\times\!\ZZ/(p)$ and write
\[ kG = k[x,y],\]
where $x^p=0=y^p$. Let $M:=k[x,y]/(x,y^2)$. Then there exists a short exact sequence
\[ (0) \lra k \lra M \lra k \lra (0),\]
whose extreme terms belong to $\EIP(G)$, while the middle term does not. We have
\[ \fK(M) = \Soc(M) \cong k.\]

\bigskip
\noindent
Given a module $M$ over a finite group scheme $\cG$, we denote by $\cx_\cG(M)$ the {\it complexity} of the $k\cG$-module $M$, see \cite[(5.1)]{Be2} for further details. By general theory, the complexity coincides with the dimension of the {\it cohomological support variety} $\cV_\cG(M)$. If $\cG$ is an infinitesimal group of height $1$ with Lie algebra $\fg$, then Jantzen's Theorem \cite[Satz]{Ja1} implies $\cx_\cG(k)=\dim \cV_\fg$.

We define the {\it abelian unipotent rank} of $\cG$ via
\[ \aurk(\cG) := \max \{\cx_\cU(k) \ ; \  \cU \subseteq \cG \ \text{abelian unipotent}\}.\]
If $G$ is a finite group, then $\aurk(G)={\rm rk}_p(G)$ is the {\it $p$-rank} of $G$, that is, the maximal rank of a $p$-elementary abelian subgroup of $G$. By Quillen's Dimension Theorem \cite[(5.3.8)]{Be2}, this number coincides with the complexity $\cx_G(k)$ of the trivial module. In general, we have $\cx_\cG(k) \ge \aurk(\cG)$, and the example of the first Frobenius kernel $\SL(2)_1$ of the reduced group $\SL(2)$ shows that this inequality may be strict. In fact, the two numbers can be arbitrarily far apart:

\bigskip

\begin{Examples} (1) Suppose that $p \ge 3$ and let $\fh_n$ be the $(2n\!+\!1)$-dimensional Heisenberg algebra with trivial $p$-map. The maximal abelian unipotent subalgebras of $\fh_n$ correspond to
the maximal totally isotropic subspaces of $\fh_n/C(\fh_n)$. Consequently, $\aurk(\fh_n)=n\!+\!1$, while $\cx_{\fh_n}(k)= \dim \cV_{\fh_n}(k)=2n\!+\!1$.

(2) Let $G$ be a connected algebraic group, that is, a connected reduced (smooth) algebraic group scheme. If $G$ is not a torus, general theory provides a non-trivial connected closed (reduced) unipotent subgroup $U \subseteq G$. In view of \cite[(IV,\S2,2.10)]{DG}, the group $U$ contains a copy of the additive group $\GG_a$. Consequently, the $r$-th Frobenius kernel $G_r$ of $G$ has abelian unipotent rank $\aurk(G_r) \ge \aurk(\GG_{a(r)}) = r$.\end{Examples}

\bigskip
\noindent
A $\cG$-module $M$ is called {\it projective-free} if $(0)$ is the only projective direct summand of $M$. Since $k\cG$ is self-injective, a projective-free $\cG$-module does not possess any non-zero projective submodules. Let $(P_n,\partial_n)_{n\ge 0}$ be a minimal projective resolution of $M$. Given $n \ge 1$, the {\it $n$-th Heller shift} $\Omega^n_\cG(M) := \im \partial_n = \ker \partial_{n-1}$ is projective-free. Dually, a minimal injective resolution $(E_n,\partial_n)_{n\ge 0}$ gives rise to negative Heller shifts $\Omega^{-n}_\cG(M):= E_{n-1}/\ker\partial_{n-1}$ ($n\ge 1$).

\bigskip

\begin{Lem} \label{EIPr5} Let $\cG$ be a finite group scheme, $M$ be a $\cG$-module.
\begin{enumerate}
\item If $M \in \EIP(\cG)$, then $M|_\cH \in \EIP(\cH)$ for every subgroup scheme $\cH \subseteq \cG$.
\item If $\cH \subseteq \cG$ is a subgroup scheme, then $\fK(M)|_{\cH} \subseteq \fK(M|_{\cH})$.
\item If $\aurk(\cG) \ge 2$, then every $M \in \EIP(\cG)$ is projective-free.
\item Suppose that $\cH \subseteq \cG$ is a subgroup scheme of abelian unipotent rank $\aurk(\cH) \ge 2$. If $\Omega^n_\cG(M) \in \EIP(\cG)$ for some $n\in \ZZ$, then $\Omega^n_\cG(M)|_\cH \cong \Omega^n_\cH(M|_\cH)$.\end{enumerate} \end{Lem}

\begin{proof} (1),(2) This is a direct consequence of the fact that the canonical inclusion $\iota : \cH \hookrightarrow \cG$ defines an inclusion $\iota_{\ast,\cH} : \Pt(\cH) \hookrightarrow \Pt(\cG)$.

(3) Let $P \subseteq M$ be a projective direct summand of $M$. Thanks to Lemma \ref{EIPr2}, we obtain $P \in \EIP(\cG)$. By assumption, $\cG$ contains an abelian unipotent subgroup $\cU$ of abelian unipotent rank $\aurk(\cU) \ge 2$ such that the projective module $P|_\cU$ belongs to $\EIP(\cU)$. As noted earlier, there exists an isomorphism
$k\cU \cong U_0(\fu)$, where $\fu$ is an abelian $p$-unipotent restricted Lie algebra. In particular, the nullcone $\cV_\fu$ is a $p$-subalgebra of $\fu$, so that $P|_{\cV_\fu} \in \EIP(\cV_\fu)$. Jantzen's Theorem \cite[Satz]{Ja1} implies $\dim_k\cV_\fu = \cx_{\fu}(k) = \aurk(\fu) \ge 2$, so that $P|_{\cV_{\fu}}=(0)$ and $P=(0)$.

(4) General theory shows that $\Omega^n_\cG(M)|_\cH \cong \Omega^n_\cH(M|_\cH) \oplus ({\rm proj.})$. Since $\Omega^n_\cG(M)|_\cH \in \EIP(\cH)$, part (3) implies our assertion. \end{proof}

\bigskip
\noindent
Recall that every $M \in \EIP(\cG)$ has constant Jordan type. The following result elaborates on \cite[(4.2)]{CFS}.

\bigskip

\begin{Prop} \label{EIPr6} Let $\cG$ be a finite group scheme of abelian unipotent rank $\aurk(\cG) \ge 2$. If $(0) \ne M \in \EIP(\cG)$, then there exist $d \in \{1,\ldots,p\}$ and
$a_1, \ldots, a_d \in \NN$ such that $\Jt(M) = \bigoplus_{i=1}^da_i[i]$. \end{Prop}

\begin{proof} Suppose first that $\cG = \cU$ is abelian unipotent. Then there exists a restricted Lie algebra $\fu$ such that
\begin{enumerate}
\item[(a)] $k\cU \cong U_0(\fu)$, and
\item[(b)] $\fu$ is an abelian $p$-unipotent restricted Lie algebra.\end{enumerate}
We consider the $p$-subalgebra $\cV_\fu \subseteq \fu$. Since $\dim_k\cV_\fu = \aurk(\cU) \ge 2$, the algebra $\cV_\fu$ contains a two-dimensional $p$-subalgebra $\fu_2$. As the restriction $M|_{\fu_2}$ has the equal images property, \cite[(4.2)]{CFS} provides $a_1,\ldots, a_d \in \NN$ such that $\Jt(M|_{\fu_2})= \bigoplus_{i=1}^da_i[i]$. Consequently, $M \in \CJT(\cU)$ has constant Jordan type $\Jt(M) = \bigoplus_{i=1}^da_i[i]$.

In the general case, $\cG$ contains an abelian unipotent $\cU$ subgroup of abelian unipotent rank $\ge 2$. Since $M|_\cU \in \EIP(\cU)$ and $M$ has constant Jordan type $\Jt(M) = \Jt(M|_\cU)$,
our assertion follows from the above. \end{proof}

\subsection{The representation type of $\EIP(\ZZ/(p)\!\times\!\ZZ/(p))$} Let $\cC \subseteq \modd_\Lambda$ be a full
subcategory of the category of finite-dimensional modules of a finite-dimensional associative $k$-algebra $\Lambda$
such that $\cC$ is closed under direct sums, direct summands and images of isomorphisms. We say that $\cC$ has {\it finite
representation type}, provided $\cC$ affords only finitely many isoclasses of indecomposable objects. The category $\cC$ is referred to as {\it tame}, if for every $d>0$ there exist finitely many $(\Lambda,k[X])$-bimodules $V_1,\ldots, V_{n_d}$ such that
\begin{enumerate}
\item[(a)] the right $k[X]$-module $V_i$ is free for all $i \in \{1,\ldots,n_d\}$, and
\item[(b)] if $M \in \cC$ is indecomposable of dimension $d$, then there exist $i \in \{1,\ldots,n_d\}$ and an algebra homomorphism $\lambda : k[X] \lra k$ such that $M \cong V_i\!\otimes_{k[X]}\!k_\lambda$. \end{enumerate}
Let $G=\ZZ/(p)\!\times\!\ZZ/(p)$. If $p=2$, then Lemma \ref{EIPr1} shows that there exists exactly one indecomposable equal images module in each odd dimension (and none in any even dimension). Thus, the category $\EIP(\cG)$ has tame representation type, with only finitely many indecomposable objects in each dimension.

The purpose of this section is to show that the category $\EIP(G)$ is far from being tame whenever $p\ge 3$. More precisely, we shall prove that for every $n \in \NN$ there exists a natural number $d_n \in \NN$ such that the indecomposable equal images $G$-modules of dimension $d_n$ cannot be parametrized by an $n$-dimensional variety.

Let $V$ be a finite-dimensional $k$-vector space. We fix a decomposition $V=W \oplus \Omega$ of $V$ and a subspace $\Gamma \subseteq W$. Let $\pi_\Gamma : W \lra W/\Gamma$ and $\pi_\Omega : V \lra V/\Omega$ be the canonical projections and choose sections $s_\Gamma : W/\Gamma \lra W$ and $s_\Omega : V/\Omega \lra V$ such that $W=s_\Omega(V/\Omega)$.

Given $f \in \Hom_k(W/\Gamma,\Gamma)$, the map $f\!+\!s_\Gamma$ is a section of $\pi_\Gamma$ and every section of $\pi_\Gamma$ arises in this fashion. We put $U_f := (f\!+\!s_\Gamma)(W/\Gamma)$ and note that
$(f\!+\!s_\Gamma)\circ \pi_\Gamma : U_0 \lra U_f$ is an isomorphism of $k$-vector spaces.

By construction, we have $W=U_f\oplus \Gamma$ and $V = U_f\oplus \Gamma \oplus \Omega$ for every $f \in \Hom_k(W/\Gamma,\Gamma)$. Moreover, the map
\[ \zeta_f : U_0\oplus \Gamma \oplus \Omega \lra U_0\oplus \Gamma \oplus \Omega \ \ ; \ \ (u,\gamma,\omega) \mapsto (s_\Gamma\circ\pi_\Gamma(u), f\circ\pi_\Gamma(u)+\gamma,\omega)\]
is an automorphism of $V$.

We let $\pr_f : V \lra \Gamma \oplus \Omega$ be the projection associated to the decomposition $V=U_f\oplus\Gamma\oplus\Omega$, so that $\pr_f\circ \lambda|_{\Gamma\oplus\Omega} \in \End_k(\Gamma\oplus\Omega)$ for every $\lambda \in \End_k(V)$.

\bigskip

\begin{Lem} \label{RTEI1} Let $\lambda \in \End_k(V)$. Then the map
\[ \tau_\lambda : \Hom_k(W/\Gamma,\Gamma) \lra \End_k(\Gamma\oplus\Omega) \ \ ; \ \ f \mapsto \pr_f\circ \lambda|_{\Gamma\oplus\Omega}\]
is a morphism of affine varieties. \end{Lem}

\begin{proof} Note that the map
\[ \zeta : \Hom_k(W/\Gamma,\Gamma) \lra \End_k(V) \ \ ; \ \ f \mapsto \zeta_f\]
is a morphism of affine varieties that factors through $\GL(V)$. Consequently, $f \mapsto \zeta_f^{-1}$ defines a morphism $\Hom_k(W/\Gamma,\Gamma) \lra \GL(V)$, so that
\[ \Hom_k(W/\Gamma,\Gamma) \lra \End_k(\Gamma\oplus\Omega) \ \ ; \ \ f \mapsto \pr_0\circ \zeta_f^{-1}\circ\lambda|_{\Gamma\oplus\Omega}\]
is also a morphism. Since $\pr_f\circ \zeta_f=\pr_0$, the latter map coincides with $\tau_\lambda$. \end{proof}

\bigskip
\noindent
Let $\Lambda$ be a finite-dimensional $k$-algebra. For $d\in \NN$, we shall interpret the affine variety $\modd_\Lambda^d$ of $\Lambda$-module structures on a given $d$-dimensional vector space $M$ as follows: If $\{x_1,\ldots,x_r\}$ is a basis of $\Lambda$, then each element of $\modd_\Lambda^d$ corresponds to an $r$-tuple $(f_1,\ldots, f_r)\in \End_k(M)^r$ satisfying the relations obtained when writing the identity and the products $x_ix_j$ of two basis vectors as a linear combination of the $x_\ell$. We let $\ind_\Lambda^d \subseteq \modd^d_\Lambda$ be the constructible subset of indecomposable modules of dimension $d$.

Suppose that $\Lambda$ is local and let $M\in \modd_\Lambda$. Then every subspace $X \subseteq \Soc(M)$ is a $\Lambda$-submodule. We fix a subspace $\Gamma \subseteq \Soc(M)$
and put $d:= \dim_k\Gamma\!+\!\dim_kM/\Soc(M)$. Choose a section $s_\Gamma : \Soc(M)/\Gamma \lra \Gamma$ of the canonical projection $\pi_\Gamma : \Soc(M) \lra \Soc(M)/\Gamma$. As before, every
$f \in \Hom_k(\Soc(M)/\Gamma,\Gamma)$ defines a $\Lambda$-submodule
\[U_f := (f\!+\!s_\Gamma)(\Soc(M)/\Gamma)\subseteq \Soc(M)\]
such that $U_f\oplus \Gamma = \Soc(M)$. In particular, we have $\dim_kM/U_f = d$.

\bigskip

\begin{Prop}\label{RTEI2} Let $\Gamma \subseteq \Soc(M)$ be a subspace, $d:=\dim_k\Gamma\!+\!\dim_kM/\Soc(M)$. Then
\[ \rho_\Gamma : \Hom_k(\Soc(M)/\Gamma,\Gamma) \lra \modd^d_\Lambda \ \ ; \ \ f \mapsto M/U_f\]
is a morphism of affine varieties. \end{Prop}

\begin{proof} Let $\Omega \subseteq M$ be a subspace such that $\Soc(M)\oplus \Omega =M$. Given $f \in \Hom_k(\Soc(M)/\Gamma,\Gamma)$, we let $\pr_f : M \lra \Gamma\oplus\Omega$ be the projection
defined by $M=U_f\oplus\Gamma\oplus\Omega$. Let $\lambda \in \End_k(M)$ be such that $U_f$ is $\lambda$-invariant. We denote by $\pi_f : M \lra M/U_f$ and $\bar{\lambda}_f \in \End_k(M/U_f)$ the canonical projection and the unique $\Lambda$-linear map satisfying
\[ \bar{\lambda}_f\circ \pi_f = \pi_f\circ \lambda,\]
respectively. Then $\pi_f|_{\Gamma\oplus\Omega} : \Gamma\oplus\Omega \lra M/U_f$ is an isomorphism and we have
\[ \bar{\lambda}_f\circ (\pi_f|_{\Gamma\oplus\Omega}) = (\pi_f|_{\Gamma\oplus\Omega})\circ \tau_{\lambda}(f),\]
whence
\[ (\ast) \ \ \ \ \ \ \ \ \ \ \tau_\lambda(f) = (\pi_f|_{\Gamma\oplus\Omega})^{-1}\circ \bar{\lambda}_f\circ (\pi_f|_{\Gamma\oplus\Omega}).\]
Let $\lambda_1, \ldots, \lambda_r \in \End_k(M)$ be linear maps determining the $\Lambda$-module $M$. Then $(\bar{\lambda}_1)_f, \ldots, (\bar{\lambda}_r)_f \in \End_k(M/U_f)$ determine $M/U_f$,
and ($\ast$) implies that the $\tau_{\lambda_i}(f)$ endow $\Gamma\oplus\Omega$ with the structure of a $\Lambda$-module such that $\pi_f|_{\Gamma\oplus\Omega}$ is an isomorphism of $\Lambda$-modules.
Thanks to Lemma \ref{RTEI1}, the map
\[ \bar{\rho}_\Gamma : \Hom_k(\Soc(M)/\Gamma,\Gamma) \lra \End_k(\Gamma\oplus\Omega)^r \ \ ; \ \ f \mapsto (\tau_{\lambda_1}(f),\ldots,\tau_{\lambda_r}(f))\]
is a morphism, which, by the above, factors through $\modd_\Lambda^d$. As a result, $\rho_\Gamma$ is a morphism. \end{proof}

\bigskip
\noindent
We require the following well-known basic result concerning actions of algebraic groups.

\bigskip

\begin{Lem}\label{RTEI3} Let $G$ be an algebraic group acting on a variety $V$. Suppose that $C_1, \, C_2 \subseteq V$ are closed subsets of $V$ such that there is a dense open subset $O_2 \subseteq C_2$ of $C_2$ with
\begin{enumerate}
\item[(a)] $O_2 \subseteq G\dact C_1$, and
\item[(b)] $|O_2\cap G\dact v| <  \infty$ for every $v \in V$.\end{enumerate}
Then we have $\dim C_2 \le \dim C_1$.\end{Lem}

\begin{proof} Let $X_1, \ldots, X_n$ be the irreducible components of $C_2$. The assumption $O_2\cap X_j = \emptyset$ implies $O_2 \subseteq \bigcup_{i\ne j} X_i$, whence $C_2 = \bar{O}_2 \subseteq \bigcup_{i\ne j} X_i$, a contradiction. As a result, $O_2\cap X_i$ is a dense open subset of $X_i$ for every $i \in \{1,\ldots,n\}$, and we may assume $C_2$ to be irreducible.

Let $X$ be the inverse image of $C_2$ under the multiplication $G\times C_1 \stackrel{\mu}{\lra}\overline{G\dact C_1}$. By virtue of (a), $\im \mu \cap C_2$ contains a dense open subset of $C_2$, implying that the restriction $\varphi : X \lra C_2$ of $\mu$ is a dominant morphism. As $C_2$ is irreducible, there exists an irreducible component $Y \subseteq X$ such that $\overline{\varphi(Y)}=C_2$. We denote the induced dominant morphism by $\psi : Y \lra C_2$.

The natural projection $G \times C_1 \lra C_1$ induces a map $Y \lra C_1$, whose image has closure $W$. There results a dominant morphism $\lambda : Y \lra W$ of irreducible varieties. A twofold application of the generic fiber dimension theorem \cite[(I,\S8,Cor.1)]{Mu} provides a dense open subset $O_1\subseteq Y$ such that
\[ \dim Y\!-\!\dim C_2 = \dim \psi^{-1}(\psi(x)) \ \ \ \text{and} \ \ \  \dim Y\!-\!\dim W = \dim \lambda^{-1}(\lambda(x)) \ \ \ \ \forall \ x \in O_1.\]
By the same token, the fibers $\psi^{-1}(\psi(x))$ and $\lambda^{-1}(\lambda(x))$ are equidimensional for every $x \in O_1$. Since $\psi$ is dominant, $\tilde{O} := \psi^{-1}(O_2)$ is a dense, open subset of $Y$.

Let $(g,a) \in \tilde{O} $. For $(h,b) \in \lambda^{-1}(\lambda(g,a))\cap \tilde{O}$ we have $a=b$ as well as $g.a, \, h.a \in O_2$. Thus, $(h,b) \in \psi^{-1}(O_2\cap G\dact a)$, so that
\[ (\ast) \ \ \ \ \ \ \ \ \lambda^{-1}(\lambda(g,a))\cap \tilde{O} \subseteq  \psi^{-1}(O_2\cap G\dact a) \ \ \ \ \ \ \ \ \forall \ (g,a) \in \tilde{O}.\]
Since $Y$ is irreducible, $O := \tilde{O} \cap O_1$ is a dense open subset of $Y$. Given $x \in  \psi^{-1}(O_2\cap G\dact a) \cap O$, we have $x \in \psi^{-1}(\psi(x))$, along with $\psi(x) \in \psi(O) \cap O_2\cap G\dact a$. Thanks to property (b), there exists a finite subset $\cF \subseteq O$ such that
\[  \psi^{-1}(O_2\cap G\dact a) \cap O \subseteq \bigcup_{x \in \cF} \psi^{-1}(\psi(x)).\]
Now let $(g,a) \in O$, and consider an irreducible component $Z \subseteq \lambda^{-1}(\lambda(g,a))$ containing $(g,a)$. Observing equidimensionality as well as $\cF \cup \{(g,a)\} \subseteq O_1$, we obtain
\begin{eqnarray*}
\dim \lambda^{-1}(\lambda(g,a)) & = & \dim Z = \dim Z\cap O \le \dim \psi^{-1}(O_2\cap G \dact a)\cap O\\
& \le & \max_{x \in \cF} \dim \psi^{-1}(\psi(x)) =  \dim \psi^{-1}(\psi(g,a)),
\end{eqnarray*}
where the first inequality follows from ($\ast$). There thus results the estimate
\[ \dim C_2 = \dim Y\! -\! \dim \psi^{-1}(\psi(g,a)) \le \dim Y\! - \!\dim \lambda^{-1}(\lambda(g,a)) \le \dim W \le \dim C_1,\]
as desired. \end{proof}

\bigskip
\noindent
Let $\cC \subseteq \modd_\Lambda$ be a full subcategory such that
\begin{enumerate}
\item[(a)] $\cC$ is closed under direct sums and direct summands, and
\item[(b)] $\cC$ is closed under isomorphisms. \end{enumerate}
Thus, every object in $\cC$ decomposes into its indecomposable constituents and the subset $\cC^d \subseteq \modd^d_\Lambda$, consisting of the module structures yielding objects of $\cC$, is $\GL_d(k)$-invariant. For every $d \in \NN$, we let $C^d \subseteq \modd_\Lambda^d$ be a closed subset of minimal dimension subject to $\ind_\Lambda^d\cap\cC^d \subseteq \GL_d(k)\dact C^d$.
Given $n \in \NN$, we say that $\cC$ {\it requires at least $n$ parameters}, provided $\dim C^d \ge n$ for some $d \ge 1$. If the sequence $(\dim C^d)_{d\ge 1}$ is unbounded, then $\cC$ is
said to {\it require infinitely many parameters}.

\bigskip

\begin{Lem} \label{RTEI4} If $\cC$ is tame, then $\dim C^d \le 1$ for every $d \in \NN$. \end{Lem}

\begin{proof} This follows as in \cite[(2.1)]{dP}. \end{proof}

\bigskip
\noindent
If $\cC = \modd_\Lambda$, then Drozd's tame-wild theorem \cite{Dr,CB} ensures that $\cC$ not being tame is equivalent to $\cC$ being wild, that is, $\cC$ fulfills the usual condition involving $(\Lambda,k\langle x,y\rangle)$-bimodules. This implies in particular that a classification of the indecomposable objects of $\cC$ is hopeless. In view of \cite[Thm. 2]{DrG}, the tame-wild dichotomy also holds if $\cC \subseteq \modd_\Lambda$ is an open subcategory. For arbitrary $\cC$, the presence of two-parameter families of indecomposables does not necessarily imply that there is no reasonable classification of the indecomposable objects.

We are now in a position to prove a criterion for certain subcategories $\cC \subseteq \modd_\Lambda$ not be tame. Let $\Lambda$ be a finite-dimensional algebra. A full subcategory $\cC \subseteq \modd_\Lambda$ which is closed under direct sums and images of all $\Lambda$-linear maps is called {\it image-closed}. Note that such a subcategory is also closed under direct summands.

\bigskip

\begin{Prop} \label{RTEI5} Let $\Lambda$ be local and suppose that $\cC \subseteq \modd_\Lambda$ is image-closed. Let $M \in \cC$ be an object such that
\begin{enumerate}
\item[(a)] there exists a submodule $(0) \subsetneq \Gamma \subsetneq \Soc(M)$ such that $M/U_f$ is indecomposable for every $f \in \Hom_k(\Soc(M)/\Gamma,\Gamma)$, and
\item[(b)] if $U,V \subseteq \Soc(M)$ are submodules such that $M/U \cong M/V$, then $U=V$. \end{enumerate}
Then $\cC$ requires at least $\dim_k\Hom_k(\Soc(M)/\Gamma,\Gamma)$ parameters. \end{Prop}

\begin{proof} We put $d:= \dim_k\Gamma\!+\!\dim_kM/\Soc(M)$. According to Proposition \ref{RTEI2}, the map
\[ \rho_\Gamma : \Hom_k(\Soc(M)/\Gamma,\Gamma) \lra \modd^d_\Lambda \ \ ; \ \ f \mapsto M/U_f\]
is a morphism of affine varieties. Since $\cC$ is image-closed and $M \in \cC$, condition (a) entails $\im \rho_\Gamma \subseteq \cC^d\cap \ind_\Lambda^d$. Let $C^d \subseteq \modd^d_\Lambda$ be a closed subset of minimal dimension subject to $\cC^d\cap \ind_\Lambda^d \subseteq \GL_d(k)\dact C^d$. Setting $C_1:= C^d$ and $C_2 := \overline{\im \rho}_\Gamma$, we obtain
\[ C_2 \subseteq \overline{\GL_d(k)\dact C_1}.\]
By Chevalley's theorem (cf.\ \cite[(10.19)]{GW}), $\im \rho_\Gamma$ contains a dense open subset $O_2$ of $C_2$. As $\im \rho_\Gamma \subseteq C_2\cap\GL_d(k)\dact C_1$, we have $O_2\subseteq C_2\cap \GL_d(k)\dact C_1$. It now readily follows from (b) and Lemma \ref{RTEI3} that $\dim C_2 \le \dim C_1$. By (b), the assumption $M/U_f=M/U_g$ implies $U_f=U_g$, whence $f=g$. Consequently, (a) implies
\[\dim C_1 \ge \dim C_2 \ge \dim_k\Hom_k(\Soc(M)/\Gamma,\Gamma),\]
so that $\cC$ requires the asserted number of parameters. \end{proof}

\bigskip

\begin{Thm}\label{RTEI6} Suppose that $p\ge 3$. Then the image-closed subcategory $\EIP(\ZZ/(p)\!\times\!\ZZ/(p))$ requires infinitely many parameters. In particular, $\EIP(\ZZ/(p)\!\times\!\ZZ/(p))$
is not tame. \end{Thm}

\begin{proof} Let $3 \le d \le p$ and $n\ge d\!+\!2$. According to \cite[(4.8)]{CFS}, the equal images module $M:=W_{n,d}$ meets condition (b) of Proposition \ref{RTEI5}. Moreover, $\Soc(M)=\Rad^{d-1}(M)$ and $M$ has constant Jordan type $\Jt(M)=\bigoplus_{i=1}^{d-1}[i]\oplus (n\!-\!d\!+\!1)[d]$, while $\dim_k\Soc(M)=n\!-\!d\!+\!1 \ge 3$, see \cite[(2.3)]{CFS}.

Let $\Gamma \subseteq \Soc(M)$ be a submodule of dimension $\dim_k\Gamma = \dim_k\Soc(M)\!-\!1$, so that
\[\dim_k \Hom_k(\Soc(M)/\Gamma,\Gamma)=\dim_k\Gamma= n\!-\!d.\]
We shall show that $\Gamma$ satisfies condition (a) of (\ref{RTEI5}). Given $f \in \Hom_k(\Soc(M)/\Gamma,
\Gamma)$, we write $\Jt(M)=\bigoplus_{i=1}^pa_i[i]$ and $\Jt(M/U_f)= \bigoplus_{i=1}^p\bar{a}_i[i]$. Let $\alpha: \fA_{p,k} \lra k(\ZZ/(p)\!\times\!\ZZ/(p))$ be a $\pi$-point of $\ZZ/(p)\!\times\!\ZZ/(p)$. The coefficients $a_i$ are determined by
\[ \dim_k \im \alpha(t)^i_M-\dim_k\im\alpha(t)^{i+1}_M = \sum_{j=i+1}^pa_j\]
for $0\le i \le p\!-\!1$. In view of $U_f \subseteq \Rad^{d-1}(M)=\im\alpha(t)^{d-1}_M$ (cf.\ \cite[(1.7)]{CFS}), we have $\im \alpha(t)^i_{M/U_f} = \im\alpha(t)^i_M/U_f$ for $0\le i \le d\!-\!1$, while $\im \alpha(t)^i_{M/U_f} = (0) = \im\alpha(t)^i_M$ for
$i\ge d$. This readily implies $\bar{a}_i = a_i=0$ for $i\ge d\!+\!1$, $\bar{a}_d=a_d\!-\!1$, $\bar{a}_{d-1} = a_{d-1}\!+\!1$, as well as $\bar{a}_i=a_i$ for $1\le i \le d\!-\!2$. Consequently,
\[\Jt(M/U_f) = \bigoplus_{i=1}^{d-2}\, [i]\oplus 2[d\!-\!1]\oplus (n\!-\!d)[d].\]
Suppose that $M/U_f = X\oplus Y$. Then $X,Y \in \EIP(\ZZ/(p)\!\times\!\ZZ/(p))$, and we have $\Jt(M/U_f)\cong \Jt(X)\oplus \Jt(Y)$. If $X \ne (0)$, then \cite[(2.3)]{CFS} provides $1 \le \ell \le d$ and
$b_1,\ldots , b_\ell \in \NN$ such that $\Jt(X)=\bigoplus_{i=1}^\ell b_i[i]$. Since $d\!-\!2 \ge 1$, we obtain $\Jt(Y) = \bigoplus_{i=2}^d c_i[i]$, so that another application of \cite[(2.3)]{CFS} gives
$Y=(0)$. As a result, the $k(\ZZ/(p)\!\times\!\ZZ/(p))$-module $M/U_f$ is indecomposable.

Let $\ell:= \dim_kW_{n,d}\!-\!1$ and let $C^\ell \subseteq \modd^\ell_{k(\ZZ/(p)\times\ZZ/(p))}$ be a closed subset of minimal dimension subject to ${\rm ind}^\ell_{k(\ZZ/p)\times\ZZ/(p))}\cap \EIP(\ZZ/(p)\!\times\!\ZZ/(p)) \subseteq \GL_\ell(k)\dact C^\ell$. It now follows from Proposition \ref{RTEI5} that $\dim C^\ell \ge n\!-\!d$. As a result, $\EIP(\ZZ/(p)\!\times\!\ZZ/(p))$ requires infinitely many parameters. In view of Lemma \ref{RTEI4} this shows in particular that the category $\EIP(\ZZ/(p)\!\times\!\ZZ/(p))$ is not tame. \end{proof}

\bigskip

\begin{Remark} Let $d \ge 3$, $U \subseteq \Soc(W_{n,d})$ be a submodule. The above arguments show that the equal images module $W_{n,d}/U$ is indecomposable of constant Jordan type
\[\Jt(W_{n,d}/U) = \bigoplus_{i=1}^{d-2}[i] \oplus (\dim_kU\!+\!1)[d\!-\!1]\oplus (n\!-\!d\!+\!1\!-\!\dim_kU)[d].\] \end{Remark}

\bigskip

\subsection{Equal images modules for Frobenius kernels}
Throughout this section, $\cG$ denotes a finite group scheme, defined over an algebraically closed field $k$ of characteristic $p>0$. Particular cases of interest are given by the Frobenius kernels $(G_r)_{r\ge 1}$ of a reduced (smooth) group scheme $G$.

Given a $\cG$-module $M$, we denote by $Z_\cG(M)$ the kernel (centralizer) of the representation $\varrho : \cG \lra \GL(M)$, see \cite[(I.2.12)]{Ja2}. Let $\cN \unlhd \cG$ be a normal subgroup. In the sequel, we shall identify $\modd (\cG/\cN)$ with the full subcategory of $\modd \cG$, whose objects satisfy $\cN \subseteq Z_\cG(M)$.

Recall that $k[\GG_{a(r)}] = k[X]/(X^{p^r})$, and let $u_0, \ldots, u_{r-1} \in k\GG_{a(r)}$ be the elements such that $u_i(x^j) = \delta_{p^i,j}$, where $x:= X\!+\!(X^{p^r})$. Then we have $k\GG_{a(r)} =k[u_0,\ldots,u_{r-1}]$.

Suppose that $\cG$ is an infinitesimal group of height $r$. In \cite[\S6]{SFB2}, the authors introduce the {\it rank variety} $V_r(\cG)_M$ of a $\cG$-module $M$. By definition,
\[ V_r(\cG)_M = \{ \varphi \in \Hom(\GG_{a(r)},\cG) \ ; \ M|_{k[u_{r-1}]} \ \text{is not projective}\}.\]
Here, $M|_{k[u_{r-1}]}$ denotes the restriction to $k[u_{r-1}]$ of the pull-back $\varphi^\ast(M) \in \modd k\GG_{a(r)}$ of the Hopf algebra homomorphism $\varphi : k\GG_{a(r)} \lra k\cG$ corresponding to $\varphi$.

\bigskip

\begin{Lem} \label{EIFK1} Let $\cG$ be a finite group scheme, $M$ be a $\cG$-module.
\begin{enumerate}
\item The $\cG$-module $M$ has constant Jordan type $\Jt(M) = (\dim_kM)[1]$ if and only if $\cG_\pi \subseteq Z_\cG(M)$.
\item If the subgroup $\cG_\pi \subseteq \cG$ generated by all quasi-elementary subgroups is normal in $\cG$, then $M$ has constant Jordan type $\Jt(M) = (\dim_kM)[1]$ if and only if $M \in \modd(\cG/\cG_\pi)$.\end{enumerate} \end{Lem}

\begin{proof} (1) We put $\cN := Z_\cG(M)$. Suppose that $\cG_\pi \subseteq \cN$. Let $\alpha_K \in \Pt(\cG)$ be a $\pi$-point that is maximal for $M$. Thus, if $\beta_L \in \Pt(\cG)$ is a $\pi$-point such that $\rk(\beta_L(t)^j_{M_L})\ge\rk(\alpha_K(t)^j_{M_K})$ for every $j \in \{1,\ldots,p\!-\!1\}$, then $\Jt(M,\beta_L) =
\Jt(M,\alpha_K)$. According to Theorem \ref{SFS3}, there exists a $\pi$-point $\beta_L \sim \alpha_K$ that factors through $L\cN$. Consequently, \cite[(4.10)]{FPS} implies $\Jt(M,\alpha_K)= \Jt(M,\beta_L)=(\dim_kM)[1]$. As a result, $M$ has constant Jordan type $\Jt(M)= (\dim_kM)[1]$.

Suppose that $\Jt(M)=(\dim_kM)[1]$. Let $\cE \subseteq \cG$ be a quasi-elementary subgroup. Then $N :=M|_{\cE}$ is an $\cE$-module of constant Jordan type $\Jt(N) = (\dim_kM)[1]$.

There exists an isomorphism $\varphi : \GG_{a(r)} \lra \cE^0$. Given $s \le r$,
\[ \psi_s : \GG_{a(r)} \lra \cE^0 \  \ ; \  \ y \mapsto \varphi(y^{p^{r-s}})\]
is an infinitesimal one-parameter subgroup of $\cE^0$, such that $\psi_s(u_{r-1})=\varphi(u_{s-1})$. Owing to \cite[(2.7)]{FPe1}, 
$\psi_s|_{k[u_{r-1}]}$ is a $\pi$-point of $\cE^0$, so that $0 = \psi(u_{r-1})_N = \varphi(u_{s-1})_N$. Consequently, the annihilator of $N|_{\cE^0}$ coincides with the augmentation ideal $(k\cE^0)^\dagger$ of $k\cE^0$. As a result, $\cE^0 \subseteq \cN$. 

Now let $C_p \subseteq \cE_{\rm red}$ be a cyclic subgroup. Then there exists $\alpha_k\in \Pt(\cG)$ such that
$kC_p = \im \alpha_k$. Since $\alpha_k^\ast(M) = (\dim_kM)[1]$, it follows that $C_p \subseteq \cN_{\rm red}$, whence $\cE_{\rm red} \subseteq \cN$.

We conclude that $\cE \subseteq \cN$, implying $\cG_\pi \subseteq \cN$.

(2) Suppose that $M$ has constant Jordan type $\Jt(M)=(\dim_kM)[1]$. In view of (1), we have $M \in \modd(\cG/\cG_\pi)$.\end{proof}

\bigskip
\noindent
Recall that a {\it $p$-point} of $\cG$ is a $\pi$-point $\alpha : \fA_{p,k}\lra k\cG$ that is defined over $k$. We let $\pt(\cG)$ be the set of $p$-points and note that the equivalence
classes of $p$-points define the closed points of $\Pi(\cG)$, cf.\ \cite[(4.7)]{FPe2}.

\bigskip

\begin{Prop} \label{EIFK2} Let $G$ be a reduced group scheme such that $\Pi(G_r) \neq \emptyset$, $M$ be a rational $G$-module. Then the following statements hold:
\begin{enumerate}
\item $\fK(M|_{G_r}) \subseteq M$ is a $G$-submodule of $M$.
\item If $M|_{G_r} \in \EIP(G_r)$, then $\ell^j_\alpha(M)$ is a $G$-submodule of $M$ for every $p$-point $\alpha \in \pt(G_r)$ and $j \in \{1,\ldots,p\!-\!1\}$.
\item If $M$ is simple, then either $\fK(M|_{G_r})=(0)$ or $M|_{G_r} \in \EIP(G_r)$ and $\Jt(M|_{G_r})= (\dim_kM)[1]$.\end{enumerate}\end{Prop}

\begin{proof} The group $G(k)$ acts on $kG_r$ via the adjoint representation and on the set $\Pt(G_r)$ of $\pi$-points of $G_r$. As $G$ is reduced, the group $G(k)$ is dense in $G$ and
\cite[(I.2.12(5)]{Ja2} implies that $N \subseteq M$ is a $G$-submodule if and only if $N$ is $G(k)$-invariant.

(1) Let $g \in G(k)$ and suppose that $N \subseteq M|_{G_r}$ belongs to $\EIP(G_r)$. If $\alpha_K \in \Pt(G_r)$ is a $\pi$-point, then
\begin{eqnarray*}
\ell_{\alpha_K}^i(g_K.N_K) & = & g_K.(g^{-1}_K.\ell_{\alpha_K}^i(g_K.N_K)) = g_K.\ell_{g^{-1}_K.\alpha_K}^i(g^{-1}_K.(g_K.N_K)) = g_K.\ell_{g^{-1}_K.\alpha_K}^i(N_K)\\
                         & = & g_K.\ell_{\alpha_K}^i(N_K)
\end{eqnarray*}
for $1 \le i \le p\!-\!1$. It follows that the $G_r$-submodule $g.N \subseteq M|_{G_r}$ also enjoys the equal images property. Application to $N:= \fK(M|_{G_r})$ implies that $\fK(M|_{G_r})$ is $G(k)$-invariant,
so that (1) follows.

(2) Let $\alpha \in \pt(G_r)$ be a $p$-point. Given a $G$-submodule $N \subseteq M$ and $g \in G(k)$,  we obtain
\[ (\ast) \ \ \ \ \ \ \ g.\ell_\alpha(N) = \ell_{g.\alpha}(N).\]
We now proceed by induction on $j$. If $j=1$, then, setting $N=M$ in ($\ast$), we obtain
\[ g.\ell_\alpha(M) = \ell_{g.\alpha}(M) = \ell_\alpha(M).\]
Since $G(k)$ is dense in $G$, it follows that $\ell_\alpha(M)$ is a $G$-submodule of $M$.

For $j>1$, we assume that $N:= \ell_\alpha^{j-1}(M)$ is a $G$-submodule of $M$. Now ($\ast$) in conjunction with the equal images property yields
\[ g.\ell_{\alpha}^j(M) = g.\ell_\alpha(N) = \ell_{g.\alpha}(N)= \ell_{g.\alpha}(\ell_\alpha^{j-1}(M)) = \ell_{g.\alpha}^j(M) = \ell_\alpha^j(M).\]
As before, this implies that $\ell^j_\alpha(M)$ is a $G$-submodule of $M$.

(3) If $M$ is simple, then (1) implies $\fK(M|_{G_r}) = (0)$ or $\fK(M|_{G_r})=M$. In the latter case, we have $M|_{G_r} \in \EIP(G_r)$. Let $\alpha \in \pt(G_r)$ be a $p$-point. As $\ell_\alpha$ is
nilpotent, (2) implies that $\ell_\alpha(M) \subsetneq M$ is a proper $G$-submodule of $M$, so that $\ell_\alpha(M) = (0)$ by simplicity of $M$. Since $M|_{G_r}$ has constant Jordan type and $\Pi(G_r)
\neq \emptyset$, this readily implies that $\Jt(M|_{G_r})=(\dim_kM)[1]$. \end{proof}

\bigskip
\noindent
For $M \in \modd \cG$ we denote by $\add(M)$ the full subcategory of $\modd \cG$, whose objects are direct sums of direct summands of $M$. If $M$ is semi-simple, then $\add(M)$ is a semi-simple category.

\bigskip

\begin{Cor} \label{EIFK3} Let $G$ be a connected reduced algebraic group scheme such that $G_r/(G_r)_\pi$ is diagonalizable. Suppose that $M$ is a rational $G$-module such that $M|_{G_r} \in \EIP(G_r)$.
\begin{enumerate}
\item Given a $p$-point $\alpha \in \pt(G_r)$, the $G_r$-module $M/\ell_\alpha(M)$ belongs to $\add(\bigoplus_{\lambda \in X(G_r)}k_\lambda)$.
\item The $G_r$-module $\Top(M|_{G_r})$ belongs to $\add(\bigoplus_{\lambda \in X(G_r)}k_\lambda)$.
\item If $M|_{G_r}$ is indecomposable, then there exists $\lambda \in X(G_r)$ such that $M|_{G_r}\!\otimes_kk_\lambda$ belongs to the principal block $\cB_0(G_r)$ of $kG_r$.\end{enumerate}\end{Cor}

\begin{proof} (1) Let $\alpha \in \pt(G_r)$ be a $p$-point. Thanks to Proposition \ref{EIFK2} and Lemma \ref{EIPr2}, $N:= \ell_\alpha(M)$ is a $G$-submodule of $M$ such that $X:=(M/N)|_{G_r}$ belongs to
$\EIP(G_r)$. In particular, $X$ has constant Jordan type. As $\ell_\alpha(X)=(0)$, we have $\Jt(X) = (\dim_kX)[1]$, and a consecutive application of Lemma \ref{SFS5} and Lemma \ref{EIFK1} shows that
$X \in \modd(G_r/(G_r)_\pi) \subseteq \add(\bigoplus_{\lambda \in X(G_r)}k_\lambda)$.

(2) Let $J$ be the Jacobson radical of $kG_r$. Then $J$ is $G(k)$-invariant, and $\Rad(M|_{G_r}) =JM$ is a $G$-submodule of $M$. Consequently, $N := M/JM$ is a $G$-module such that $N|_{G_r}$ is
semi-simple. Given a simple $G_r$-module $S$, we let $N(S) \subseteq N$ be the isotypic component of type $S$. Since $G$ is connected, Lemma \ref{SM2} yields $g.S \cong S^{(g)} \cong S$, showing that $N(S)$ is a
$G$-submodule of $N$ with $N(S) \in \EIP(G_r)$ and $N(S)\cong S^n$ for some $n \in \NN$.

If $\Pi(G_r)=\emptyset$, then $G_r$ is diagonalizable and our assertion follows. Alternatively, let $\alpha \in \pt(G_r)$ be a $p$-point. By (1), $\ell_\alpha(N(S))$ is a proper submodule of $N(S)$
such that $N(S)/\ell_\alpha(N(S))\in \add(\bigoplus_{\lambda \in X(G_r)}k_\lambda)$. Since $N(S)/\ell_\alpha(N(S)) \cong S^m$ for some $m\in \NN$, we have $S \cong k_\lambda$ for some $\lambda \in X(G_r)$. As a result, $N = \bigoplus_{\lambda \in X(G_r)}N(k_\lambda) \in \add(\bigoplus_{\lambda \in X(G_r)}k_\lambda)$.

(3) Let $\eta$ be the antipode of $kG_r$. By (2), there exists a character $\lambda \in X(G_r)$ such that $k_\lambda$ is a composition factor of $M|_{G_r}$. Thus, $k \cong k_\lambda\!\otimes_k\!k_{\lambda\circ\eta}$ is a composition factor of the indecomposable module $M|_{G_r}\!\otimes_k\!k_{\lambda\circ \eta}$, so that this module belongs to $\cB_0(G_r)$. \end{proof}

\bigskip
\noindent
Let $M$ be a $\cG$-module. If $S$ is a simple $\cG$-module, then $[M\!:\!S]$ denotes the Jordan-H\"older multiplicity of $S$ in $M$.

\bigskip

\begin{Thm} \label{EIFK4} Let $\cG$ be a finite group scheme of characteristic $p\ge 3$ such that
\begin{enumerate}
\item[(a)] $\SL(2)_1 \subseteq \cG$, and
\item[(b)] $\cG_\pi$ is a normal subgroup of $\cG$.
\end{enumerate}
Then $\EIP(\cG)=\modd(\cG/\cG_\pi)$. \end{Thm}

\begin{proof} We first show that $\EIP(\SL(2)_1) = \add(k)$.

Let $M \in \EIP(\SL(2)_1)$ be an indecomposable equal images module. Then $M$ is either projective or has full support, and the classification of indecomposable $\SL(2)_1$-modules \cite[Thm.]{Pr} shows in particular that $M$ is the restriction of a rational $\SL(2)$-module. Since the group scheme $\SL(2)_1$ is simple, Theorem \ref{SFS6} yields $(\SL(2)_1)_\pi = \SL(2)_1$, so that Corollary \ref{EIFK3}(3) implies that $M$ belongs to the principal block $\cB_0(\SL(2)_1)$ along with $\Top(M) \cong k^n$. Thus, if $M$ is projective, then $M\cong P(0)$ is the projective cover of the trivial
module. Consequently, the baby Verma module $Z(0):=k\SL(2)_1\!\otimes_{kB_1}\!k$, which is induced from the trivial module of the first Frobenius kernel of a Borel subgroup $B \subseteq \SL(2)$, is an image of $P(0)\in \EIP(\SL(2)_1)$, and Lemma \ref{EIPr2} forces $Z(0) \in \EIP(\SL(2)_1)$. Since $Z(0)$ is neither projective nor of full support, we have
reached a contradiction. In virtue of \cite[Theorem 2]{Po}, we thus conclude that the module $M$ has Loewy length $\ell\ell(M) \le 2$.

If $\ell\ell(M)=2$, then \cite[Theorem 3]{Po} in conjunction with $\cB_0(\SL(2)_1)$ being Morita equivalent to the trivial extension of the path algebra of the Kronecker quiver (cf.\ \cite{Dr2,Fi,Ru}) implies $\Soc(M)\cong L(p\!-\!2)^m$ with $m\ge 1$ and $|m\!-\!n|=1$. Hence there exists a submodule $N \subseteq \Soc(M)$ such that the equal images module $M/N$ has composition factors $k$ and $L(p\!-\!2)$ with multiplicities $[M/N\!:\!k]=n$ and $[M/N\!:\!L(p\!-\!2)]=1$ and $L(p\!-\!2) \subseteq \Soc(M/N)$. Since the simple $\SL(2)_1$-module $L(p\!-\!2)$ with highest weight $p\!-\!2$ does not have the equal images property, it follows that $M/N$ has an indecomposable constituent $M_0 \in \EIP(\SL(2)_1)$ of Loewy length $2$, and with socle $\Soc(M_0)\cong L(p\!-\!2)$ and top $\Top(M_0)\cong k^2$. As before, $M_0$ is the restriction of a rational $\SL(2)$-module.

Let $e,f,h$ be the canonical basis of the Lie algebra $\fsl(2)$. Since $e^{[p]}=0=f^{[p]}$, there exist $p$-points $\alpha_e,\, \alpha_f: \fA_p \lra U_0(\fsl(2))$ such that $\alpha_x(t)=x$ for $x=e,f$. In view of Proposition \ref{EIFK2}, the subspaces $\ell^j_e(M_0)$ and $\ell^j_f(M_0)$ are $\SL(2)$-submodules with $\ell_e(M_0),\ell_f(M_0) \subseteq \Soc(M_0)\cong L(p\!-\!2)$. If $\ell_e(M_0)=(0)$, then $\ell_f(M_0)=\ell_e(M_0)=(0)$ and $\fsl(2)$ acts trivially on $M_0$, a contradiction. Alternatively, we have
$\ell_e(M_0)=\Soc(M_0)=\ell_f(M_0)$. Since $\ell_e|_{\Soc(M_0)}$ and $\ell_f|_{\Soc(M_0)}$ are nilpotent endomorphisms of the vector space $\Soc(M_0)$, it follows that
\[ \ell_e(\Soc(M_0))=\ell^2_e(M_0)=(0) =\ell^2_f(M_0)= \ell_f(\Soc(M_0)).\]
This contradicts the fact that $\Soc(M_0)\cong L(p\!-\!2)$. Consequently,  $\ell\ell(M)=1$, so that $M \cong k \in \add(k)$.

Let $M \in \EIP(\SL(2)_1)$. Being an image of $M$, each indecomposable constituent $N\!\mid\!M$ of $M$ has the equal images property. By the above, $N\in \add(k)$, so that $M \in \add(k)$,
as desired.

Now let $M \in \EIP(\cG)$. In view of (a),
\[ M|_{\SL(2)_1} \in \add(k)\]
so that $M \in \CJT(\cG)$ has Jordan type $\Jt(M)=(\dim_kM)[1]$. Now (b) in conjunction with Lemma \ref{EIFK1} yields $M \in \modd(\cG/\cG_\pi)$.

The inclusion $\modd(\cG/\cG_\pi) \subseteq \EIP(\cG)$ follows directly from Lemma \ref{EIFK1}.\end{proof}

\bigskip
\noindent
We turn to restricted simple Lie algebras. According to the
general classification, these are either classical or of Cartan type. We refer the reader to \cite[Chap.4]{SF} for a discussion of non-classical simple Lie algebras.

\bigskip

\begin{Cor} \label{EIFK5} Let $(\fg,[p])$ be a restricted simple Lie algebra of Cartan type. If $p \ge 3$, then $\EIP(\fg)=\add(k)$. \end{Cor}

\begin{proof} By assumption, the Lie algebra $\fg$ contains a copy of $\fsl(2)$ as a $p$-subalgebra, see \cite{SF}.

Let $M \in \EIP(\fg)$. Arguing as in the proof of Theorem \ref{EIFK4}, we show that $\Jt(M) = (\dim_kM)[1]$. Thanks to Lemma \ref{EIFK1}(1), we have $\fg_\pi \subseteq Z_\fg(M):=\{x \in \fg \ ; \ x_M=0\}$, and the simplicity of $\fg$ yields $Z_\fg(M) = \fg$. Consequently, $M \in \add(k)$. \end{proof}

\bigskip
\noindent
When combined with Theorem \ref{MCR2}, our next result completes the proof of Theorem A. It also shows in particular, that classical simple Lie algebras possess no non-trivial modules having the equal images property.

\bigskip

\begin{Cor} \label{EIFK6} Let $G$ be a smooth reductive group. If $\Char(k)=p\ge 3$, then $\EIP(G_r) = \add(\bigoplus_{\lambda \in X(G_r)}k_\lambda)$. \end{Cor}

\begin{proof} In view of Theorem \ref{SFS6}, the infinitesimal group scheme $G_r$ fulfills condition (b) of Theorem \ref{EIFK4}, with $G_r/(G_r)_\pi$ being diagonalizable. If $G$ is a torus, then $\modd(G_r) = \add(\bigoplus_{\lambda \in X(G_r)}k_\lambda)$. Alternatively, \cite[(8.2.4),(9.3.5)]{Sp} shows that $G$ contains a copy of $\SL(2)$ or $\PSL(2)$, so that $\SL(2)_1 \subseteq G_1 \subseteq G_r$. As a result, condition (a) of Theorem \ref{EIFK4} also holds, and $\EIP(G_r) = \modd (G_r/(G_r)_\pi) \subseteq \add(\bigoplus_{\lambda \in X(G_r)}k_\lambda)$. The reverse inclusion follows from Corollary \ref{EIPr4}. \end{proof}

\bigskip

\begin{Example} Let $G=\SL(2)$ and suppose that $p\ge 3$. Then every simple $G_1$-module has constant Jordan type \cite[(2.5)]{CFP}, whereas for $r\ge 2$ the trivial module and the Steinberg module are the only simple $G_r$-modules of constant Jordan type. For $r=2$, this was observed in \cite[(4.12)]{FPe3}. If $r>2$, then \cite[(3.15)]{Ja2} and \cite[(3.16)]{Ja2} imply that every simple $G_r$-module $S$ is of the form
\[ S\cong M\!\otimes_k\!N,\]
where $M|_{G_2}$ is simple and $G_2$ acts trivially on $N$. Thus, if $S$ has constant Jordan type, then
\[S|_{G_2} \cong (M|_{G_2})^{\dim_kN},\]
has constant Jordan type and the case $r=2$ yields $\Jt(S)=(\dim_kS)[1]$, or $\Jt(S)=\frac{\dim_kS}{p}[p]$. In the latter case, $S$ is projective and thus isomorphic to the Steinberg module.
In the former case, $S\in \EIP(G_r)$, so that Corollary \ref{EIFK6} yields the assertion. \end{Example}

\bigskip

\subsection{Heller shifts of equal images modules}
Suppose that $\cG$ is a finite group scheme of abelian unipotent rank $\aurk(\cG)\ge 2$. If $M \in \EIP(\cG)\!\smallsetminus\!\{(0)\}$ is an equal images module of constant Jordan type $\Jt(M)=\bigoplus_{i=1}^da_i[i]$ for some $d\le p\!-\!2$ and $n \in \ZZ$ is odd, then $\Omega^n_\cG(M)$ has constant Jordan type $\Jt(\Omega^n_\cG(M))=\bigoplus_{i=1}^da_i[p\!-\!i]\oplus m[p]$. Consequently, Proposition \ref{EIPr6} shows that $\Omega^n_\cG(M) \not \in \EIP(\cG)$. In this section, we shall investigate the analogous property for even $n\in \ZZ$.

In what follows, $\HH^\bullet(\cG,k):= \bigoplus_{n\ge 0} \HH^{2n}(\cG,k)$ denotes the even cohomology ring of $\cG$. Thanks to the Friedlander-Suslin Theorem \cite[(1.1)]{FS}, $\HH^\bullet(\cG,k)$ is a finitely generated commutative $k$-algebra. By the same token, the canonical algebra homomorphism
\[ \Phi_M : \HH^\bullet(\cG,k) \lra \Ext^\ast_\cG(M,M) \ \ ; \ \ [f] \mapsto [f\!\otimes\!\id_M]\]
is finite. We let $I_M \unlhd \HH^\bullet(\cG,k)$ be the kernel of $\Phi_M$. Then
\[\cV_\cG(M) := Z(I_M) \subseteq {\rm Maxspec}(\HH^\bullet(\cG,k))\]
is the {\it cohomological support variety} of $M$. By general theory \cite[(5.3.5),(5.4.6)]{Be2}, the variety $\cV_\cG(M)$ has dimension $\dim \cV_\cG(M) = \cx_\cG(M)$.

\bigskip

\begin{Thm} \label{EIHS1} Let $\cG$ be a finite group scheme of abelian unipotent rank $\aurk(\cG) \ge 2$. Then the following statements hold:
\begin{enumerate}
\item If $M \in \EIP(\cG)_{p-1}\!\smallsetminus\!\{(0)\}$, then $\Omega^{2n}_\cG(M) \not \in \EIP(\cG)_{p-1}$ for every $n \in \ZZ$ with $|n| \ge (\dim_kM)^2$.
\item If $M \in \EIP(\cG)_{p-2}\!\smallsetminus\!\{(0)\}$, then $\Omega^{2n}_\cG(M) \not \in \EIP(\cG)$ for every $n \in \ZZ$ with $|n| \ge (\dim_kM)^2$. \end{enumerate}\end{Thm}

\begin{proof} (1) Our assumption implies that
\[\Jt(M) = \bigoplus_{i=1}^da_i[i]\]
for some $d \in \{1,\ldots,p\!-\!1\}$ and $a_i \in \NN$.

We first assume that $\cG$ is an abelian unipotent group scheme. Then \cite[(14.4)]{Wa} provides an isomorphism
\[k\cG \cong k[X_1,\ldots, X_r]/(X_1^{p^{n_1}}, \ldots, X_r^{p^{n_r}}),\]
so that the algebra $\HH^\bullet(\cG,k)$ is generated in degree $2$, cf.\ \cite[(3.2),(3.5)]{Be1}.

Since $M \in \EIP(\cG)$ is non-projective and of constant Jordan type, we have $\cV_\cG(M)=\cV_\cG(k)$ and \cite[(5.3.5),(5.4.6)]{Be2} yields
\[ \dim \HH^\bullet(\cG,k)/I_{M} = \dim \cV_\cG(M) = \dim \cV_\cG(k) = \cx_\cG(k) \ge \aurk(\cG) \ge 2.\]
We put $A_M := \HH^\bullet(\cG,k)/I_M$, and note that the graded algebra $A_M=\bigoplus_{n\ge 0}(A_{M})_{2n}$ is generated in degree $2$. The Noether Normalization Lemma \cite[(VI.3.1)]{Ku} thus provides algebraically independent elements $x,y \in (A_{M})_2$. Consequently,
\[ (\ast) \ \ \ \ \  n\!+\!1 = \dim_kk[x,y]_{2n} \le \dim_k (A_M)_{2n} \le \dim_k \Ext^{2n}_\cG(M,M) \ \ \ \ \forall \ n \ge 0.\]
Suppose that $\Omega^{2n}_\cG(M) \in \EIP(\cG)_{p-1}$ for some $n\in \ZZ\!\smallsetminus\!\{0\}$. By general theory, the stable Jordan types of $M$ and $\Omega^{2n}_\cG(M)$ coincide, so that
\[ \Jt(\Omega^{2n}_\cG(M)) = \bigoplus_{i=1}^d a_i[i]\oplus m[p],\]
while $\Omega^{2n}_\cG(M)\in \EIP(\cG)_{p-1}$ forces $m=0$. As a result, $\dim_k\Omega^{2n}_\cG(M) = \dim_k M$.

Suppose that $n>0$. Then we have
\[\dim_k \Ext^{2n}_\cG(M,M) \le \dim_k \Hom_\cG(\Omega^{2n}_\cG(M),M) \le \dim_k \Hom_k(\Omega^{2n}_\cG(M),M) \le (\dim_kM)^2.\]
In view of ($\ast$), this implies $n \le (\dim_kM)^2\!-\!1$.

If $n<0$, we arrive at a similar estimate:
\[ \dim_k \Ext^{-2n}_\cG(M,M) \le \dim_k \Hom_k(M,\Omega^{2n}_\cG(M)) \le (\dim_kM)^2.\]
Thus, ($\ast$) implies $|n| = -n \le (\dim_kM)^2\!-\!1$.

In the general case, we let $\cU \subseteq \cG$ be an abelian, unipotent subgroup scheme of abelian unipotent rank $\aurk(\cU) \ge 2$.  Then $M|_\cU \in \EIP(\cU)_{p-1}$ is an equal images module. If $\Omega^{2n}_\cG(M) \in \EIP(\cG)_{p-1}$, then Lemma \ref{EIPr5} implies $\Omega^{2n}_\cU(M|_\cU) \cong \Omega^{2n}_\cG(M)|_\cU \in \EIP(\cU)_{p-1}$, and the first part of the proof yields $|n| \le (\dim_kM)^2\!-\!1$, as desired.

(2) Suppose that $M \in \EIP(\cG)_{p-2}\smallsetminus \{(0)\}$. Writing $\Jt(M) = \bigoplus_{i=1}^{d}a_i[i]$ for some $d\le p\!-\!2$, we obtain $\Jt(\Omega^{2n}_\cG(M))=\bigoplus_{i=1}^da_i[i]\oplus m[p]$. If $\Omega^{2n}_\cG(M) \in \EIP(\cG)$, then Proposition \ref{EIPr6} yields $m=0$, whence $\Omega^{2n}_\cG(M) \in \EIP(\cG)_{p-1}$. Our assertion thus follows from (1). \end{proof}

\bigskip

\begin{Remarks} (1) Suppose that $p=2$ and $G=\ZZ/(2)\!\times\!\ZZ/(2)$. Given $n \in \NN$, the K\"unneth Formula implies $\dim_k\Ext^n_G(k,k)=n\!+\!1$. By the same token
\[ \cdots \lra kG^m \lra \cdots \lra kG^2 \lra kG \lra k \lra (0)\]
is a minimal projective resolution of the $G$-module $k$. Consequently, $\Omega^n_G(k)$ is an indecomposable $kG$-module of Loewy length $2$, whose top and socle have dimensions $n\!+\!1$
and $n$, respectively. As a result, $\Omega^n_G(k)\cong W_{n+1,2}$ has the equal images property, see Lemma \ref{EIPr1}.

(2) We shall show in Section $4.5$ that the conclusions of Theorem \ref{EIHS1} may fail for $M \in \EIP(\cG)\!\smallsetminus\!\EIP(\cG)_{p-2}$  and arbitrary $p\ge 3$.\end{Remarks}

\bigskip
\noindent
Suppose that $p \ge 3$. By the above result, the non-trivial even Heller shifts of one-dimensional $\cG$-modules do not possess the equal images property. Our next result provides a somewhat
stronger statement.

\bigskip

\begin{Prop} \label{EIHS2} Suppose that $p\ge 5$ and let $\cG$ be a finite group scheme such that $\aurk(\cG) \ge 2$. If $M \in \EIP(\cG)_2\!\smallsetminus\!\{(0)\}$, then $\Omega_\cG^{2n}(M) \not \in \EIP(\cG)$ for every $n \in \ZZ\!\smallsetminus\!\{0\}$. \end{Prop}

\begin{proof} We first assume that $k\cG \cong k(\ZZ/(p)\!\times\!\ZZ/(p))$. Since $M \in \EIP(\cG)$ has constant Jordan type $\Jt(M)=a_1[1]\oplus a_2[2]$, an application of \cite[(1.7),(1.9)]{CFS} shows that $\Rad^2(M)=(0)$. Consequently, \cite[(4.1)]{CFS} provides $n_1,\ldots,n_r \in \NN$ such that
\[ M \cong \bigoplus_{i=1}^rW_{n_i,2},\]
and \cite[(IV.3.6)]{ARS} furnishes an isomorphism
\[ \Omega^{2n}_\cG(M) \cong \bigoplus_{i=1}^r \Omega^{2n}_\cG(W_{n_i,2}).\]
Suppose that $\Omega^{2n}_\cG(M) \in \EIP(\cG)$. In view of Lemma \ref{EIPr2} and \cite[(2.3)]{CFS}, the $\cG$-module $\Omega^{2n}_\cG(W_{n_i,2})$ belongs to $\EIP(\cG)$ and has constant Jordan type $\Jt(\Omega^{2n}_\cG(W_{n_i,2})) = [1]\oplus (n_i\!-\!1)[2]\oplus m[p]$. Since $p\ge 5$, \cite[(4.2)]{CFS} implies $m=0$. As $\Omega^{2n}_\cG(W_{n_i,2})$ is indecomposable (cf.\ \cite{He}), the above arguments in conjunction with \cite[(4.1)]{CFS} imply $\Omega^{2n}_\cG(W_{n_i,2}) \cong W_{n_i,2}$. Since $\cx_\cG(W_{n_i,2})=2$, it follows that $n=0$.

In the general case, we observe that $\cG$ possesses an abelian unipotent subgroup $\cU \subseteq \cG$ such that $\aurk(\cU) \ge 2$. If $\Omega^{2n}_\cG(M) \in \EIP(\cG)$, then Lemma \ref{EIPr5} shows that $N:= M|_\cU$ and $\Omega^{2n}_\cU(N) \cong \Omega^{2n}_\cG(M)|_\cU$ belong to $\EIP(\cU)$, with $N$ having constant Jordan type $\Jt(N)=a_1[1]\oplus a_2[2]$.

As before, we note that $k\cU \cong U_0(\fu)$ is isomorphic to the restricted enveloping algebra of an abelian, $p$-unipotent restricted Lie algebra $\fu$ that contains a two-dimensional $p$-subalgebra $\fu_2$ with trivial $p$-map. By the above arguments, it suffices to prove the assertion for $N':= N|_{\fu_2}$. Since $U_0(\fu_2)\cong k(\ZZ/(p)\!\times\!
\ZZ/(p))$, our result now follows from the first part of the proof. \end{proof}

\bigskip
\noindent
In preparation for our study of equal images modules via Auslander-Reiten theory we record properties of $\cG$-modules $M \in \EIP(\cG)$, whose Heller shifts $\Omega^2_\cG(M)$
have the equal images property. Following \cite{CFS}, we say that a $\cG$-module $M$ has the {\it equal kernels property}, provided there exists for every $j \in \{1,\ldots,p\!-\!1\}$ a $k$-subspace $V_j \subseteq M$ such that
\[ \ker \ell_{\alpha_K}^j = V_j\!\otimes_k\!K\]
for every $\pi$-point $\alpha_K \in \Pt(\cG)$. If $\cG$ is abelian, then it suffices to test this property for $j=1$.

\bigskip

\begin{Lem} \label{EIHS3} Let $\cG$ be a finite group scheme of abelian unipotent rank $\aurk(\cG)\ge 2$. If $M \in \EIP(\cG)$ has the equal kernels property, then $\Jt(M)=(\dim_kM)[1]$.
\end{Lem}

\begin{proof} We assume $M\ne (0)$, so that $M$ has constant Jordan type
\[\Jt(M) = \bigoplus_{i=1}^d a_i[i],\]
with $1 \le d \le p$ and $a_d \ne 0$. By assumption, $\cG$ contains an abelian unipotent subgroup $\cU \subseteq \cG$ with $\aurk(\cU)\ge 2$. Since $M|_\cU \in \EIP(\cU)$ has the equal kernels property, we may assume that $\cU = \cG$. Thus, $k\cU \cong U_0(\fu)$ is the restricted enveloping algebra of
an abelian $p$-unipotent restricted Lie algebra $\fu$, so that it suffices to consider this case. Since $\dim\cV_\fu \ge 2$, there exists a two-dimensional $p$-subalgebra $\fu_2 \subseteq \cV_u$. By the above arguments, we may thus assume that $\fu = \fu_2$. Since $U_0(\fu_2) \cong k(\ZZ/(p)\!\times\!\ZZ/(p))$, this amounts to addressing the case where $\cG=\ZZ/(p)\!\times\!\ZZ/(p)=:G$.

Suppose that $d \ge 2$. Being a submodule of an equal kernels module, it follows that $N:=\Rad^{d-2}(M) \in \EIP(G)_2$ is an equal kernels module, cf.\ \cite[(1.9)]{CFS}. According to
\cite[(4.1)]{CFS}, there exists a decomposition
\[ N \cong \bigoplus_{i=1}^\ell W_{n_i,2},\]
with all constituents having the equal kernels property. Since $d \ge 2$, we can find $i \in \{1,\ldots,\ell\}$ such that $n_i\ge 2$. Writing $kG=k[x,y]$ with $x^p=0=y^p$, we observe $\ker\ell_x \ne \ker\ell_y \subseteq W_{n_i,2}$, a contradiction. Thus, $d=1$, whence $\Jt(M)=(\dim_kM)[1]$. \end{proof}

\bigskip

\begin{Thm} \label{EIHS4} Let $\cG$ be a finite group scheme with $\aurk(\cG) \ge 2$, $M \in \EIP(\cG)\!\smallsetminus\{(0)\}$ be an equal images module.
\begin{enumerate}
\item If $\Jt(\Omega^2_\cG(M)) = \Jt(M)$, then $\Omega^2_\cG(M) \not \in \EIP(\cG)$.
\item If $M \in \EIP(\cG)_{p-2}$, then $\Omega^j_\cG(M) \not \in \EIP(\cG)$ for $j \in \{-2,2\}$. \end{enumerate}\end{Thm}

\begin{proof} (1) We put $N:= \Omega^2_\cG(M)$, write $\Jt(N)=\bigoplus_{i=1}^da_i[i]$ with $a_i \in \NN$ and assume that $N \in \EIP(\cG)$. Let $\cH \subseteq \cG$ be a closed subgroup scheme with $\aurk(\cH)\ge 2$. According to Lemma \ref{EIPr5},
\[ \Omega_\cH^2(M|_\cH) \cong \Omega^2_\cG(M)|_\cH \in \EIP(\cH).\]
We may therefore derive a contradiction by considering the case $\cG =\ZZ/(p)\!\times\!\ZZ/(p)=:G$.

If $\alpha$ is a $p$-point for $G$, then \cite[(1.7)]{CFS} yields
\[ \Rad^j(M) = \ell_\alpha^j(M) \ \ \text{and} \ \ \Rad^j(N) = \ell_\alpha^j(N) \ \ \ \ \text{for} \ j \ge 0.\]
Since $\Jt(M)=\Jt(N)$, we have $\dim_k\ell_\alpha^j(M)=\dim_k\ell_\alpha^j(N)$, so that $\dim_k \Rad^j(M)=\dim_k\Rad^j(N)$. We thus obtain $\dim_k\Top(M)=\dim_k\Top(N)$.

We consider a minimal projective presentation
\[ (0) \lra N \lra P_1 \lra P_0 \lra M \lra (0)\]
of $M$. Thus, $P_0$ is a projective cover of $M$ and $P_1$ is an injective hull of $N$. By the above, we have $\dim_kM=\dim_kN$, whence $\dim_kP_0=\dim_kP_1$. Since $kG$ is local, this
implies $P_0 \cong P_1$, so that there exists an exact sequence
\[ (0) \lra N \lra P \lra P \lra M \lra (0),\]
with $N \lra P$ and $P\lra M$ being injective hulls and projective covers, respectively. In view of $\dim_k\Top(M)=\dim_k\Top(N)$, this yields
\[ (\ast) \ \ \ \ \dim_k\Soc(N) = \dim_k\Soc(P) = \dim_k\Top(P) = \dim_k\Top(M) = \dim_k\Top(N).\]
Let $\alpha$ be a $p$-point. Then $\Soc(N) \subseteq \ker \ell_\alpha$. On the other hand, ($\ast$) in conjunction with $N$ having the equal images property implies
\[ \dim_k\Soc(N) = \dim_k\Top(N) = \dim_k N/\ell_\alpha(N) = \sum_{i=1}^da_i = \dim_k\ker \ell_\alpha,\]
so that $\ker \ell_\alpha = \Soc(N)$. Since $k$ is algebraically closed, \cite[(7.8)]{CFS} ensures that the $G$-module $N$ has the equal kernels property. Lemma \ref{EIHS3} now yields $\Jt(M) = \Jt(N) = (\dim_kN)[1] = (\dim_kM)[1]$. Thus, $\Rad(M)=(0)=\Rad(N)$, so that $G$ acts trivially on $M$ and $N$. As $M \ne (0)$, it follows that $G$ acts trivially on the indecomposable direct summand $\Omega^2_G(k)$ of $N$. Consequently, $\Omega^2_G(k) \cong k$, so that $\cx_G(k)=1$, a contradiction.

(2) By assumption, there exist $1 \le d \le p\!-\!2$ and $a_i \in \NN$ such that
\[ \Jt(M)=\bigoplus_{i=1}^da_i[i].\]
As the stable Jordan types of $M$ and $\Omega^2_\cG(M)$ coincide, the assumption $\Omega^2_\cG(M) \in \EIP(\cG)$ in conjunction with Proposition \ref{EIPr6} implies $\Jt(\Omega^2_\cG(M))=\Jt(M)$, so that (1) yields a contradiction. If $V:=\Omega^{-2}_\cG(M) \in \EIP(\cG)$, then Proposition \ref{EIPr6} yields $V \in \EIP(\cG)_{p-2}$. Thanks to Lemma \ref{EIPr5}, the $\cG$-module $M$ is projective-free, so that $\Omega^2_\cG(V) \cong M \in \EIP(\cG)$. However, we just showed that this cannot happen. \end{proof}

\bigskip

\subsection{Heller shifts of $W$-modules}
In this section we turn to the example $G=\ZZ/(p)\!\times\!\ZZ/(p)$ and consider Heller shifts of the $W$-modules $W_{n,d}$ that were introduced in \cite{CFS}, see Section $4.1$. This is done by employing gradations, which amounts to working in the category $\modd_\ZZ G$ of $\ZZ$-graded modules and degree zero homomorphisms. We dispense with elaborating on the formalism and only use those features that are needed in our context.

Note that
$kG = k[X,Y]/(X^p,Y^p)$ inherits the canonical $\ZZ$-grading from the polynomial ring $k[X,Y]$, with generators $x:= X\!+\!(X^p,Y^p)$ and $y:= Y\!+\!(X^p,Y^p)$ being homogeneous of degree $1$.
In particular,
\[ kG = \bigoplus_{i=0}^{2p-2}kG_i \ \ \ \text{and} \ \ \ \Rad(kG) = \bigoplus_{i=1}^{2p-2}kG_i.\]
If $M=\bigoplus_{i\in \ZZ}M_i$ is a $\ZZ$-graded $G$-module and $j \in \ZZ$, then $M[j]$ is the graded $G$-module with underlying $G$-space $M$ and $\ZZ$-grading defined via
\[ M[j]_i := M_{i-j} \ \ \ \ \ \ \ \text{for all} \ i \in \ZZ.\]
Note that each $W_{n,d}$ is $\ZZ$-graded
\[ W_{n,d} = \bigoplus_{i=0}^{d-1}(W_{n,d})_i.\]
We first determine a minimal graded presentation of $W_{n,d}$.

\bigskip

\begin{Lem} \label{HSW1} Let $d \in \{2,\ldots,p\}$ and $n\ge d$. The following statements hold:
\begin{enumerate}
\item Suppose that $d \le p\!-\!1$. Then there exists an exact sequence
\[ (0) \lra \Omega_G^2(W_{n,d}) \lra kG^{n+1}[1]\oplus kG^{n-d}[d] \stackrel{\partial_1}{\lra} kG^n \stackrel{\partial_0}\lra W_{n,d} \lra (0),\]
with $\deg(\partial_j)=0$ for $j \in \{0,1\}$. In particular, $\Omega^2_G(W_{n,d})$ is a $\ZZ$-graded $G$-module.
\item There is an exact sequence
\[ (0) \lra \Omega^2_G(W_{n,p}) \lra kG^{n+1}[1] \stackrel{\partial_1}{\lra} kG^n \stackrel{\partial_0}{\lra} W_{n,p} \lra (0)\]
with $\deg(\partial_j)=0$ for $j \in \{0,1\}$.\end{enumerate}\end{Lem}

\begin{proof} Let $d \in \{2,\ldots,p\}$ and recall that $N_{n,d}$ is the kernel of the canonical surjection $kG^n \tha W_{n,d}$. We begin
by proving two statements concerning the defining relations of $W_{n,d}$.

\medskip
(a) \ {\it The $G$-module $N_{n,d}$ is generated by the $k$-subspace $S_1\oplus T_d$, where}
\[ S_1 := \langle \{x.v_1,y.v_n\} \cup \{y.v_i\!-\!x.v_{i+1}  \ ; \ 1 \le i \le n\!-\!1\}\rangle_k \ \ ; \ \ T_d := \langle \{x^d.v_i \ ; \ d\!+\!1 \le i \le n\} \rangle_k.\]

\smallskip
\noindent
Let $N'_{n,d} := \langle\{x.v_1,y.v_n\} \cup \{y.v_i\!-\!x.v_{i+1} \ ; \ 1\le i \le d\!-\!1\}\cup\{x^d.v_j \ ; \ d\!+\!1\le j \le n\}\rangle_{kG} \subseteq N_{n,d}$ and suppose that $x^j.v_j \in N'_{n,d}$ for some $j \in \{1,\ldots,d\!-\!1\}$. Then we have
\[ x^{j+1}.v_{j+1} = x^j.(x.v_{j+1}-y.v_j)\!+\!x^j.y.v_j \in N'_{n,d},\]
so that induction implies $x^j.v_j\in N'_{n,d}$ for all $j \in \{1,\ldots,d\}$. Consequently, $N_{n,d}=N'_{n,d}$ is generated by $S_1\oplus T_d$. \hfill $\diamond$

\medskip

(b) \ $T_d\cap kG_{d-1}.S_1 = (0)$.

\smallskip
\noindent
Observing $T_p=(0)$, we assume that $d\le p\!-\!1$. Let $v=\sum_{j=d+1}^n\alpha_jx^dv_j$ be an element of $kG_{d-1}.S_1$. There results an identity
\[ v = f_0x.v_1\! +\! \sum_{i=1}^{n-1}f_i(y.v_i\!-\!x.v_{i+1})\!+\!f_ny.v_n,\]
with $f_i \in kG_{d-1}$. Upon comparing coefficients of the $v_i$, we arrive at
\[ 0=f_0x\!+f_1y \ \ ; \ \ 0=f_iy\!-\!f_{i-1}x \ \ 2 \le i \le d \ \ ; \ \ \alpha_jx^d =  f_jy\!-\!f_{j-1}x  \ \ \ d\!+\!1 \le j \le n.\]
Since $d\le p\!-\!1$ and $f_i \in kG_{d-1}$, while $(X^p,Y^p) \subseteq \bigoplus_{i\ge p}k[X,Y]_i$, the first $d$-equations may be interpreted to hold in $k[X,Y]$. We thus obtain
\[ \deg_X(f_i) = \deg_X(f_iY) = \deg_X(f_{i-1}X) \ge \deg_X(f_{i-1})\!+\!1\]
for $1\le i \le d$, whence
\[ \deg_X(f_i) \ge i \ \ \ \ \ \ \ \ \forall \ i \in \{0,\ldots,d\}.\]
(Here we set $\deg_X0=\infty$.) Since $f_d \in k[X,Y]_{d-1}$, we obtain $f_d=0$, implying $f_j=0$ for $0\le j \le d$.

Returning to the original system, we assume that $f_0=f_1=\cdots =f_{r-1}=0$ for some $d<r\le n$. Then we have
\[ \alpha_rx^d = f_ry,\]
so that $0=f_ry^p=\alpha_rx^dy^{p-1}$. Thus, $\alpha_r=0$ and $f_ry=0$. Consequently, $f_r \in kGy^{p-1}\cap kG_{d-1} = (0)$. Thus, all $f_j$ vanish and $v=0$. \hfill $\diamond$

\medskip
(1) Let $d\le p\!-\!1$. Recall that $J:=\bigoplus_{i\ge 1}kG_i$ is the radical of $kG$. According to (a), we have $N_{n,d}=kG.(S_1\oplus T_d)$. Thus, (b) implies
\begin{eqnarray*}
(S_1\oplus T_d)\cap J.N_{n,d} & \subseteq & (S_1\oplus T_d)\cap (J.S_1+J.T_d) = (S_1\oplus T_d)\cap J.S_1 = T_d\cap J.S_1 \\
                        & = & T_d\cap kG_{d-1}.S_1=(0).
\end{eqnarray*}
Let $\pi: N_{n,d} \lra \Top(N_{n,d})$ be the canonical projection. The identity above implies that $\pi$ induces an injection $S_1\oplus T_d \hookrightarrow \Top(N_{n,d})$. Moreover, property (a) shows that the $kG/J$-module $\Top(N_{n,d})$ is generated by $\pi(S_1\oplus T_d)$. Since $kG/J\cong k$, we obtain $\pi(S_1\oplus T_d) = \Top(N_{n,d})$.

We let $\partial_0 : \bigoplus_{i=1}^nkGv_i \lra W_{n,d}$ be the canonical projection, so that $\partial_0$ is homogeneous of degree $0$. Since $\dim_k\Top(W_{n,d})=n$, the pair $(kG^n,\partial_0)$ is a projective cover of $W_{n,d}$. Owing to (a), $N_{n,d} \subseteq kG^n$ is a homogeneous submodule, whose generators belong to $S_1\cup T_d$ and are linearly independent. By the above, we have $\dim_k\Top(N_{n,d})=2n\!-\!d\!+\!1$. Writing $kG^{n+1}=\bigoplus_{i=1}^{n+1}kGw_i$ and $kG^{n-d}=\bigoplus_{i=1}^{n-d}kGu_i$, we define the map
$\partial_1 : kG^{n+1}[1]\oplus kG^{n-d}[d] \lra kG^n$ via
\[ \partial_1(w_1) := x.v_1 \ ; \ \partial_1(w_j) = y.v_{j-1}\!-\!x.v_j \ \ (2 \le j \le n) \ \ ; \ \ \partial_1(w_{n+1})=y.v_n\]
as well as 
\[ \partial_1(u_i)=x^d.v_{i+d} \ \ \ \ \ 1\le i \le n\!-\!d.\]
Hence $\partial_1$ is homogeneous of degree $\deg(\partial_1)=0$. Since $\partial_1$ induces an isomorphism $\Top(kG^{2n\!-\!d\!+\!1})\cong \Top(N_{n,d})$, the pair $(kG^{n+1},\partial_1)$ is a projective cover of $N_{n,d}$. The resulting exact sequence
\[ (0) \lra \Omega^2_G(W_{n,d}) \lra kG^{n+1}[1]\oplus kG^{n-d}[d] \stackrel{\partial_1}{\lra} kG^n \stackrel{\partial_0}{\lra} W_{n,d} \lra (0)\]
enjoys the requisite properties.

(2) Let $d=p$. Then $T_d = (0)$ and the arguments of (1) yield the desired result. \end{proof}

\bigskip

\begin{Cor} \label{HSW2} Let $n\ge d$.
\begin{enumerate}
\item $\Jt(\Omega^2_G(W_{n,d}))=\bigoplus_{i=1}^{d-1}[i]\oplus (n\!-\!d\!+\!1)[d]\oplus (n\!-\!d\!+\!1)p[p]$ for $d\le p\!-\!1$.
\item $\Jt(\Omega^2_G(W_{n,p}))=\bigoplus_{i=1}^{p-1}[i]\oplus (n\!+\!1)[p]$. \end{enumerate} \end{Cor}

\begin{proof} We put $M := \Omega^2_G(W_{n,d})$. By general theory, the stable Jordan types of $M$ and $W_{n,d}$ coincide. We thus formally write
\[ \Jt(M) = \Jt(W_{n,d})\oplus \ell [p],\]
where $\ell \in \ZZ$. In particular, $\dim_kM = \ell p \!+\!\dim_kW_{n,d}$.

(1) Let $d\le p\!-\!1$. Counting dimensions, we obtain from Lemma \ref{HSW1}(1)
\[ \ell p \!+\!\dim_kW_{n,d} = (2n\!-\!d\!+\!1)p^2\!-\!np^2\!+\dim_kW_{n,d},\]
so that $\ell = (n\!-\!d\!+\!1)p$. The assertion thus follows from $\Jt(W_{n,d}) = \bigoplus_{i=1}^{d-1}[i]\oplus (n\!-\!d\!+\!1)[d]$.

(2) Using Lemma \ref{HSW1}(2), we arrive at
\[ \ell p \!+\!\dim_kW_{n,p} = (n\!+\!1)p^2\!-\!np^2\!+\dim_kW_{n,p},\]
so that $\ell = p$. Observing $\Jt(W_{n,p}) = \bigoplus_{i=1}^{p-1}[i]\oplus (n\!-\!p\!+\!1)[p]$, we obtain the assertion. \end{proof}

\bigskip
\noindent
Given a $G$-module $M$ of constant Jordan type, we will write
\[ \Jt(M) = \bigoplus_{i=1}^pa_i(M)[i].\]
In order to identify some of the modules $\Omega^2_G(W_{n,d})$, we require the following auxiliary result:

\bigskip

\begin{Lem} \label{HSW3} Let $M \in \EIP(G)$ be an equal images module.
\begin{enumerate}
\item The $G$-module $M/\Rad^2(M)$ is indecomposable if and only if $a_1(M)=1$.
\item If $M \in \EIP(G)_d$ and $a_1(M)=1$, then there exists a surjective homomorphism $W_{n,d} \tha M$, where $n = \dim_k\Top(M)$. \end{enumerate}\end{Lem}

\begin{proof} Since $M$ is an equal images module, \cite[(1.7)]{CFS} implies $\Rad^2(M) = \ell^2_\alpha(M)$ for every $p$-point $\alpha \in \pt(G)$. Consequently, $M/\Rad^2(M)$ is an equal images module of constant Jordan type
\[\Jt(M/\Rad^2(M)) = (\sum_{i=2}^pa_i(M))[2]\oplus a_1(M)[1],\]
whence $a_2(M/\Rad^2(M))=\sum_{i=2}^pa_i(M)$ and $a_1(M/\Rad^2(M))=a_1(M)$.

(1) Suppose that $a_1(M)=1$. If $X$ and $Y$ are submodules of $M/\Rad^2(M)$ such that $M/\Rad^2(M)$ $=X\oplus Y$, then $X$ and $Y$ are equal images modules with $\Jt(M/\Rad^2(M))=\Jt(X)\oplus\Jt(Y)$. Thus $a_1(X)=0$ or $a_1(Y)= 0$, so that \cite[(4.2)]{CFS} forces $X=(0)$ or $Y=(0)$. As a result, $M/\Rad^2(M)$ is indecomposable.

Conversely, if $M/\Rad^2(M)$ is indecomposable, then \cite[(4.1)]{CFS} yields $M/\Rad^2(M) \cong W_{n,2}$ for some $n \in \NN$, so that $a_1(M)=a_1(M/\Rad^2(M))=a_1(W_{n,2})=1$.

(2) In view of (1), the $G$-module $M/\Rad^2(M)$ is indecomposable and \cite[(4.6)]{CFS} provides a surjection $\lambda: W_{n,p} \lra M$. Since $M \in \EIP(G)_d$, we have
$\lambda(\Rad^d(W_{n,p})) = \Rad^d(M)=(0)$. According to \cite[(2.4)]{CFS}, there is an isomorphism $W_{n,p}/\Rad^d(W_{n,p}) \cong W_{n,d}$, so that $\lambda$ induces the desired map.\end{proof}

\bigskip
\noindent
The following result shows that Theorems \ref{EIHS1}(2) and \ref{EIHS4}(2) may fail for equal images modules not belonging to $\EIP(G)_{p\!-\!2}$.

\bigskip

\begin{Lem}\label{HSW4} The following statements hold:
\begin{enumerate}
\item $\Omega^2_G(W_{n,p})\cong W_{n+p,p}$ for every $n\ge p$.
\item $\Omega^2_G(W_{p-1,p-1})\cong W_{2p-1,p}$.
\item $\Omega^2_G(W_{n,p-1})\not \in \EIP(G)$ for $n\ge p$. \end{enumerate}\end{Lem}

\begin{proof} (1) Thanks to Lemma \ref{HSW1}(2) there is an exact sequence
\[ (\ast) \ \ \ \ \ \ (0) \lra \Omega^2_G(W_{n,p}) \lra kG^{n+1}[1] \stackrel{\partial_1}{\lra} kG^n \stackrel{\partial_0}{\lra} W_{n,p} \lra (0)\]
of graded $kG$-modules with $\deg(\partial_j)=0$ for $j \in \{0,1\}$.

\medskip
($\ast\ast$) {\it We have $\dim_k\Top(\Omega^2_G(W_{n,p})) \ge n\!+\!p$.}

\smallskip
\noindent
To ease notation, we put $M:=\Omega^2_G(W_{n,p})$. Let $j\ge 0$. Thanks to ($\ast$), there exists an exact sequence
\[ (0) \lra M_j \lra (kG)^{n+1}_{j-1} \lra (kG)^n_j \lra (W_{n,p})_j \lra  (0).\]
If $j\le p\!-\!1$, we obtain
\[\dim_kM_j =(n\!+\!1)j\!-\!n(j\!+\!1) \!+\!n\!-\!j =0.\]
Thus, $M_j=(0)$, so that $M = \bigoplus_{j\ge p}M_j$.

Since $(W_{n,p})_p=(0)$, we have an exact sequence
\[ (0) \lra M_p \lra (kG^{n+1})_{p-1} \lra (kG^n)_p \lra (0).\]
Consequently,
\[ \dim_k M_p = \dim_k(kG^{n+1})_{p-1}\!-\!\dim_k(kG^n)_p = (n\!+\!1)p\!-\!n(p\!-\!1)=n\!+\!p.\]
Observing $\Rad(M) =\Rad(kG)M = \sum_{i\ge 1}kG_iM\subseteq \sum_{i\ge p+1}M_i$, we conclude
\[ \dim_k\Top(M) = \dim_kM\!-\!\dim_k \Rad(M) \ge \dim_kM_p = n\!+\!p,\]
as asserted. \hfill $\diamond$

\medskip
\noindent
Let $\alpha \in \pt(G)$ be a $p$-point of $G$. Then we have $\ell_\alpha(M)\subseteq \Rad(M)$, while Corollary \ref{HSW2}(2) and ($\ast\ast$) imply
\[n\!+\!p = \dim_k \coker \ell_\alpha \ge \dim_k \Top(M)\ge n\!+\!p.\]
Consequently, $\Rad(M) = \ell_\alpha(M)$, so that \cite[(1.7)]{CFS} yields $M \in \EIP(G)$. Since $a_1(M)=1$, Lemma \ref{HSW3} now provides a surjection $W_{n+p,p} \tha M$, which, by equality of dimensions, is an isomorphism.

(2) Lemma \ref{HSW1}(1) provides an exact sequence
\[ (0) \lra M \lra kG^p[1] \stackrel{\partial_1}{\lra} kG^{p-1} \stackrel{\partial_0}{\lra} W_{p-1,p-1} \lra (0),\]
of homomorphisms of degree $0$, with $M:=\Omega^2_G(W_{p-1,p-1})$. The arguments of (1) show that $M=\bigoplus_{j\ge p}M_j$ as well as $\dim_kM_p=2p\!-\!1$. Consequently, $\dim_k\Top(\Omega^2_G(W_{p-1,p-1}))\ge 2p\!-\!1$, and the proof may now be completed by adopting the arguments of (1) verbatim.

(3) Let $n\ge p$ and $M:=\Omega^2_G(W_{n,p-1})$. Recall that $M=\bigoplus_{i\in \ZZ}M_i$ is a $\ZZ$-graded $G$-module. We put $\supp(M) :=\{i\in \ZZ \ ; \ M_i \ne (0)\}$ as well as $[a,b] := \{x \in \ZZ \ ; \ a \le x \le b\}$ for $a,b \in \ZZ$. Lemma \ref{HSW1}(1) implies
\[ \supp(M) = [p,3p\!-\!3].\]
According to Lemma \ref{HSW2}, we have $a_p:=a_p(M) = (n\!-\!p\!+\!2)p$ as well as $a_1(M)=1$. Suppose that $M \in \EIP(G)$. Then Lemma \ref{HSW2} yields $\dim_k\Top(M)=n\!+\!a_p$ and Lemma \ref{HSW3} provides a surjection $f : W_{n+a_p,p} \lra M$. Consequently, the restriction $g : \Rad^{p-1}(W_{n+a_p,p}) \lra \Rad^{p-1}(M)$ is also surjective. Direct computation shows that $\dim_k\ker f = n\!-\!p\!+\!1=\dim_k \ker g$. Thus, there exists an exact sequence
\[ (0) \lra V \lra W_{n+a_p,p} \lra M \lra (0)\]
with $V \subsetneq \Rad^{p-1}(W_{n+a_p,p})$. Recall that $W_{n+a_p,p}$ is graded and that $V \subsetneq \Rad^{p-1}(W_{n+a_p,p})\subseteq (W_{n+a_p,p})_{p-1}$ is a homogeneous submodule. This implies that the indecomposable $G$-module $W_{n+a_p,p}/V$ is $\ZZ$-graded with $\supp(W_{n+a_p,p}/V)=[0,p\!-\!1]$. Consequently, \cite[(4.1)]{GG} provides an isomorphism
\[ (W_{n+a_p,p}/V)[j] \cong M\]
of $\ZZ$-graded modules for some $j \in \ZZ$. As a result, $\supp(M)=[j,j\!+\!p\!-\!1]$, a contradiction.\end{proof}

\bigskip

\begin{Remark} Let $\Omega_{\gr G}$ be the Heller operator of the Frobenius category $\modd_\ZZ G$. The proof of Lemma \ref{HSW4} yields
\[ \Omega^2_{\gr G}(W_{n,p}) \cong W_{n+p,p}[p] \ \ \text{and} \ \ \Omega^2_{\gr G}(W_{p-1,p-1}) \cong W_{2p-1,p}[p].\]
Consequently, each of these modules has a graded projective resolution, $(P_n)_{n\ge 0}$ such that each $P_n$ is generated in degree $\delta(n)$, where
\[ \delta(n) := \left\{ \begin{array}{cl} \frac{n}{2}p & \text{if} \ n \in 2\NN_0\\ \frac{n\!-\!1}{2}p\!+\!1 &\text{otherwise.}\end{array}\right.\]
Modules with this property are called {\it $p$-Koszul modules}, cf.\ \cite{GMMZ04}. Note that $kG$ is a $p$-Koszul algebra if and only if $p=2$. \end{Remark}

\bigskip

\section{Auslander-Reiten Components}
In this section we investigate the category of equal images modules over finite group schemes of abelian unipotent rank $\ge 2$ by means of Auslander-Reiten theory. For these groups, Auslander-Reiten components of tree class $A_\infty$ occur most often. Moreover, if such a component $\Theta$ contains a module of constant $j$-rank, then $\Theta\cong \ZZ[A_\infty]$ has the structure described in the Introduction. While AR-components containing a module of constant $j$-rank or of constant Jordan type consist entirely of such modules, the distribution of equal images modules depends on their
Loewy lengths.

\subsection{The distribution of equal images modules}

Given a finite group scheme $\cG$, we denote by $\Gamma_s(\cG):= \Gamma_s(k\cG)$ the stable Auslander-Reiten quiver of the self-injective algebra $k\cG$. By general theory, the Auslander-Reiten translation $\tau_\cG : \Gamma_s(\cG) \lra \Gamma_s(\cG)$ is
given by
\[ \tau_\cG(M) \cong \Omega^2_\cG(M)\!\otimes_k\!k_\zeta,\]
where $\zeta : k\cG \lra k$ is the {\it modular function} of the cocommutative Hopf algebra $k\cG$, see \cite[(1.5)]{FMS}.

According to the Riedtmann Structure Theorem \cite[Struktursatz]{Rie}, each component $\Theta \subseteq \Gamma_s(\cG)$ is of the form $\Theta \cong \ZZ[T_\Theta]/\Pi$, where $T_\Theta$ is a directed
tree and $\Pi \subseteq \Aut_k(\ZZ[T_\Theta])$ is an admissible subgroup of the automorphism group of the stable translation quiver $\ZZ[T_\Theta]$. Moreover, the underlying undirected tree $\bar{T}_\Theta$ is uniquely determined by $\Theta$, and is customarily referred to as the {\it tree class} of $\Theta$.

A non-projective indecomposable $\cG$-module $M$ is called {\it quasi-simple}, if
\begin{enumerate}
\item[(a)] the component $\Theta \subseteq \Gamma_s(\cG)$ containing $M$ has tree class $A_\infty$, and
\item[(b)] the module $M$ has exactly one predecessor in $\Theta$. \end{enumerate}
By way of illustration, we take another look at elementary abelian $p$-groups of rank $2$.

\bigskip

\begin{Examples} (1) Consider $kG = k(\ZZ/(p)\!\times\!\ZZ/(p))$ and recall that there is a canonical functor $F : \modd k(\bullet \rightrightarrows \bullet) \lra \modd G$. In view of Lemma \ref{EIPr1}, a pre-injective $k(\bullet \rightrightarrows \bullet)$-module $M$ gives rise to an indecomposable equal images module $F(M)$ of constant Jordan type $\Jt(F(M))
=[1]\oplus n[2]$.

For $p=2$, the remark following (\ref{EIHS1}) shows that the modules $F(M)$ are just the Heller shifts $\Omega_G^n(k)$, with $n \ge 0$. By \cite[(II.7.3)]{Er1}, these modules belong to a stable Auslander-Reiten component of type $\ZZ[\tilde{A}_{1,2}]$.

Alternatively, the local algebra $kG$ is wild and Erdmann's Theorem \cite[Thm.1]{Er2} ensures that the indecomposable $\cG$-module $F(M)$ belongs to a stable Auslander-Reiten component $\Theta \cong \ZZ[A_\infty]$. Thanks to \cite[(3.1.2)]{Fa5}, each $F(M)$ is quasi-simple. Thus, if $F(M)\cong W_{n,2}$ and $F(M') \cong W_{n',2}$ belong to $\Theta$, then there exists $m \in \ZZ$ such that $W_{n',2} \cong \tau_\cG^m(W_{n,2}) \cong \Omega^{2m}_G(W_{n,2})$. As $p\ge 3$, this implies $n'=n$, so that $\Omega^{2m}_G(W_{n,2})\cong W_{n,2}$. Observing that equal images modules have full support, we obtain $\cx_G(W_{n,2})=2$. This forces $m=0$ and we conclude that the various $F(M)$ all belong to different AR-components of $kG$.

(2) Suppose that $p\ge 3$. Given $1 \le d \le p$ and $n\ge d$, \cite[(2.3)]{CFS} yields
\[ \Jt(W_{n,d}) = \bigoplus_{i=1}^{d-1}[i]\oplus (n\!-\!d\!+\!1)[d].\]
Hence $W_{n,d}$ is indecomposable for $d\ge 2$ and by \cite[(3.1.2)]{Fa5}, each $W_{n,d}$ is quasi-simple.

(3) The same arguments yield that each member $M/U_f$ of the algebraic family of indecomposable $G$-modules constructed in the proof of Theorem \ref{RTEI6} is quasi-simple. \end{Examples}

\bigskip

\begin{Lem} \label{ARD1} Let $\cG$ be a finite group scheme of abelian unipotent rank $\aurk(\cG)\ge 2$. If $\Theta \subseteq \Gamma_s(\cG)$ is a component such that $\Theta\cap \EIP(\cG) \ne \emptyset$, then either $\Theta$ has Euclidean tree class, $\Theta \cong \ZZ[\tilde{A}_{p,q}]$, or $\Theta \cong \ZZ[A_\infty]$, $\ZZ[A_\infty^\infty]$, $\ZZ[D_\infty]$. \end{Lem}

\begin{proof} According to \cite[(3.2)]{Fa4}, the tree class $T_\Theta$ of the component $\Theta$ is either a finite or infinite Dynkin diagram, or a Euclidean diagram. Since $\Theta\cap \EIP(\cG) \ne \emptyset$, the component $\Theta$ contains a $\cG$-module $M$ of full support $P(\cG)_M = P(\cG)$, so that $\dim P(\cG)_\Theta \ge \aurk(\cG)\!-\!1 \ge 1$. In view of \cite[(3.3)]{Fa4}, either the tree
class $\bar{T}_\Theta$ is Euclidean, or $\Theta \cong \ZZ[ \tilde{A}_{p,q}], \ZZ[A_\infty]$, $\ZZ[A_\infty^\infty]$, $\ZZ[D_\infty]$. \end{proof}

\bigskip

\begin{Lem} \label{ARD2} Let $\cG$ be a finite group scheme of abelian unipotent rank $\aurk(\cG)\ge 2$. If $\Theta \subseteq \Gamma_s(\cG)$ is a component of tree class
$\bar{T}_\Theta \neq A_\infty$, then $\Theta\cap \EIP(\cG)_{p-1}$ is finite. \end{Lem}

\begin{proof} Since $\aurk(\cG) \ge 2$, Lemma \ref{ARD1} implies that the component $\Theta$ has tree class $\bar{T}_\Theta \in \{\tilde{A}_{1,2},\tilde{D}_n$, $\tilde{E}_n, A_\infty, A_\infty^\infty,
D_\infty\}$. In view of $\bar{T}_\Theta \ne A_\infty$, \cite[(VII.3)]{ARS} shows that every $\tau_\cG$-invariant subadditive function $f : \Theta \lra \NN_0$ is bounded.

Let $\alpha_K \in \Pt(\cG)$ be a $\pi$-point. For $M \in \Theta$, we write
\[\Jt(M) = \bigoplus_{i=1}^p\alpha_{K,i}(M)[i].\]
If $\Theta\cap \EIP(\cG)\ne \emptyset$, then $\dim \Pi(\cG)_\Theta = \dim \Pi(\cG) \ge \aurk(\cG)\!-\!1 \ge 1$, and a consecutive application of \cite[(2.3)]{Fa5} and \cite[(2.4)]{Fa5} shows that, for $i \in \{1,\ldots, p\!-\!1\}$, the function $\alpha_{K,i} : \Theta \lra \NN_0$ is $\tau_\cG$-invariant and additive. Hence there exists $\ell \in \NN$ such that $\alpha_{K,i}(M) \le \ell$ for all $M \in \Theta$ and $i \in \{1,\ldots,p\!-\!1\}$.

Let $\alpha_K \in \Pt(\cG)$ be a $\pi$-point. Given $M \in \Theta\cap\EIP(\cG)_{p-1}$, we have 
\[\Jt(M) = \bigoplus_{i=1}^{p-1}\alpha_{K,i}(M)[i],\]
so that
\[ \dim_kM = \sum_{i=1}^{p-1}i\alpha_{K,i}(M) \le \frac{(p\!-\!1)p}{2}\ell.\]
Thanks to \cite[(3.2)]{Fa1}, which also holds for finite group schemes, this implies that $\Theta\cap \EIP(\cG)_{p-1}$ is finite. \end{proof}

\bigskip
\noindent
Let $\cG$ be a finite group scheme. A short exact sequence
\[ (0) \lra N \lra E \lra M \lra (0)\]
of $\cG$-modules is referred to as {\it locally split}, provided the exact sequence
\[ (0) \lra \alpha_K^\ast(N_K) \lra \alpha^\ast_K(E_K) \lra \alpha_K^\ast(M_K) \lra (0)\]
is split exact for every $\pi$-point $\alpha_K \in \Pt(\cG)$.

\bigskip

\begin{Lem} \label{ARD3} Let $\cG$ be a finite group scheme. If
\[ (0) \lra N \lra E \lra M \lra (0)\]
is a locally split short exact sequence such that $E \in \EIP(\cG)_d$, then $M,N \in \EIP(\cG)_d$. \end{Lem}

\begin{proof} Thanks to Lemma \ref{EIPr2}, the $\cG$-module $M$ belongs to $\EIP(\cG)_d$.

Let $\alpha_K \in \Pt(\cG)$ be a $\pi$-point. We claim that
\[ (\ast) \ \ \ \ \ \ \ell_{\alpha_K}^j(N_K) = \ell_{\alpha_K}^j(E_K)\cap N_K \ \ \ \ \text{for} \ 1\le j \le p\!-\!1.\]
Since the given sequence is locally split, there exists an $\fA_{p,K}$-submodule $X \subseteq \alpha_K^\ast(E)$ such that
\[ \alpha_K^\ast(E_K)=\alpha_K^\ast(N_K)\oplus X.\]
Let $a \in \fA_{p,K}$ be an element. If $v \in a.\alpha_K^\ast(E_K)\cap \alpha_K^\ast(N_K)$, then there exits $n \in \alpha^\ast_K(N_K)$ and $x \in X$ such that $v=a.n+a.x$. Thus,
$v-a.n=a.x \in \alpha^\ast_K(N_K)\cap X = (0)$, so that $v=a.n \in a.\alpha^\ast_K(N_K)$. Thus, $a.\alpha_k^\ast(N_K) = a.\alpha_K^\ast(E_K)\cap \alpha_K^\ast(N_K)$ and assertion ($\ast$)
follows by specializing $a = t^j$.

Let $j \in \{1,\ldots,p\!-\!1\}$. Since $E \in \EIP(\cG)_d$, there exists a subspace $V \subseteq E$, with $V=(0)$ for $j \ge d$, such that
\[ \ell_{\alpha_K}^j(E_K) = V_K \ \ \ \ \ \ \forall \ \alpha_K \in\Pt(\cG).\]
Let $\alpha_K$ be a $\pi$-point. Thanks to ($\ast$), we obtain
\[ \ell_{\alpha_K}^j(N_K) = \ell_{\alpha_K}^j(E_K)\cap N_K = V_K\cap N_K = (V\cap N)_K.\]
As a result, $N \in \EIP(\cG)_d$. \end{proof}

\bigskip
\noindent
Let $\Theta \subseteq \Gamma_s(\cG)$ be a component. Then $\Theta$ is referred to as {\it regular}, provided $\Rad(P) \not \in \Theta$ for every principal indecomposable $\cG$-module $P$. In case $\Theta\cap \EIP(\cG) \neq \emptyset$, we define
\[ \delta_\Theta := \min \{ (\dim_kN)^2 \ ; \ N \in \Theta\cap\EIP(\cG)\}.\]
Let $\Theta \cong \ZZ[A_\infty]$ and suppose that $M \in \Theta$ is quasi-simple. Then there exists a unique infinite sectional path
\[ \cdots \rightarrow (r)M \rightarrow (r\!-\!1)M \rightarrow \cdots \rightarrow (1)M=M\]
in $\Theta$. If $X \in \Theta$ is an arbitrary module, then there exist uniquely determined $\ell \in \NN$ and $n \in \ZZ$ such that $X\cong \tau_\cG^n((\ell)M)$. In that case,
we say that $X$ has {\it quasi-length} $\ql(X)=\ell$. Following \cite[(2.2)]{EK}, we call the full subquiver $\cW(X) \subseteq \Theta$, whose vertices are the modules $\tau^n_\cG((m)M)$
with $1 \le m \le \ell$ and $0\le n \le \ell\!-\!m$, the {\it wing spanned by $X$}.

\bigskip

\begin{Thm} \label{ARD4} Let $\cG$ be a finite group scheme of abelian unipotent rank $\aurk(\cG) \ge 2$, $\Theta \subseteq \Gamma_s(\cG)$ be a regular component of tree
class $A_\infty$. Then the following statements hold:
\begin{enumerate}
\item If $M \in \Theta\cap\EIP(\cG)$, then $\cW(M) \subseteq \Theta \cap \EIP(\cG)$.
\item Every $M \in \Theta\cap\EIP(\cG)_{p-1}$ is quasi-simple, and $\Theta\cap \EIP(\cG)_{p-1}$ is finite.
\item If $\Theta \cap \EIP(\cG)_{p-2} \ne \emptyset$, then $|\Theta \cap \EIP(\cG)| \le \delta_\Theta$ and every $M \in \Theta\cap\EIP(\cG)$ is quasi-simple.
 \end{enumerate}\end{Thm}

\begin{proof}  Since $\aurk(\cG)\ge 2$, Lemma \ref{ARD1} yields $\Theta \cong \ZZ[A_\infty]$.

(1) We prove our assertion by induction on the quasi-length $\ql(M)$ of $M$. If $\ql(M) = 2$, then, observing $\Theta$ being regular, we conclude that $M$ is the middle term of an
almost split sequence, whose three terms form $\cW(M)$. Lemma \ref{ARD3} thus yields $\cW(M) \subseteq \EIP(\cG)$.

Suppose that $\ql(M) \ge 3$. Since $\Theta$ is regular, there exist modules $X,Y \in \Theta$ of quasi-lengths $\ql(X) = \ql(M)\!-\!1$ and $\ql(Y)=\ql(M)\!-\!2$ such that
\[ (0) \lra \tau_\cG(X) \lra M\oplus Y \lra X \lra (0)\]
is the almost split sequence terminating in $X$. In view of
\[ \dim \Pi(\cG)_M = \dim \Pi(\cG) \ge \aurk(\cG)\!-\!1 \ge 1\]
it follows from \cite[(8.5)]{CFP} that the above sequence is locally split. Moreover, the resulting map $M \lra X$ is surjective, so that Lemma \ref{EIPr2} ensures that $X \in \EIP(\cG)$. Since $Y \in \cW(X)$, the
inductive hypothesis yields $Y \in \EIP(\cG)$. Thus, $M\oplus Y \in \EIP(\cG)$, and Lemma \ref{ARD3} implies $\tau_\cG(X) \in \EIP(\cG)$. Since $\ql(\tau_\cG(X)) = \ql(X) = \ql(M)\!-\!1$ and
\[ \cW(M) = \{M\} \cup \cW(\tau_\cG(X))\cup \cW(X),\]
the inductive hypothesis implies $\cW(M) \subseteq \EIP(\cG)$.

(2) Let $M$ be an element of $\Theta\cap\EIP(\cG)_{p-1}$. Suppose that $M$ has quasi-length $\ql(M)\ge 2$. Since $\Theta$ is regular, there exists a sequence $M \tha M_{m-1} \tha \cdots \tha M_1$ of irreducible epimorphisms such that $\ql(M_i)=i$. According to Lemma \ref{EIPr2}, each $M_i$ belongs to $\EIP(\cG)_{p-1}$. There results an almost split sequence
\[ (\ast) \ \ \ \ \ (0) \lra \tau_\cG(M_1) \lra M_2 \lra M_1 \lra (0).\]
As before, the almost split sequence ($\ast$) is locally split. A consecutive application of Lemma \ref{ARD3} and Corollary \ref{EIPr4} now implies $\Omega^2_\cG(M_1) \in \EIP(\cG)_{p-1}$. Since the stable Jordan types of $M_1$ and $\Omega_\cG(M_1)$ coincide, we obtain $\Jt(M_1) = \Jt(\Omega^2_\cG(M_1))$. Theorem \ref{EIHS4}(1) now implies $\Omega_\cG^2(M_1) \not \in \EIP(\cG)$, a contradiction. As a result, the $\cG$-module $M$ is quasi-simple.

Fix a $\cG$-module $M_0 \in \Theta\cap\EIP(\cG)_{p-1}$ and let $M\in \Theta\cap\EIP(\cG)_{p-1}$. Since $M$ and $M_0$ are quasi-simple, there exists $n \in \ZZ$ such that $M=\tau_\cG^n(M_0)$. A consecutive application of Corollary \ref{EIPr4}(2) and Theorem \ref{EIHS1}(1) yields $|n| \le (\dim_kM_0)^2$. This shows that $\Theta\cap\EIP(\cG)_{p-1}$ is finite. 

(3) Let $M \in \Theta\cap\EIP(\cG)$. Since $\Theta \cap \EIP(\cG)_{p-2}\ne \emptyset$, it follows from \cite[(3.2.3)]{Fa5} that $\Jt(M)=\bigoplus_{i=1}^{p-2}a_i[i]\oplus a_p[p]$. Proposition \ref{EIPr6} thus forces $M \in \EIP(\cG)_{p-2}$, so that (2) implies that $M$ is quasi-simple.

Let $M_0 \in \Theta\cap\EIP(\cG)$ be a $\cG$-module of dimension $\sqrt{\delta_\Theta}$. Using Theorem \ref{EIHS1}(2) we obtain
\[ \Theta \cap \EIP(\cG) \subseteq \{\tau_\cG^n(M) \ ; \ |n| \le \delta_\Theta \!-\!1\}.\]
Our assertion now follows from Theorem \ref{EIHS4}(2).\end{proof}

\bigskip

\begin{Remark} In Section 5.2 we will provide examples of components $\Theta$ with $\Theta\cap\EIP(\cG)_{p-1}\ne \emptyset$ such that $\Theta\cap\EIP(\cG)$ is infinite and contains modules of
arbitrarily large quasi-length. \end{Remark}

\bigskip
\noindent
Our first application concerns Auslander-Reiten components containing one-dimensional modules.

\bigskip

\begin{Prop} \label{ARD5} Suppose that $p\ge 3$, and let $\cG$ be a finite group scheme of abelian unipotent rank $\aurk(\cG)\ge 2$ such that the principal block $\cB_0(\cG)\subseteq k\cG$ possesses no simple module $S$ of constant Jordan type $\Jt(S) = m[p\!-\!1]\oplus n[p]$. Let $\Theta \subseteq \Gamma_s(\cG)$ be a component containing a one-dimensional $\cG$-module $k_\lambda$.
Then the following statements hold:
\begin{enumerate}
\item The component $\Theta$ is regular, and $\Theta \cong \ZZ[A_\infty], \ZZ[A_\infty^\infty]$, or $\ZZ[D_\infty]$.
\item If $\Theta \cong \ZZ[A_\infty]$, then $\Theta \cap \EIP(\cG) = \{k_\lambda\}$.
\item If $\Theta \cong \ZZ[A_\infty^\infty]$ and $\cG(k)$ has odd order, then $\Theta \cap \EIP(\cG) = \{k_\lambda\}$.
\item If $\Theta \cong \ZZ[D_\infty]$, then $\Theta\cap \EIP(\cG)=\{ M \in \Theta \ ; \ \dim_kM \le 2\}$. \end{enumerate}\end{Prop}

\begin{proof} According to Corollary \ref{EIPr4}, the functor
\[ T_\lambda : \modd(\cG) \lra \modd(\cG) \ \ ; \ \ M \mapsto M\!\otimes_k\!k_\lambda\]
is an auto-equivalence which leaves $\EIP(\cG)$ invariant.  It sends the principal bock $\cB_0(\cG)$ to the block $\cB_\lambda \subseteq k\cG$ containing $k_\lambda$ and the component containing
$k$ onto the component containing $k_\lambda$. Moreover, $\alpha_K^\ast(T_\lambda(M)_K) \cong \alpha^\ast_K(M_K)$ for all $\alpha_K \in \Pt(\cG)$, and we may thus assume without loss of generality that $k_\lambda =k$ is the trivial one-dimensional $\cG$-module.

(1) Suppose that $\Theta$ is not regular. Then there exists a principal indecomposable $\cG$-module $P$ such that $\Rad(P) \in \Theta$. According to \cite[(3.1.2)]{Fa5} (see also \cite[(8.8)]{CFP}), there exist $m \in \NN$ and $n \in \NN_0$ such that $\Jt(\Rad(P))=m[1]\oplus n[p]$. Thus, the simple $\cB_0(\cG)$-module $S := \Omega^{-1}_\cG(\Rad(P))$ has constant Jordan type $\Jt(S)=m[p\!-\!1]\oplus n'[p]$. As this contradicts our current assumption on $\cB_0(\cG)$, the component $\Theta$ is regular.

Since components of Euclidean tree class and components of type $\ZZ[ \tilde{A}_{p,q}]$ are not regular (cf.\ \cite[Thm.A]{We} and \cite[p.155]{BR}), Lemma \ref{ARD1} now implies $\Theta \cong
\ZZ[A_\infty],\, \ZZ[A_\infty^\infty]$, or $\ZZ[D_\infty]$.

(2) Suppose that $\Theta \cong \ZZ[A_\infty]$. According to Theorem \ref{ARD4}, the $\cG$-module $k$ has quasi-length $\ql(k)=1$. Thus, $\Theta\cap \EIP(\cG)_{p-2} \ne \emptyset$, while $\delta_\Theta = 1$. Hence Theorem \ref{ARD4} gives $\Theta\cap\EIP(\cG)=\{k\}$.

(3) Suppose that $\Theta \cong \ZZ[A_\infty^\infty]$. Since $k$ has constant Jordan type $\Jt(k)=[1]$, it follows from \cite[(3.1.2)]{Fa5} that
\[ \Jt(M) = [1]\oplus n_M[p]\]
for every $M \in \Theta$. As $p\ge 3$, Proposition \ref{EIPr6} yields $n_M=0$ whenever $M \in \Theta\cap\EIP(\cG)$. Consequently, $\dim_kM=1$ and $M \cong k_\lambda$ for some $\lambda \in X(\cG)$.
As a result, the auto-equivalence $T_\lambda$ induces an automorphism of $\Theta$ of finite order $n_\lambda$. Owing to \cite[(4.15)]{Be1}, we obtain $T^2_\lambda = \id_\Theta$, so
that $k_{2\lambda} \cong k$. Thus, $\lambda \in X(\cG)$ has order a divisor of $2$.

The character group $X(\cG^0)$ of the infinitesimal group $\cG^0$ is a $p$-group. As $p\ge 3$, the restriction $\lambda|_{\cG^0} \in X(\cG^0)$ of $\lambda \in X(\cG)$ is trivial,
so that $\lambda \in X(\cG_{\rm red})$. Since $X(\cG_{\rm red})$ and $\cG(k)$ have the same order, we obtain $\lambda = 0$ and $M \cong k$.

(4) Since $\Jt(k)=[1]$, it follows from \cite[(3.1.1)]{Fa5} and \cite[(4.5)]{Be1} that every $M \in \Theta$ has constant Jordan type
\[ \Jt(M) = a_M[1]\oplus n_M[p]\]
with $a_M \in \{1,2\}$. The arguments of (3) thus yield $\dim_kM \le 2$ for every $M \in \Theta \cap \EIP(\cG)$. Conversely, if $M \in \Theta$ has dimension $\le 2$, then there is $a_M \in \{1,2\}$ with $\Jt(M)=a_M[1]$, so that $M$ has the equal images property.\end{proof}

\bigskip¥¥

\begin{Remarks}
\begin{enumerate}
\item If $\cx_\cG(k) \ge 3$, then \cite[(3.3)]{Fa4} implies that the component containing $k_\lambda$ is isomorphic to $\ZZ[A_\infty]$. If $G$ is a finite group, then the same conclusion already holds if $k_\lambda$ belongs to a block of wild representation type, cf.\ \cite[Thm.1]{Er2}.
\item If the group $\cG$ is supersolvable, then all simple $\cB_0(\cG)$-modules are one-dimensional \cite[(I.2.5)]{Vo}, so that the technical conditions of Proposition \ref{ARD5} hold in that case.
\item Suppose that $G$ is a finite group and let $\cB \subseteq kG$ be a block whose stable Auslander-Reiten quiver contains a component of tree class $D_\infty$. According to \cite[Thm.4]{Er2}, the defect group $D_\cB \subseteq G$ of $\cB$ is semidihedral and $p=2$. \end{enumerate}\end{Remarks}

\bigskip

\begin{Example} We consider the restricted Lie algebra $\fsl(2)$ along with its standard basis $\{e,f,h\}$. Let $\fsl(2)_s := \fsl(2)\oplus kv_0$ be the four-dimensional central
extension of $\fsl(2)$, whose bracket and $p$-map are given by
\[ [x+\alpha v_0,y+\beta v_0] := [x,y] \ \ \ \ \forall \ x,y \in \fsl(2) \ \ \text{and} \ \ e^{[p]}=0=f^{[p]} \ , \ h^{[p]} = h+v_0  \ , \ v_0^{[p]}=0,\]
respectively. According to \cite[(8.2.2)]{Fa5}, the trivial $U_0(\fsl(2)_s)$-module $k$ is the only simple module of constant Jordan type. By the same token, the AR-component $\Theta_k$
containing $k$ is of type $\ZZ[A_\infty^\infty]$. Since $\fsl(2)_s$ corresponds to an infinitesimal group $\cG$ (of height $1$), we have $\cG(k)=\{1\}$, and Proposition \ref{ARD5} implies
$\Theta_k\cap\EIP(\fsl(2)_s)=\{k\}$. \end{Example}

\bigskip
\noindent
A finite group scheme $\cG$ is called {\it trigonalizable} if all simple $\cG$-modules are one-dimensional. Our next result should be contrasted with the fact that AR-components containing a module of constant Jordan type consist entirely of such modules, see \cite[(8.7)]{CFP}.

\bigskip

\begin{Cor} \label{ARD6} Suppose that $p\ge 3$. Let $\cG$ be a trigonalizable finite group scheme of abelian unipotent rank $\aurk(\cG) \ge 2$, $\Theta \subseteq \Gamma_s(\cG)$ be a component.
\begin{enumerate}
\item If $\Theta\cap\EIP(\cG) \ne \emptyset$, then $\Theta$ is regular and $\Theta \cong \ZZ[A_\infty], \ZZ[A_\infty^\infty]$, or $\ZZ[D_\infty]$.
\item If $\Theta$ has tree class $A_\infty$, then the following statements hold:
\begin{enumerate}
\item $\cW(M) \subseteq \EIP(\cG)$ for every $M \in \Theta\cap\EIP(\cG)$.
\item The set $\Theta\cap\EIP(\cG)_{p-1}$ is finite with every $M \in \EIP(\cG)_{p-1}$ being quasi-simple.
\item If $\Theta\cap \EIP(\cG)_{p-2} \ne \emptyset$, then $|\Theta\cap\EIP(\cG)| \le \delta_\Theta$ and $\Theta\cap\EIP(\cG)$ consists of quasi-simple modules. \end{enumerate}
\item If $\Theta\cap\EIP(\cG)_{p-2}\ne \emptyset$, then $\Theta\cap \EIP(\cG)$ is finite.\end{enumerate}\end{Cor}

\begin{proof} (1) Suppose that $\Theta$ is not regular. Then there exists a principal indecomposable $\cG$-module $P$ such that $\Rad(P) \in \Theta$. Since $\cG$ is trigonalizable, there exists $\lambda \in X(\cG)$ such that $\Rad(P) \cong \Omega_\cG(k_\lambda)$. Thus, $\Rad(P)$ has constant Jordan type $\Jt(\Rad(P))=[p\!-\!1]\oplus n[p]$ for some $n \in \NN_0$. Since $\Theta \cap \EIP(\cG)\ne \emptyset$ and $\aurk(\cG)\ge 2$, Proposition \ref{EIPr6} and \cite[(3.2.3)]{Fa5} now yield a contradiction. As components of Euclidean tree class of type $\ZZ[\tilde{A}_{p,q}]$ are not regular (cf.\ \cite[Thm.A]{We} and \cite[p.155]{BR}), our assertions concerning the isomorphism type of $\Theta$ now follow from Lemma \ref{ARD1}.

(2a) This follows from (1) and Theorem \ref{ARD4}(1).

(2b) This is a direct consequence of (1) and Theorem \ref{ARD4}(2).

(2c) The assertion follows from (1) and Theorem \ref{ARD4}(3).

(3) The assertion follows from a consecutive application of (2c) and Lemma \ref{ARD2}. \end{proof}

\bigskip

\begin{Remark} If $p\ge 5$ and $M \in \Theta\cap\EIP(\cG)_2$, then, using Proposition \ref{EIHS2} instead of Theorem \ref{ARD4}, the conclusion of ($2$c) can be strengthened to $\Theta\cap\EIP(\cG)=\{M\}$. \end{Remark}

\bigskip

\begin{Cor} \label{ARD7} Suppose that $p\ge 3$ and let $\cU$ be a unipotent group scheme of abelian unipotent rank $\aurk(\cU) \ge 2$. Let $\Theta \subseteq \Gamma_s(\cU)$ be a component.
\begin{enumerate}
\item Then $\Theta\cap\EIP(\cU)_{p-1}$ is finite and every $M \in \Theta \cap \EIP(\cU)_{p-1}$ is quasi-simple.
\item If $\Theta\cap\EIP(\cU)_{p-2} \neq \emptyset$, then $|\Theta\cap\EIP(\cU)| \le \delta_\Theta$ and $\Theta\cap\EIP(\cG)$ consists of quasi-simple modules.\end{enumerate}\end{Cor}

\begin{proof} We first show that the algebra $k\cU$ is wild. If this is not the case, then $k\cU$ is tame or representation-finite. Since $k\cU$ is a local algebra, \cite[(5.1.5)]{Fa3} shows that
$k\cU$ is not tame. Hence $k\cU$ is representation-finite, so that $2\le \aurk(\cU) \le \cx_\cU(k) \le 1$, a contradiction.

As noted in \cite[(5.6)]{Fa1}, Erdmann's Theorem \cite[Thm.1]{Er2} holds for unipotent group schemes, implying that every component $\Theta$ which meets $\EIP(\cU)$ is isomorphic to
$\ZZ[A_\infty]$.

The result thus follows from Corollary \ref{ARD6}(2b),(2c). \end{proof}

\bigskip

\subsection{Components for $\ZZ/(p)\!\times\!\ZZ/(p)$}
By way of illustration, we consider the case where $G:=\ZZ/(p)\!\times\!\ZZ/(p)$ for $p\ge 3$. Recall that Erdmann's Theorem \cite[Thm.1]{Er2} implies that a component $\Theta \subseteq \Gamma_s(G)$ containing an equal images module is of type $\ZZ[A_\infty]$. If $M$ is a $G$-module of constant Jordan type, we write
\[ \Jt(M) = \bigoplus_{i=1}^pa_i(M)[i],\]
with $a_i(M) \in \NN_0$ for $1\le i \le p$. Given $d \in \{1,\ldots,p\}$ and $n \ge d$, we let $\Theta_{n,d} \subseteq \Gamma_s(G)$ be the AR-component containing the $G$-module $W_{n,d}$.

\bigskip

\begin{Prop} \label{ARW1} The following statements hold:
\begin{enumerate}
\item If $d \le p\!-\!1$, then $\Theta_{n,d}\cap\EIP(G)_{p-1} = \{W_{n,d}\}$.
\item If $d \le p\!-\!2$, then $\Theta_{n,d}\cap\EIP(G) = \{W_{n,d}\}$. \end{enumerate} \end{Prop}

\begin{proof} (1) If $N \in \Theta_{n,d}\cap\EIP(G)_{p-1}$, then Corollary \ref{ARD7} implies that the $G$-modules $N$ and $W_{n,d}$ are quasi-simple. Hence there exists $m \in \ZZ$ such that
$N \cong \Omega^{2m}_G(W_{n,d})$. Thus,
\[ \Jt(N) = \Jt(W_{n,d})\oplus a_p[p]\]
for some $a_p \in \NN_0$, while our assumption $N \in \EIP(G)_{p-1}$ forces $a_p=0$. Consequently, $\dim_k\Top(N) = \dim_k \Top(W_{n,d})=n$ and $a_1(N)=a_1(W_{n,d})=1$. Lemma \ref{HSW3} thus provides a surjection $W_{n,d} \lra N$. Since $N$ and $W_{n,d}$ have the same dimension, it follows that $N\cong W_{n,d}$.

(2) If $d\le p\!-\!2$, then a consecutive application of \cite[(3.2.3)]{Fa5} and Proposition \ref{EIPr6} gives $\Theta_{n,d}\cap\EIP(G)=\Theta_{n,d}\cap\EIP(G)_{p-1}$, so that (1) yields the assertion. \end{proof}

\bigskip
\noindent
We finally discuss the sets $\Theta_{n,d}\cap\EIP(G)$ for $d\ge p\!-\!1$ that turn out to behave differently. Recall that $kG=\bigoplus_{i=0}^{2p-2}kG_i$ inherits the canonical $\ZZ$-grading from the polynomial ring in two variables. Moreover, the $kG$-modules
\[ W_{n,d} = \bigoplus_{i=0}^{d-1}(W_{n,d})_i\]
are also $\ZZ$-graded, cf.\ \cite[(2.1)]{CFS}.

\bigskip

\begin{Lem} \label{ARW2} If $(0) \lra N \lra E \lra M \lra (0)$ is an almost split sequence such that $M,N \in \EIP(G)$, then $E \in \EIP(G)$.\end{Lem}

\begin{proof} Given a $G$-module $X$ and a $p$-point $\alpha \in \pt(G)$, we have $\ell_\alpha(X) \subseteq \Rad(X)$. There results a natural epimorphism $\lambda_X : X/\ell_\alpha(X) \lra \Top(X)$.

If $N$ is simple, then $N\cong k$ and $\Omega^{-2}_\cG(k)\cong M \in \EIP(G)$, which contradicts Theorem \ref{EIHS1}. Consequently,  $\ell_\alpha(N) = \Rad(N) \ne (0)$ for
every $\alpha \in \pt(G)$, see \cite[(1.7)]{CFS}. Since the sequence is almost split, \cite[(V.3.2)]{ARS} provides a commutative diagram
\[ \begin{CD} (0) @>>>N/\ell_\alpha(N)@>>> E/\ell_\alpha(E)@>>>M/\ell_\alpha(M) @>>> (0) \\
                         @.   @VV\lambda_NV @VV\lambda_EV @ VV\lambda_MV @.\\
 (0) @>>> \Top(N)@>>> \Top(E) @>>> \Top(M) @>>>(0) \end{CD}\]
with exact rows. In view of $M,N \in \EIP(G)$ being equal images modules, the maps $\lambda_N$ and $\lambda_M$ are isomorphisms. Hence $\lambda_E$ also has this property, so that $\ell_\alpha(E) = \Rad(E)$. Thanks to \cite[(1.7)]{CFS}, this implies $E \in \EIP(G)$.\end{proof}

\bigskip

\begin{Prop} \label{ARW3} The following statements hold:
\begin{enumerate}
\item Let $\Theta \subseteq \Gamma_s(G)$ be a component and suppose that $n\ge p$ is minimal subject to $W_{n,p} \in \Theta$. Then $n<2p$ and
\begin{enumerate} \item[(a)] $\Theta\cap\EIP(G) = \{(r)W_{n+mp,p} \ ; \ r\ge 1, \, m \ge 0\}$ for $n\ne 2p\!-\!1$.
\item[(b)] $\Theta\cap\EIP(G) = \{(r)W_{(m+2)p-1,p} \ ; \ r\ge 1, \, m \ge 0\}\cup \{(r)W_{p-1,p-1}\ ; \ r \ge 1\}$ for $n=2p\!-\!1$. \end{enumerate}
\item $|\{\Theta_{n,p} \ ; \ n \ge p\}|=p\!-\!1$.\end{enumerate} \end{Prop}

\begin{proof} (1) Recall that $\Omega^2_G$ is the Auslander-Reiten translation of $\modd G$. If $n\ge 2p$, then Lemma \ref{HSW4} implies 
\[ W_{n-p,p} \cong \Omega^{-2}_G(W_{n,p}) \in \Theta,\]
which contradicts the choice of $n$. Consequently, $p\le n <2p$.

If $\Theta$ is not regular, then $\Rad(kG)\in \Theta$. Since $\Jt(\Rad(kG)) = [p\!-\!1] \oplus (p\!-\!1)[p]$, while $\Jt(W_{n,p})=\bigoplus_{i=1}^{p-1}[i] \oplus (n\!-\!p\!+\!1)[p]$, an application
of \cite[(3.2.3)]{Fa5} yields a contradiction.

\medskip

(i) {\it We have $\{(r)W_{n+mp,p} \ ; \ r\ge 1, \, m \ge 0\} \subseteq \Theta\cap\EIP(G)$}.

\smallskip
\noindent
Setting $\Theta(r) := \{(r)W_{n+mp,p} \ ; \ m \ge 0\}$, we shall show inductively that $\Theta(r) \subseteq \Theta\cap\EIP(G)$ for all $r \ge 1$, the case $r=1$ being an immediate consequence of
Lemma \ref{HSW4}. Let $r>1$ and suppose that $\Theta(r\!-\!1)\subseteq \Theta\cap\EIP(G)$. Since $\Theta$ is regular, there exists an almost split sequence
\[ (0) \lra (r\!-\!1)W_{n+(m+1)p,p} \lra (r)W_{n+mp,p}\oplus (r\!-\!2)W_{n+(m+1)p,p} \lra (r\!-\!1)W_{n+mp,p} \lra (0)\]
for any given $m \ge 0$. (Here we set $(0)M=(0)$ for every $M \in \Theta$.) Lemma \ref{ARW2} in conjunction with the inductive hypothesis implies
\[(r)W_{n+mp,p}\oplus (r\!-\!2)W_{n+(m+1)p,p} \in \EIP(G),\]
so that Lemma \ref{EIPr2} gives $(r)W_{n+mp,p} \in \EIP(G)$. \hfill $\diamond$

\medskip

(ii) {\it If $\{(r)W_{n+mp} \ ; \ r\ge 1, \, m \ge 0\} \subsetneq \Theta\cap\EIP(G)$, then $n=2p\!-\!1$}.

\smallskip
\noindent
If equality does not hold, then $\Theta$ being regular implies that $\Theta\cap\EIP(G)\!\smallsetminus\!\{(r)W_{n+mp} \ ; \ r\ge 1, \, m \ge 0\}$ contains a chain
\[ M_\ell \tha M_{\ell-1} \tha \cdots \tha M_1\]
of epimorphisms with $\ql(M_j)=j$ for $j \in \{1,\ldots,\ell\}$. In particular, the $G$-module $M_1$ is quasi-simple and there exists $m> 0$ such that
\[ W_{n,p} \cong \Omega^{2m}_G(M_1).\]
Consequently, $\Jt(M_1)=\bigoplus_{i=1}^{p-1}[i]\oplus a_p(M_1)[p]$, so that $a_1(M_1)=1$.

If $a_p(M_1)\ne 0$, then Lemma \ref{HSW3} shows that $M_1\cong W_{a_p(M)+p-1,p}$, while Lemma \ref{HSW4} yields $n=a_p(M_1)\!+\!(m\!+\!1)p\!-\!1$. Since $p\le n<2p$, we obtain $m=0$, a contradiction. Alternatively, $\Jt(M_1)=\bigoplus_{i=1}^{p-1}[i]$, and Lemma \ref{HSW3} implies $M_1\cong W_{p-1,p-1}$. In view of Lemma \ref{HSW4}, application of $\Omega^{2m}_G$ to the short exact sequence
$(0) \lra W_{p-1,p-1} \lra W_{p,p} \lra k^p\lra (0)$ (cf.\ \cite[(2.5)]{CFS}) gives a short exact sequence
\[ (0)\lra W_{n,p} \lra W_{(m+1)p,p}\oplus kG^r \lra \Omega^{2m}_G(k)^p \lra (0)\]
for some $r\ge 0$. Observing $\dim_k\Omega^{2m}_G(k)=mp^2\!+\!1$ as well as $\dim_kW_{n,p}=(n\!-\!p\!+\!1)p+p\frac{p-1}{2}$, we arrive at
\begin{eqnarray*}
(mp\!+\!1)p\!+\!p\frac{p\!-\!1}{2}\!+\! rp^2 &=& \dim_kW_{(m+1)p,p}\!+\!\dim_kkG^r = \dim_kW_{n,p}\!+\!p\dim_k\Omega^{2m}_G(k)\\
& =&(n\!-\!p\!+\!1)p\!+\!p\frac{p\!-\!1}{2}\!+\!mp^3\!+\!p,
\end{eqnarray*}
whence
\[ (m\!+\!r)p= n\!-\!p\!+\!1\!+\!mp^2.\]
Thus, $n\equiv -1 \ \modd(p)$, so that $n=2p\!-\!1$. \hfill $\diamond$

\medskip
\noindent
Now suppose that $n=2p\!-\!1$. Lemma \ref{HSW4}(2) yields
\[ \Omega^2_G(W_{p-1,p-1}) \cong W_{2p-1,p},\]
so that $W_{p-1,p-1}\in \Theta$. Lemma \ref{ARW2} in conjunction with (i) now implies $\{(r)W_{p-1,p-1} \ ; \ r \ge 1\} \subseteq \Theta\cap \EIP(\cG)$.

If $\Theta\cap\EIP(G)\!\smallsetminus\!(\{(r)W_{(m+2)p-1,p} \ ; \ r\ge 1, \, m \ge 0\}\cup \{(r)W_{p-1,p-1}\})\ne \emptyset$, then the arguments of (ii) imply that
this set contains a quasi-simple module $M$. Hence there exists $m>0$ with
\[ \Omega^{2m}_G(M)\cong W_{p-1,p-1},\]
so that $\Jt(M)=\bigoplus_{i=1}^{p-1}[i]\oplus a_p[p]$. In view of Proposition \ref{ARW1}, we have $a_p\ne 0$. Since $a_1(M)=1$, Lemma \ref{HSW3} provides a surjection
$W_{a_p+p-1,p} \tha M$, which is an isomorphism. Thus, Lemma \ref{HSW4} gives
\[ W_{p-1,p-1} \cong \Omega^{2m}_G(W_{a_p+p-1,p}) \cong W_{a_p+(m+1)p-1,p},\]
a contradiction.

(2) By (1), we have $\{\Theta_{n,p} \ ; \ n \ge p\} = \{\Theta_{p,p}, \ldots , \Theta_{2p-1,p}\}$.\end{proof}

\bigskip

\begin{Remark} Working in the category $\modd_\ZZ G$, a much more detailed analysis shows that, for $n\ge p$, $W_{n,p-1}$ is the only equal
images module in its AR-component. Hence $\EIP(G)\cap\Theta_{n,d}$ is infinite if and only if $W_{n,d}$ is a $p$-Koszul module. \end{Remark}

\end{document}